\documentclass[11pt]{amsart}
 
\usepackage[margin=1in]{geometry} 
\usepackage{amsmath,amsthm,amssymb,tikz-cd,enumerate,mathrsfs,bbm,stmaryrd,subcaption,graphicx,setspace}
\usepackage[bookmarksopen=true]{hyperref}
\usepackage{cite}

\newcommand{\cat}[1]{\mathcal{#1}}

\newcommand{\vol}{\text{vol}}

\newcommand{\GL}{\text{GL}}
\newcommand{\Aut}{\text{Aut}}

\newcommand{\img}{\text{img}}

\newcommand{\st}{\text{ : }}

\newcommand{\Gal}{\text{Gal}}
\newcommand{\Mat}{\text{Mat}}

\newcommand{\End}{\text{End}}

\newcommand{\diag}{\text{diag}}

\newcommand{\sm}{\setminus}

\newcommand{\into}{\hookrightarrow}
\newcommand{\onto}{\twoheadrightarrow}

\newcommand{\Tr}{\text{Tr}}

\newcommand{\Nm}{\text{Nm}}

\renewcommand{\O}{\mathcal{O}}
\newcommand{\ord}{\text{ord}}
\newcommand{\p}{\mathfrak{p}}

\newcommand{\m}{\mathfrak{m}}

\newcommand{\Z}{\mathbb{Z}}
\newcommand{\Q}{\mathbb{Q}}
\newcommand{\C}{\mathbb{C}}
\newcommand{\N}{\mathbb{N}}
\newcommand{\F}{\mathbb{F}}
\newcommand{\A}{\mathbb{A}}
\newcommand{\R}{\mathbb{R}}
\newcommand{\T}{\mathbb{T}}
\newcommand{\inv}{{-1}}
\renewcommand{\subset}{\subseteq}

\newcommand{\sK}{\cat{K}}

\newcommand{\sX}{\cat{X}}

\newcommand{\Cl}{\mathnormal{Cl}}

\newcommand{\SU}{\text{SU}}

\renewcommand{\mod}{\text{ mod }}

\newtheorem{theorem}{Theorem}

\newtheorem{proposition}[theorem]{Proposition}

\newtheorem{lemma}[theorem]{Lemma}
\theoremstyle{remark}
\newtheorem{remark}[theorem]{Remark}

\newtheorem{example}{Example}

\theoremstyle{definition}
\newtheorem{definition}[theorem]{Definition}
\newtheorem*{definition*}{Definition}
\newtheorem*{claim*}{Claim}
\newtheorem*{theorem*}{Theorem}
\newtheorem*{corollary*}{Corollary}

\graphicspath{{Figures/}}

\title{Visual Aspects of Gaussian Periods and Analogues}
\author[Samantha Platt]{Samantha Platt$^\dag$}
\thanks{$^\dag$Partially supported by the Paul and Harriet Civin Memorial Graduate Student Award and the E. M. Johnson Memorial Scholarship.}

\begin{document}

\begin{abstract}
	Gaussian periods have been studied for centuries in the realms of number theory, field theory, cryptography, and elsewhere. However, it was only within the last decade or so that they began to be studied from a visual perspective. By plotting Gaussian periods in the complex plane, various interesting and insightful patterns emerge, leading to various conjectures and theorems about their properties. In this paper, we offer a description of Gaussian periods, along with examples of the structure that can occur when plotting them in the complex plane. In addition to this, we offer two ways in which this study can be generalized to other situations---one relating to supercharacter theory, the other relating to class field theory---along with discussions and visual examples of each.
\end{abstract}

\maketitle

\tableofcontents

\section{Introduction}

Gaussian periods are certain exponential sums which are important in various areas of mathematics, especially number theory. Gauss originally used them when studying ruler and compass construction and quadratic reciprocity, and they have since been used by other well-known mathematicians. For example, Kummer studied them while proving part of Fermat's Last Theorem (and elsewhere), and Lenstra and Pomerance used them to better optimize the AKS primality test (the first deterministic, polynomial-time primality test). Gaussian periods also have connections to Gauss sums, which show up in the functional equations of Dirichlet $L$-functions. For these reasons and more, Gaussian periods have been used and studied in a variety of ways throughout the years.

However, Gaussian periods were not studied \emph{visually} until only very recently, where the increased computational capabilities of modern computers have made this possible. By plotting Gaussian periods in the complex plane, one immediately notices that many striking visual patterns emerge. Given how historically important and useful Gaussian periods have been, it is natural to study these plots in their own right in order to explain the mathematical structures underlying them.

To the author's knowledge, the visual study of Gaussian periods started mainly with works from Brumbaugh et al., Duke, Garcia, Hyde, and Lutz \cite{ExpSumsBrumbaugh, GraphicSymmetric, DGL, Menagerie}. In Section \ref{sec:defns}, we formally introduce Gaussian periods and Gaussian period plots, and we provide examples of these plots and some of their properties.

Generally speaking, the above mentioned authors were studying Gaussian periods through the framework of supercharacter theory---a generalization of character theory which can be more conducive to computations. In Section \ref{sec:supchar}, we discuss this framework in more depth, and we describe one way in which Gaussian periods can be generalized using this perspective. Using this generalization, we prove Theorem \ref{thm:DGLgeneralize}, which substantially extends a theorem of Duke, Garcia, and Lutz \cite[Theorem 6.3]{DGL}. Additionally, in Section \ref{sec:tracemaps}, we introduce a more dynamic perspective with which to study Gaussian periods, and we prove Proposition \ref{prop:animationswithrollinghypocycloids}, which describes and explains certain dynamic behaviors of Gaussian period plots.

Although the existing literature has focused on Gaussian period plots through a supercharacter theory perspective, there is another (seemingly less explored) framework with which to view plots of Gaussian periods, which is that of number theory and class field theory. In particular, the well-known result of Kronecker and Weber states that every finite abelian extension of the rational numbers is contained in some cyclotomic field. Given that Gaussian periods are certain sums of roots of unity, one might then wonder how Gaussian periods could be generalized to other base fields, what sorts of behaviors these generalizations exhibit, and what sort of insight this study could provide. We discuss these questions and more in Section \ref{sec:classfield}, and in particular, we explicitly define such a generalization for quadratic imaginary base fields in Definition \ref{def:RCFP}. Included in this discussion is a description of the explicit class field theory for quadratic imaginary fields in Section \ref{sec:classfieldellipticcurvesandCM}, which involves the study of torsion points of elliptic curves with complex multiplication. Additionally, we provide a proof of Proposition \ref{prop:GalgroupofrayclassfieldoverHilbert}, which computes the Galois group of ray class fields of quadratic imaginary fields over their Hilbert class fields.

Overall, the goal of this paper is to continue the visual exploration of Gaussian periods, while also motivating and initiating the visual exploration of certain generalizations of Gaussian periods from other perspectives.

\subsection{Acknowledgements}

I would like to thank the following people for their help at various points of this project. First, I thank Ellen Eischen for the original idea for the project and for the suggestions and advice throughout. I thank Benjamin Young for his serendipitous observations about the project, which led to interesting insights. I would also like to thank David Lowry-Duda, John Voight, and Joseph Silverman for their helpful conversations and suggestions about the implementation of coding ideas, which took place mainly at the Algorithmic Number Theory Symposium in 2022 and at an MSRI workshop in winter 2023. I would also like to thank April Wade for her extensive help with the code itself, especially with helping the code run efficiently. 

\section{Definitions and Motivations}

\label{sec:defns}

We begin by clarifying the definition of Gaussian periods which we will be using. Throughout this paper, we define $e(x) := e^{2 \pi i x}$. 

\begin{definition} 
	Let $n$ be an integer, and let $\omega$ be an integer coprime to $n$. Then $\omega \in (\Z/n\Z)^\times$, and we let $d$ be the multiplicative order of $\omega$ modulo $n$. For an integer $k$ (and using notation similar to \cite{Gallery}), we define the following map: $$\eta_{n, \omega}: \Z/n\Z \to \C, \hspace{0.5in} \eta_{n, \omega}(k) := \sum_{j = 0}^{d - 1} e\left(\frac{\omega^j k}{n}\right).$$ We call $\eta_{n, \omega}(k)$ a \textit{Gaussian period of modulus $n$ and generator $\omega$}; note that we do not require $k$ to be relatively prime to $n$ in our definition. Additionally, we call $\img(\eta_{n, \omega})$ the \textit{Gaussian period plot of modulus $n$ and generator $\omega$} (or simply the \textit{Gaussian period plot} where $n$ and $\omega$ are clear from context).
\end{definition}

In Figure \ref{fig:GaussPerPlots}, we provide examples of Gaussian period plots for various choices of $n$ and $\omega$. Two Gaussian periods $\eta_{n, \omega}(k)$ and $\eta_{n, \omega}(k')$ have the same color in these plots if $k \equiv k' \mod c$, where $c$ is a chosen color modulus. We won't focus too much on the color scheme in this paper, but the reader can refer to \cite{Gallery} and \cite[\S 3]{Menagerie} for more information.

\begin{figure}[h!]
	\centering
	\begin{subfigure}[b]{0.32\linewidth}
		\includegraphics[width=\linewidth]{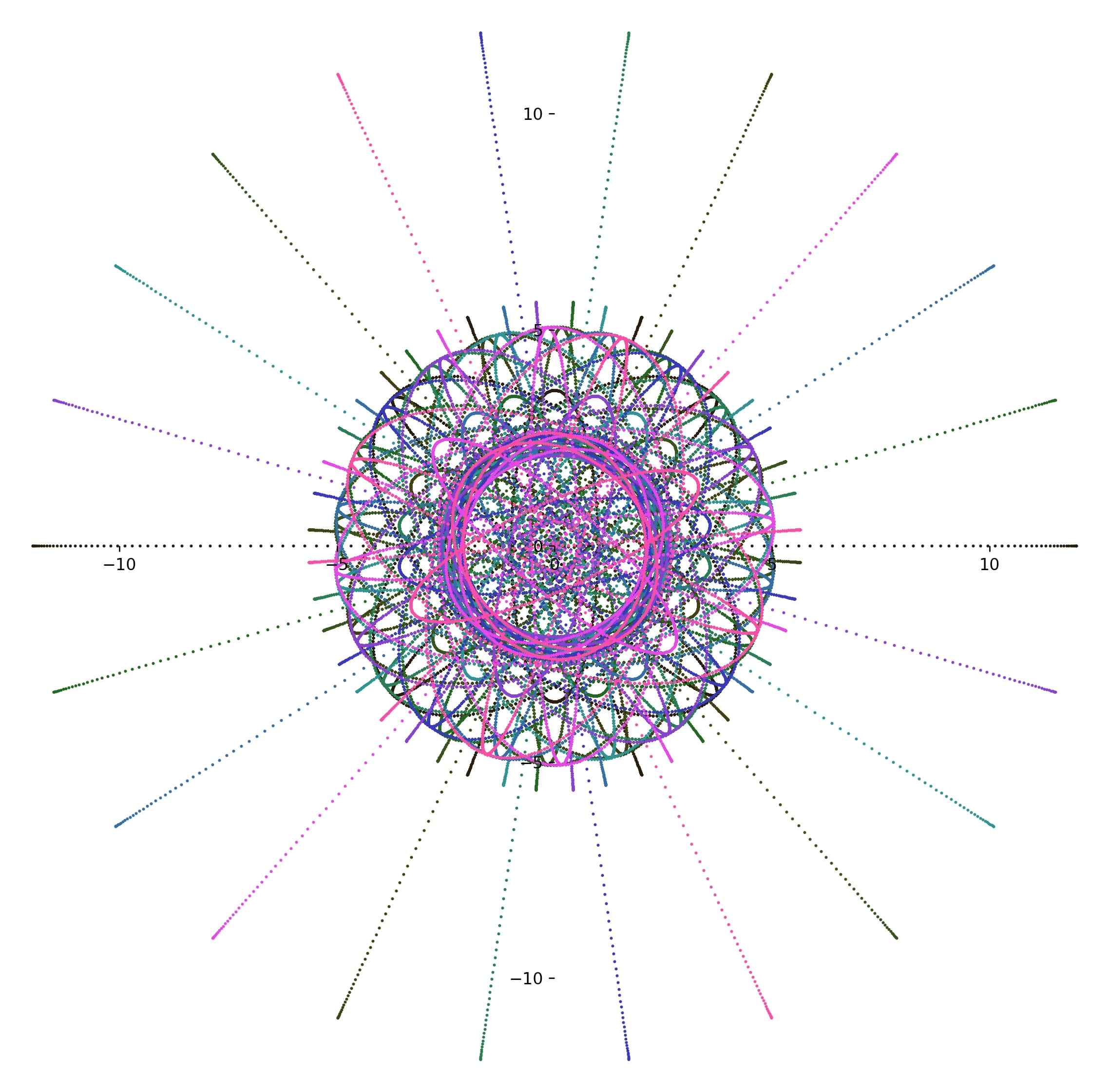}
		\caption{$n = 255255$, $\omega = 254$, $c = 11$}
	\end{subfigure}
	\begin{subfigure}[b]{0.32\linewidth}
		\includegraphics[width=\linewidth]{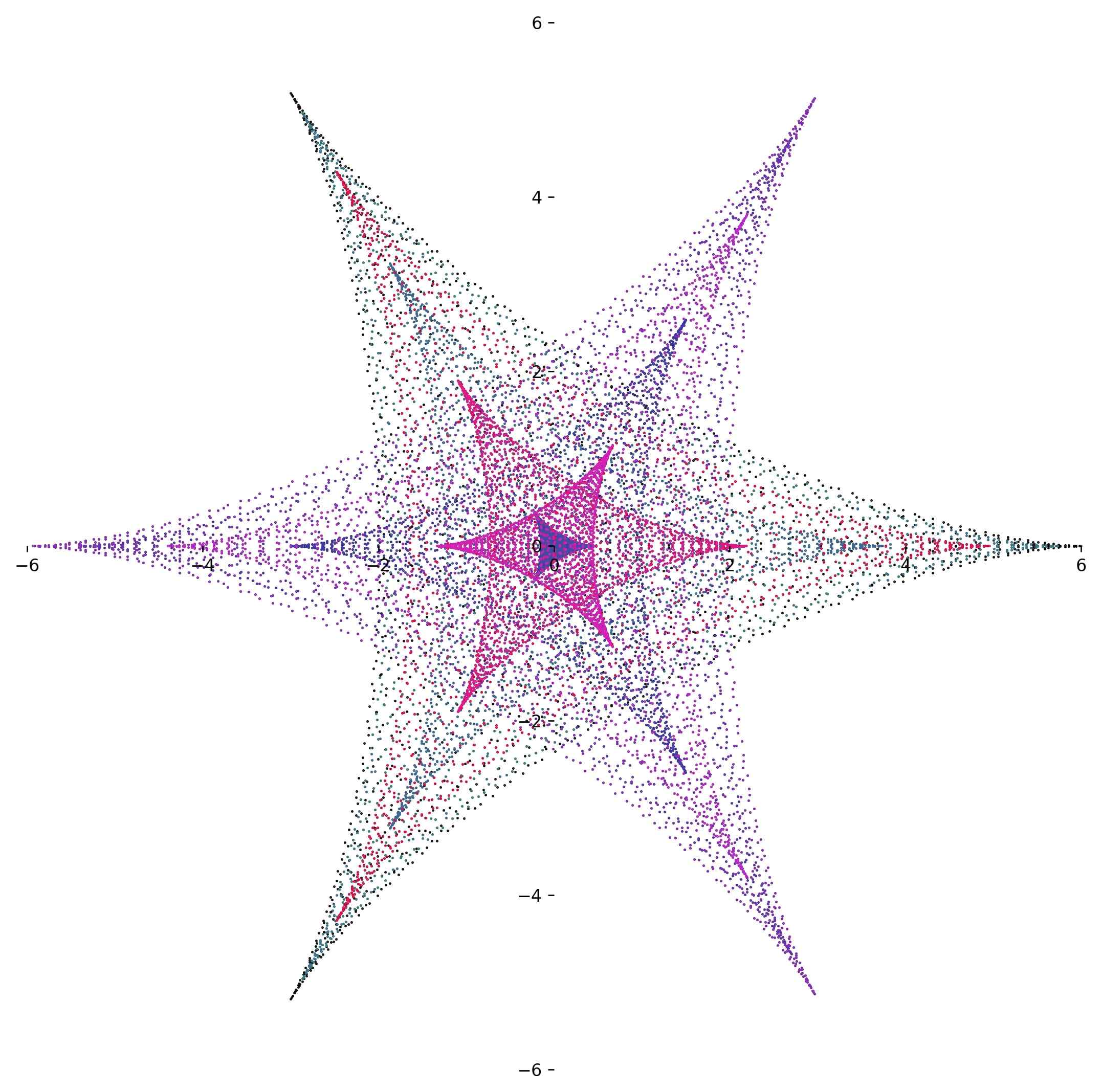}
		\caption{$n = 82677$, $\omega = 8147$, $c = 21$}
	\end{subfigure}
	\begin{subfigure}[b]{0.32\linewidth}
		\includegraphics[width=\linewidth]{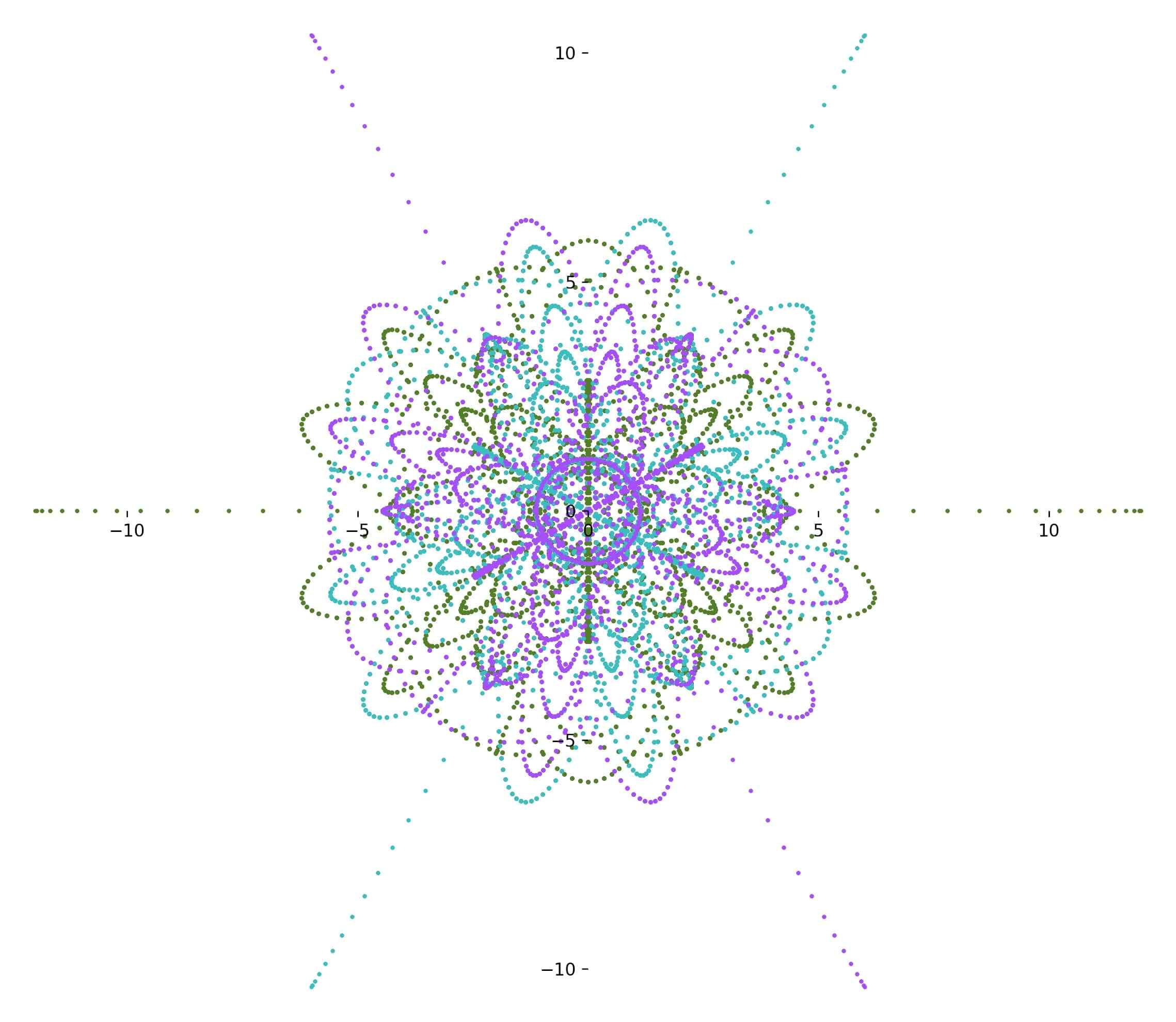}
		\caption{$n = 62160$, $\omega = 319$, $c = 3$}
	\end{subfigure}

	\begin{subfigure}[b]{0.32\linewidth}
		\includegraphics[width=\linewidth]{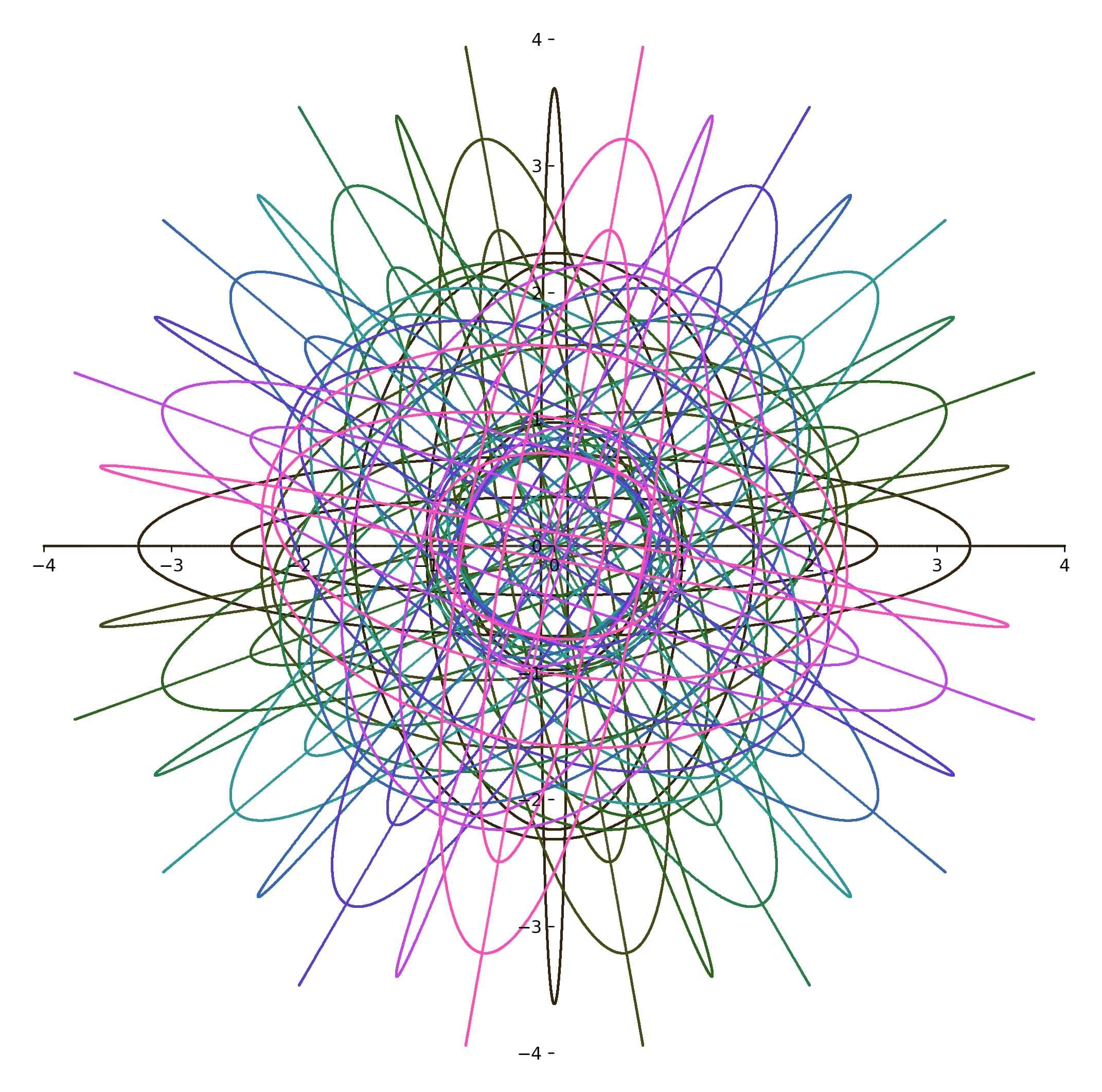}
		\caption{$n = 455175$, $\omega = 107218$, $c = 9$}
	\end{subfigure}
	\begin{subfigure}[b]{0.32\linewidth}
		\includegraphics[width=\linewidth]{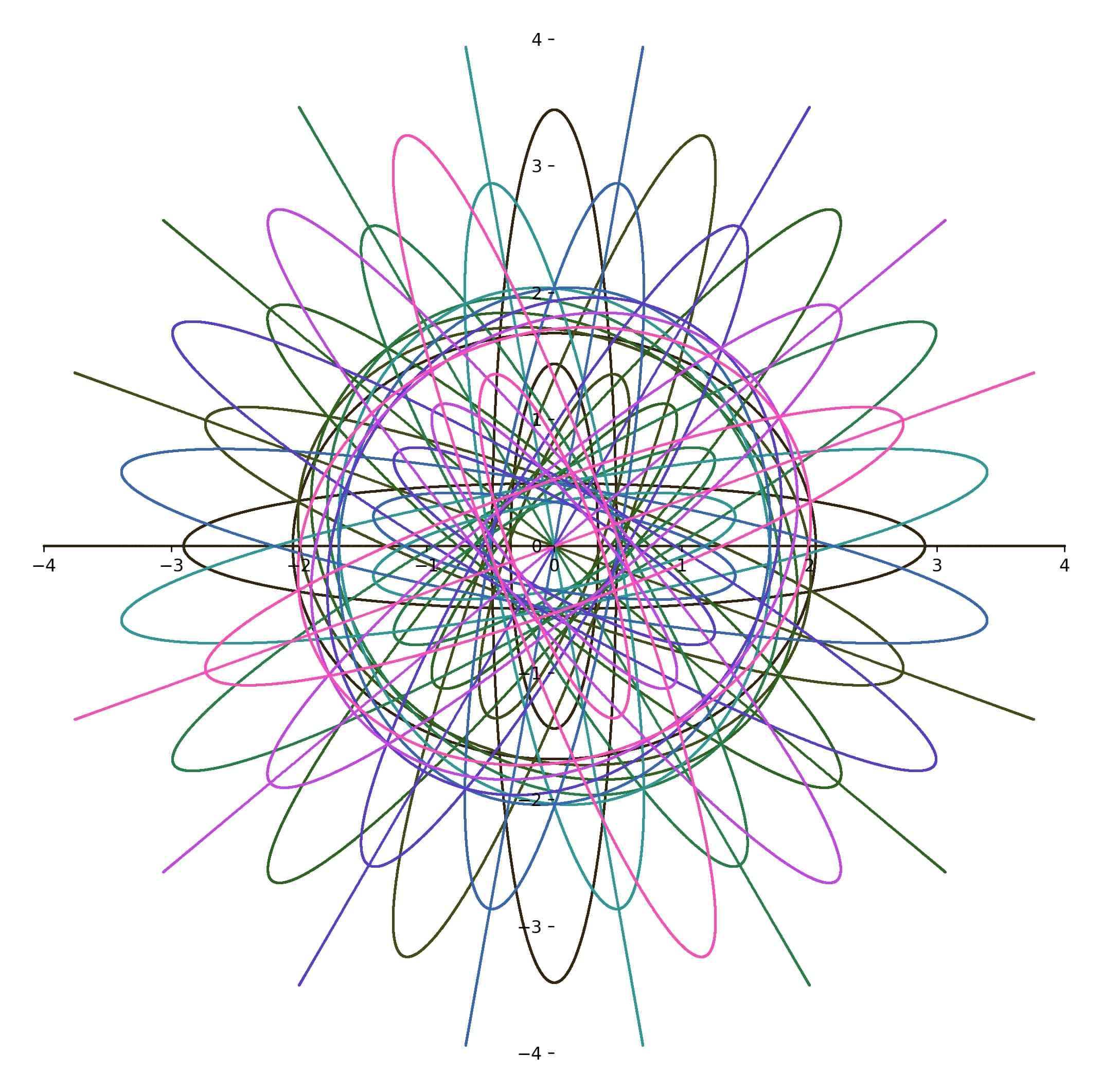}
		\caption{$n = 478125$, $\omega = 3124$, $c = 9$}
	\end{subfigure}
	\begin{subfigure}[b]{0.32\linewidth}
		\includegraphics[width=\linewidth]{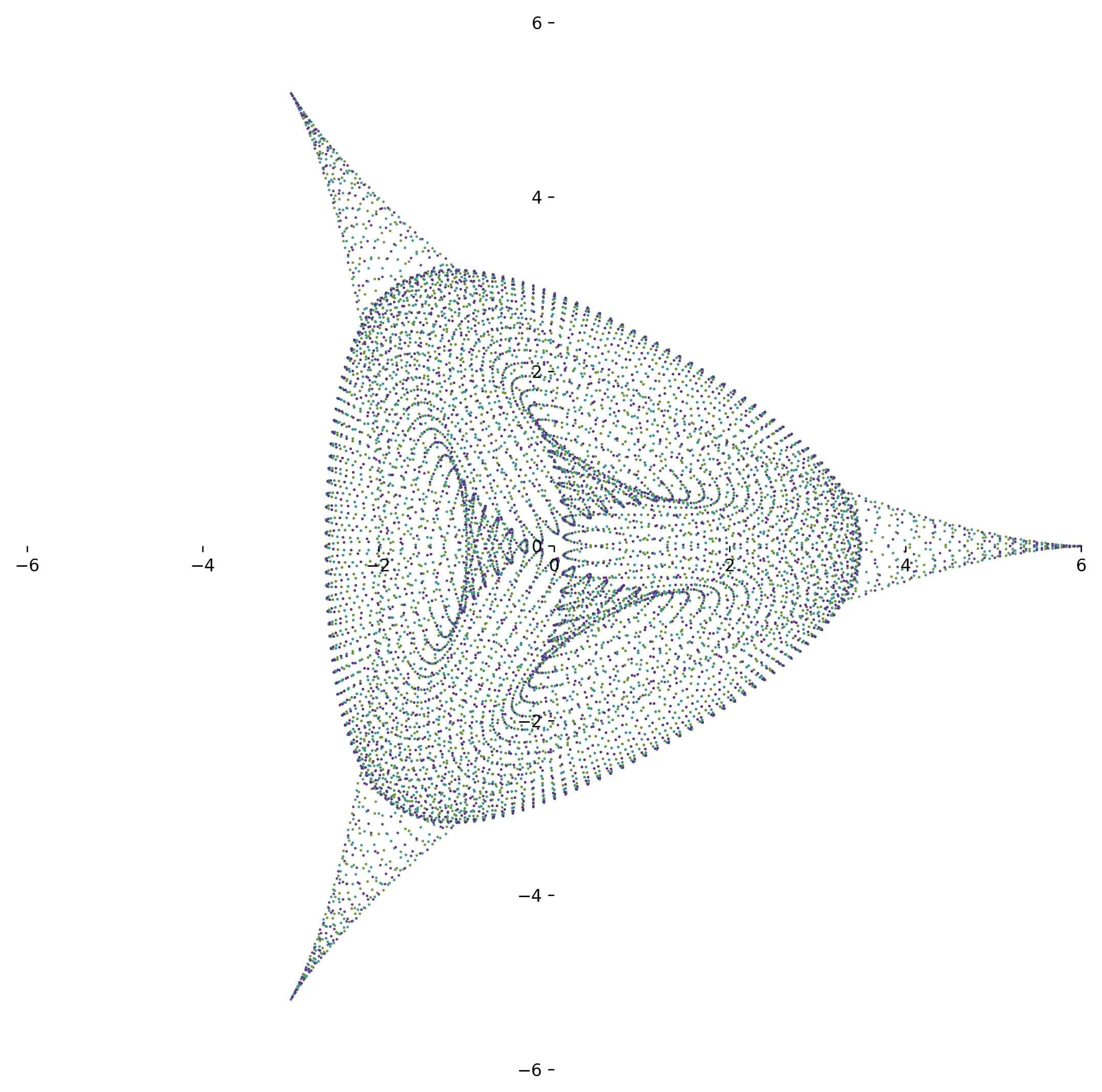}
		\caption{$n = 52059$, $\omega = 766$, $c = 3$}
	\end{subfigure}

	\begin{subfigure}[b]{0.32\linewidth}
		\includegraphics[width=\linewidth]{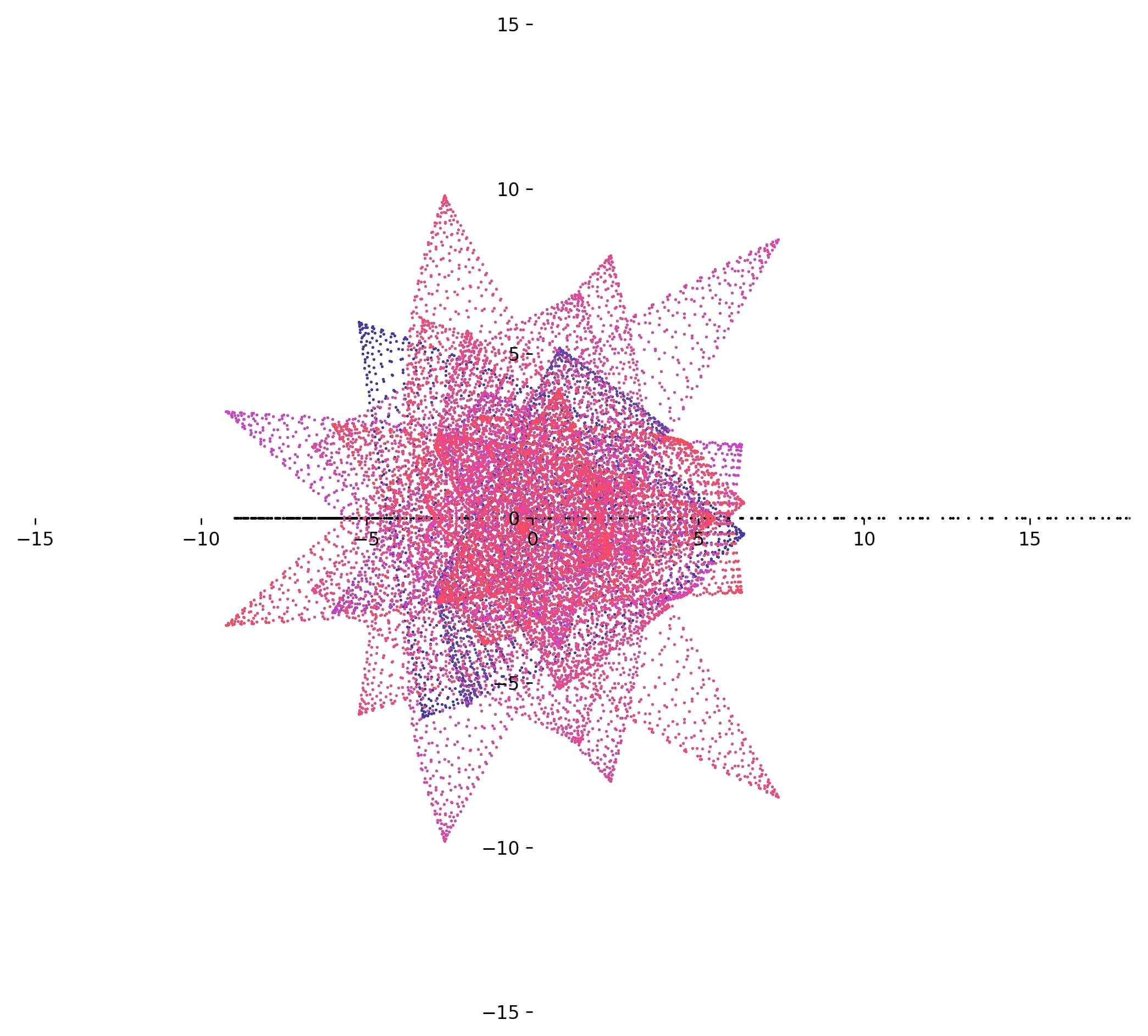}
		\caption{$n = 328549$, $\omega = 9247$, $c = 127$}
	\end{subfigure}
	\begin{subfigure}[b]{0.32\linewidth}
		\includegraphics[width=\linewidth]{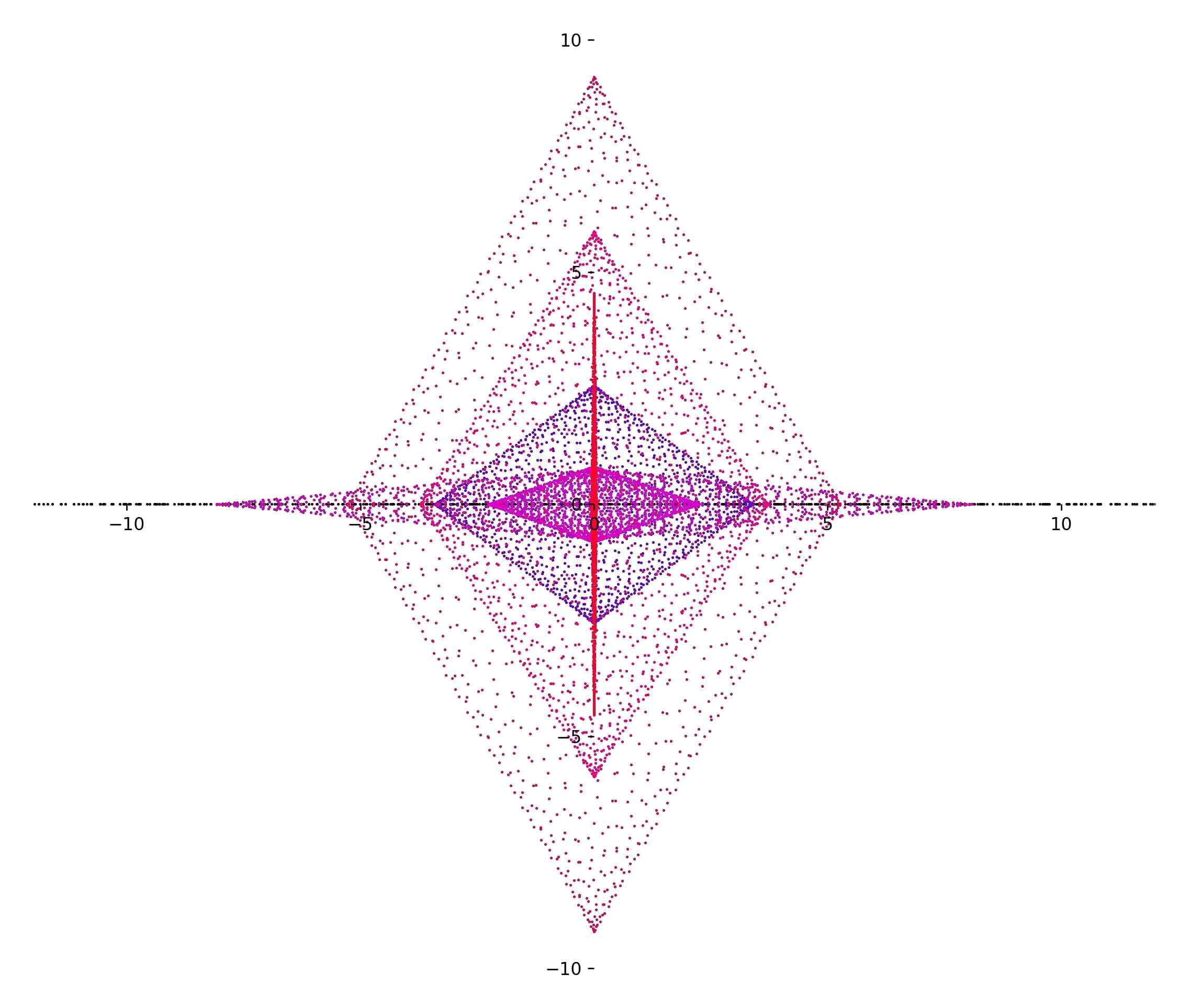}
		\caption{$n = 91205$, $\omega = 1322$, $c = 37$}
	\end{subfigure}
	\begin{subfigure}[b]{0.32\linewidth}
		\includegraphics[width=\linewidth]{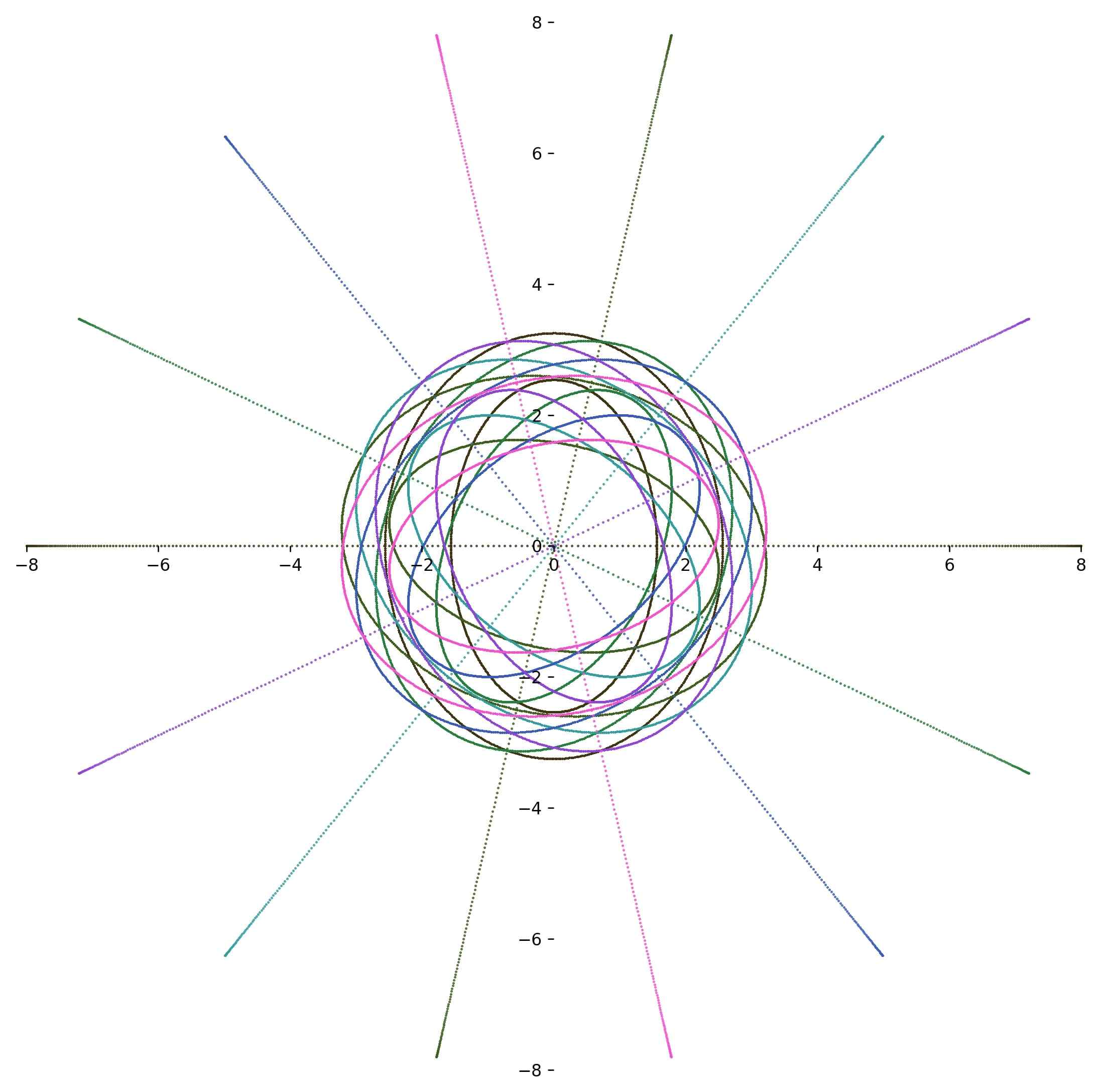}
		\caption{$n = 70091$, $\omega = 21792$, $c = 7$}
	\end{subfigure}
	\caption{Examples of Gaussian period plots for various choices of $n$ and $\omega$}
	\label{fig:GaussPerPlots}
\end{figure}

It is now that one might wonder if Gaussian period plots might be described or explained mathematically, as they obviously seem to have some sort of structure and symmetry to them. While there are some cases in which such a description is still unknown, there are many in which we do indeed have an explanation. We offer one such explanation, proved by Duke, Garcia, and Lutz \cite{DGL}, as motivation.

\begin{theorem}[Theorem 6.3 of \cite{DGL}]
	Let $n = p^a$, where $p$ is an odd prime. Choose $\omega \in (\Z/p^a\Z)^\times$ so that it has multiplicative order $d$ dividing $p - 1$. Let $\Phi_d(x)$ denote the $d$-th cyclotomic polynomial, and let $\T$ denote the complex unit circle. Then the Gaussian period plot $\img(\eta_{n, \omega})$ is contained in the image of the Laurent polynomial function $g_d: \T^{\varphi(d)} \to \C$ defined by $$g_d(z_1, z_2, \ldots, z_{\varphi(d)}) = \sum_{j = 0}^{d - 1} \prod_{m = 0}^{\varphi(d) - 1} z_{m + 1}^{c_{mj}},$$ where the constants $c_{mj}$ are defined by the following relations: $$x^j \equiv \sum_{m = 0}^{\varphi(d) - 1} c_{mj} x^m \mod \Phi_d(x).$$ Moreover, for a fixed $d$, as $p^a$ tends to infinity (where $p$, $a$, and the choice of $\omega$ are allowed to vary, as long as $d$ divides $p - 1$), every nonempty open disc contained in $\img(g_d)$ eventually contains points in $\img(\eta_{n, \omega})$. In other words, $\img(g_d)$ is ``filled out'' by Gaussian periods as $n$ goes to infinity.
	
	\label{thm:DGL}
\end{theorem}

It is interesting to note that in the case where $d$ is itself also a prime, the image $\img(g_d)$ is a \emph{$d$-sided hypocycloid}. We provide examples of this phenomenon (for $d$ both prime and composite) in Figure \ref{fig:DGLExamples}. Additionally, it may be of interest to the reader to compute $g_d$ for various values of $d$; to this end, we wrote some code in Sage that generates $g_d$, and it is included in the GitHub link provided in Section \ref{sec:Code}. 

\begin{figure}[h!]
	\centering
	\begin{subfigure}[b]{0.37\linewidth}
		\includegraphics[width=\linewidth]{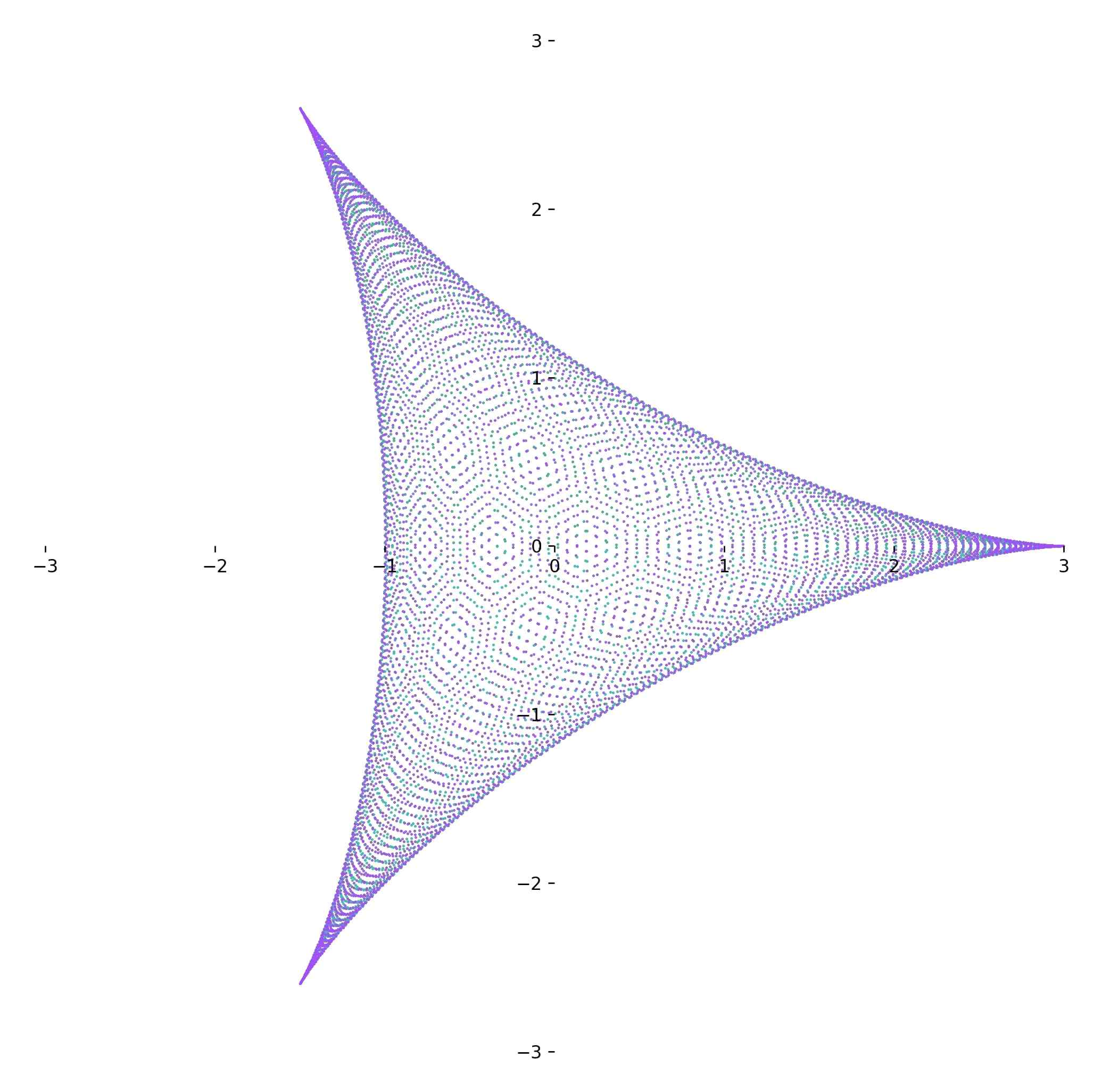}
		\caption{$n = 31^3$, $\omega = 23503$, $d = 3$}
	\end{subfigure}
	\begin{subfigure}[b]{0.37\linewidth}
		\includegraphics[width=\linewidth]{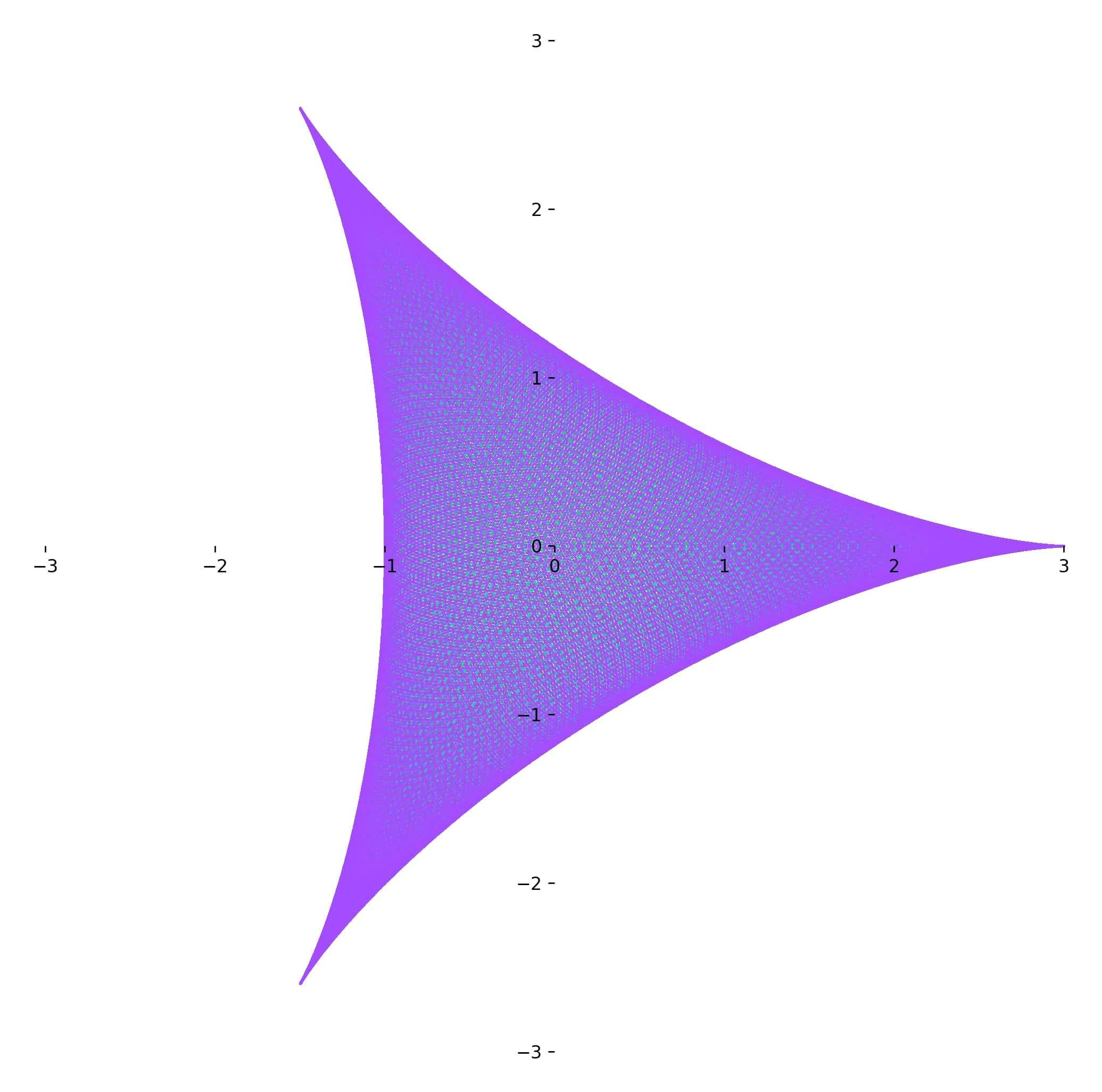}
		\caption{$n = 13^5$, $\omega = 220860$, $d = 3$}
	\end{subfigure}
	
	\begin{subfigure}[b]{0.37\linewidth}
		\includegraphics[width=\linewidth]{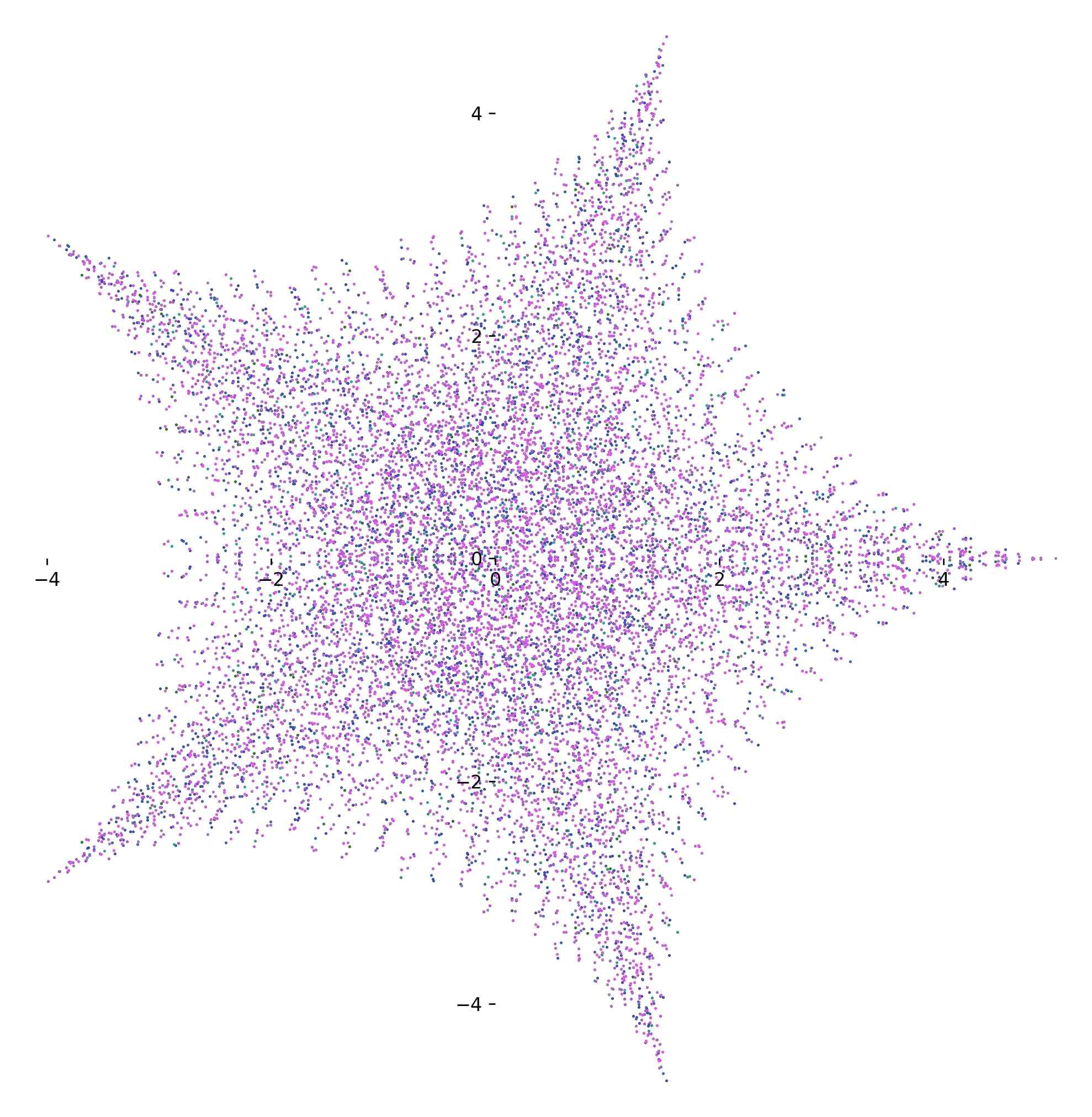}
		\caption{$n = 41^3$, $\omega = 46244$, $d = 5$}
	\end{subfigure}
	\begin{subfigure}[b]{0.37\linewidth}
		\includegraphics[width=\linewidth]{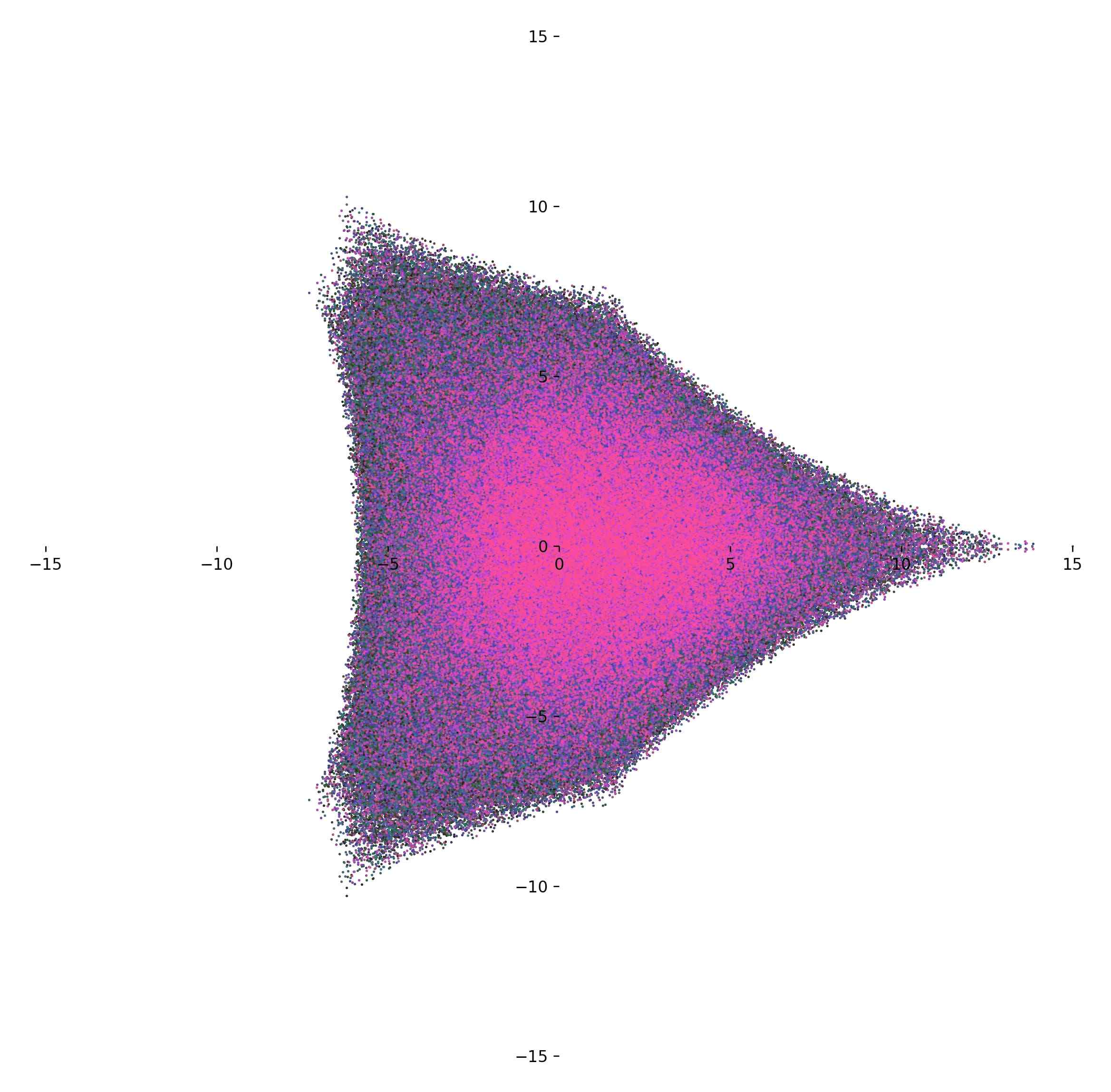}
		\caption{$n = 31^5$, $\omega = 17404906$, $d = 15$}
	\end{subfigure}

	\caption{Examples of Duke--Garcia--Lutz Theorem for various values of $d$}
	\label{fig:DGLExamples}
\end{figure}

Additionally, we believe it is worth noting that this theorem has been reinterpreted by Untrau in \cite{untrau2021equidistribution} and Kowalski--Untrau in \cite{untraukowalski2023ultrashort} as a statement about the distribution of points in the complex plane under the pushforward measure of the Haar measure on $\T^{\varphi(d)}$ via the map $g_d$. Untrau and Kowalski then use this reinterpretation as a starting point from which they generalize to other situations.

\subsection{Supercharacter Theory}

As mentioned previously, there are two main perspectives which this paper will use when viewing Gaussian periods and their analogues: supercharacter theory and class field theory. We start by defining a supercharacter theory.

\begin{definition}
	Let $G$ be a finite group with identity $1$, let $\sK$ be a partition of $G$, and let $\sX$ be a partition of the set of irreducible characters of $G$. Then $(\sX, \sK)$ is a \textit{supercharacter theory} for $G$ if the following hold:
	\begin{itemize}
		\item $\{1\} \in \sK$
		
		\item $|\sX| = |\sK|$
		
		\item For each $X \in \sX$, the function $\sigma_X = \sum_{\chi \in X} \chi(1) \chi$ is constant on each $K \in \sK$. 
	\end{itemize}
	The elements $K \in \sK$ are called \emph{superclasses}, and the functions $\sigma_X$ are called \emph{supercharacters}. 
\end{definition}

Using the above definition, let $G = \Z/n\Z$. For every $x \in G$, there exists an irreducible character $\chi_x$ such that $\chi_x(y) = e\left(\frac{xy}{n}\right)$ for every $y \in G$, and these characters describe all of the irreducible characters of $G$. For the cyclic subgroup $\langle \omega \rangle \subset (\Z/n\Z)^\times$, let $\sK$ be the partition of $G$ corresponding to the orbits of the action $a \cdot x = ax$ for $a \in \langle \omega \rangle$. Additionally, let $\sX$ be the partition of the irreducible characters of $G$ corresponding to the action $a \cdot \chi_x = \chi_{a^\inv x}$. One can then check that $\sX$ and $\sK$ are compatible as in the definition given above, so $(\sX, \sK)$ defines a supercharacter theory on $\Z/n\Z$. 

Thus the Duke--Garcia--Lutz Theorem, for example, gives a geometric description of the values which show up in a ``supercharacter table'' for $(\sX, \sK)$.

\begin{remark}
	It should be noted that the above construction of the partitions $\sX$ and $\sK$ work for \emph{any} subgroup $\Gamma \subset (\Z/n\Z)^\times$ and not just the cyclic subgroups as we have described them. In this paper, however, we mainly restrict ourselves to the case in which $\Gamma$ is cyclic. Supercharacter theories induced by non-cyclic subgroups of $(\Z/n\Z)^\times$ have received little attention in previous papers, though we believe it is a subject worth exploring.
\end{remark}

\begin{remark}
	Supercharacter theory was first described axiomatically by Diaconis and Isaacs in 2008 in \cite{Supercharacter}. As mentioned in \cite{Menagerie}, supercharacter theory has been used in the study of a variety of objects, including the Hopf algebra of symmetric functions of non-commuting variables, random walks on upper triangular matrices, and Ramanujan sums. 
\end{remark}

\subsection{Class Field Theory}

\begin{remark}
	It is worth noting that we will not be describing class field theory in great detail in this paper, as we will be providing only a general overview for motivation. Because of this, it isn't expected that the reader know these details in order to follow any results; the reader may need to accept certain statements as fact, but we believe that very little will be lost in doing so. For further reading about class field theory, we recommend \cite{Childress}.
\end{remark}

Given a base field $K$, the main goal of class field theory is to describe all finite abelian extensions of $K$ using its local properties. In most cases, class field theory allows us to compute the Galois groups of these field extensions fairly easily; however, explicitly finding the fields corresponding to such Galois groups is often a much harder task. In fact, this is the subject of Hilbert's 12th problem, which asks which algebraic numbers must be adjoined to $K$ in order to generate its abelian extensions. 

The answer to Hilbert's 12th problem is known in only very few cases. For example, when $K = \Q$, the Kronecker-Weber Theorem states that every finite abelian extension of $\Q$ is contained in some finite cyclotomic extension of $\Q$. Thus the roots of unity are the algebraic numbers needed to generate abelian extensions. 

Another case in which the answer is known (and which is a source of exploration later in this paper) is the case in which $K$ is a quadratic imaginary field, where the well-known theory of complex multiplication tells us that all finite abelian extensions of $K$ are generated by adjoining certain values of the modular $j$-function along with coordinates of torsion points of certain elliptic curves.

With this in mind, we offer the following definitions from class field theory.

\begin{definition}
	Let $K$ be a number field, and let $\O_K$ be its ring of integers. Given an ideal $\m \subset \O_K$ called the \emph{modulus}, we obtain a \emph{ray class group of modulus $\m$}, which will be denoted $\Cl_K(\m)$ in this paper. For each ray class group of modulus $\m$, there exists a \emph{ray class field of modulus $\m$}, denoted $K[\m]$, whose Galois group $\Gal(K[\m]/K)$ is isomorphic to $\Cl_K(\m)$ and whose set of ramified primes are only those which divide $\m$. In the special case where $\m = (1)$, we call $K[1]$ the \emph{Hilbert class field}, whose Galois group $\Gal(K[1]/K)$ is isomorphic to the ideal class group $\Cl_K(1)$ of $K$. 
\end{definition}

In addition to the above definitions, we make two important notes about ray class groups and ray class fields. First, if $\mathfrak{n}$ is an ideal dividing $\m$, then $\Cl_K(\mathfrak{n}) \subset \Cl_K(\m)$ and $K[\mathfrak{n}] \subset K[\m]$. One important implication of this is that the Hilbert class field is an intermediate field extension for every modulus $\m$; that is, $K \subset K[1] \subset K[\m]$ for every $\m$.

Additionally, it is worth noting an alternate description of ray class fields. Rather than viewing them through the correspondence provided above, one can instead define ray class fields of modulus $\m$ to be the maximal abelian extension of $K$ which is ramified only at the primes dividing $\m$, assuming a certain boundedness condition related to $\m$. These two characterizations turn out to be equivalent, and we will be using them interchangeably in our discussion.

\begin{example} 
	In the case where $K = \Q$, the ring of integers is $\Z$. The ideal class group of $\Q$ is trivial, so the Hilbert class field is $\Q$ itself. Given an ideal $(m) \subset \Z$, the ray class group of modulus $m$ is isomorphic to $(\Z/m\Z)^\times$, and the corresponding ray class field is $\Q(\mu_m)$, where $\mu_m \subset \C^\times$ denotes the subgroup of $m$-th roots of unity.
\end{example}

Returning to the discussion of Gaussian periods, note that $e\left(\frac{k}{n}\right)$ is an $n$-th root of unity for every integer $k$. Thus every summand in $\eta_{n, \omega}(k)$ sits inside the $n$-th cyclotomic field $\Q[n] = \Q(\mu_n)$. Since $\omega$ is assumed to be in $(\Z/n\Z)^\times \cong \Gal(\Q[n]/\Q)$, then the Gaussian periods of modulus $n$ and generator $\omega$ are sums over the Galois action of the cyclic subgroup $\langle \omega \rangle \subset (\Z/n\Z)^\times$. In particular, this implies that $\eta_{n, \omega}(k) \in \Q[n]^{\langle \omega \rangle}$, the subfield of $\Q[n]$ fixed by the action of $\langle \omega \rangle$.

In fact, Gaussian periods are not only \emph{contained} in subfields of ray class fields, but they are \emph{generators} of those subfields. These subfields are important objects of study in number theory, as they are still abelian extensions of $\Q$ ramified only at the primes dividing the modulus of the ray class field; they just aren't the maximal such extension. With this perspective in mind, we see that the Duke--Garcia--Lutz Theorem shows how these generators are distributed in $\C$, as well as a description of their asymptotic behavior.

\section{Observations and Generalizations: Supercharacter Theory Perspective}

\label{sec:supchar}

In this section, we explore certain aspects of the Duke--Garcia--Lutz Theorem from the supercharacter theory perspective, as well as prove a generalization of this theorem. As the Duke--Garcia--Lutz Theorem will be discussed at length in this section, we recommend the reader review Theorem \ref{thm:DGL} above. 

\subsection{Gaussian Periods as Traces of Special Unitary Matrices}

\label{sec:tracemaps}

We first explore some observations made by the authors in sections 5 and 6 of \cite{Menagerie}. These observations relate to Gaussian period plots of modulus $n$ and generator $\omega$ when in the setting of the Duke--Garcia--Lutz Theorem. 

In particular, they discuss the Laurent polynomials $g_d$ showing up in the theorem, where $d$ is the multiplicative order of $\omega$. They start by noting that when $d$ is a prime, then $g_d$ has an easily described form. Since the cyclotomic polynomial $\Phi_d(x) = x^{d - 1} + x^{d - 2} + \cdots + x + 1$ when $d$ is prime, then we get the following: $$g_d(z_1, \ldots, z_{d - 1}) = z_1 + \cdots + z_{d - 1} + \frac{1}{z_1 z_2 \cdots z_{d - 1}}.$$ Because each $z_i$ is an element of $\T$ (i.e. a complex number on the unit circle), we can view the sum on the right-hand side as being the trace of some $d \times d$ special unitary matrix. This is because any special unitary matrix $U \in \SU(d)$ has eigenvalues $\{\lambda_1, \ldots, \lambda_d\}$, where $\lambda_i \in \T$ and $$\lambda_d = \frac{1}{\prod_{i = 1}^{d - 1} \lambda_i}.$$ Since the trace of a matrix is the sum of its eigenvalues, the trace of any $U \in \SU(d)$ is given as some output of $g_d$. 

There is, in fact, a little more to this story. In writing the code needed for plotting Gaussian periods, it seems that everyone (including this author) decided simply to plot all of the Gaussian periods at once. That is, given an $n$ and $\omega$, most algorithms that were written would compute all of the Gaussian periods, plot them all at once, and then return the result. However, when one instead plots the Gaussian periods $\eta_{n, \omega}(k)$ in batches instead of all at once, an interesting behavior takes place. This behavior was first discovered (after a happy coding mishap) by Benjamin Young, and we formalize this new perspective below. 

Given a modulus $n$ and element $\omega$ of multiplicative order $d$ in $(\Z/n\Z)^\times$, choose a constant $C$ that is small relative to $n$ (the author recommends $C \approx \sqrt{n}$, though sometimes the behavior is more clear with smaller or larger $C$). We then create $\lceil n/C \rceil$ ``frames'' by plotting the Gaussian periods $\eta_{n, \omega}(k)$ in batches of size approximately $C$. That is, the $f$-th frame will be the plotted image of $\eta_{n, \omega}(k)$ for all $0 \leq k < f C$. In this way, every frame adds approximately $C$ new points to the Gaussian period plot. Stringing these frames together gives an animation showing the behavior of Gaussian periods as they fill the Gaussian period plot $\img(\eta_{n, \omega})$. 

\begin{example}
	We consider the $d$-sided hypocycloid case from the Duke--Garcia--Lutz Theorem; that is, $n$ is a power of an odd prime and $d$ is itself also a prime, causing $\img(\eta_{n, \omega})$ to be contained within the image of a hypocycloid with $d$ sides. If we instead plot these Gaussian periods in frames using the constant $C = \sqrt{n}$, it becomes apparent that a $(d - 1)$-sided hypocycloid rolls counterclockwise along the inner boundary of the $d$-sided hypocycloid. For example, if $d = 3$, then a 2-sided hypocyloid (i.e. a straight line) rolls along the inside of a 3-sided hypocycloid, and when $d = 5$, a 4-sided hypocycloid rolls along in the inside of a 5-sided hypocycloid. We include some still frames of this phenomenon in Figure \ref{fig:AnimationStillFrames}. 
\end{example}

\begin{figure}[h!]
	\centering
	\begin{subfigure}[b]{0.3\linewidth}
		\includegraphics[width=\linewidth]{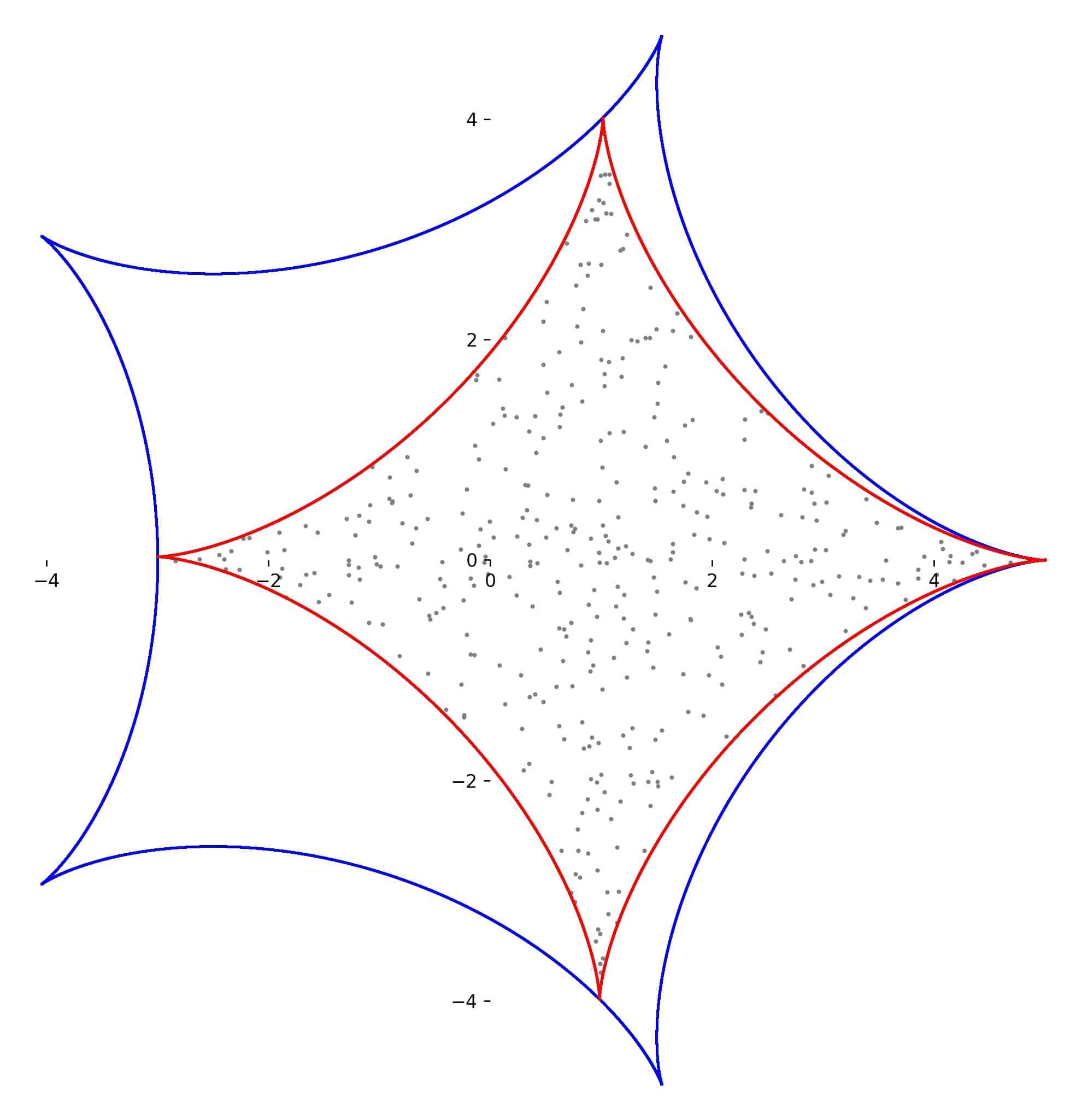}
		\caption{Up to $\lfloor \sqrt{n} \rfloor$}
	\end{subfigure}
	\begin{subfigure}[b]{0.3\linewidth}
		\includegraphics[width=\linewidth]{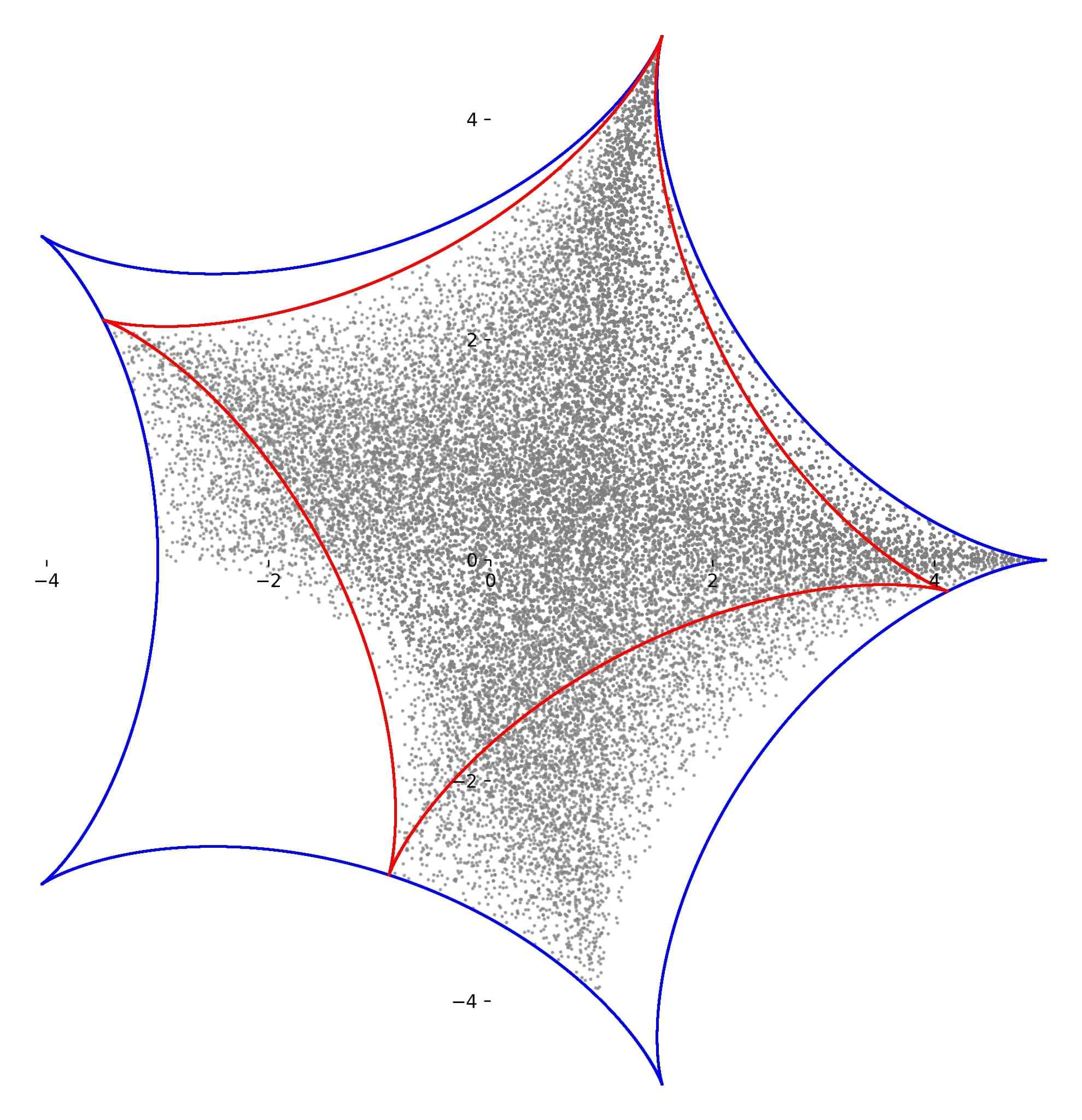}
		\caption{Up to $80 \lfloor \sqrt{n} \rfloor$}
	\end{subfigure}
	\begin{subfigure}[b]{0.3\linewidth}
		\includegraphics[width=\linewidth]{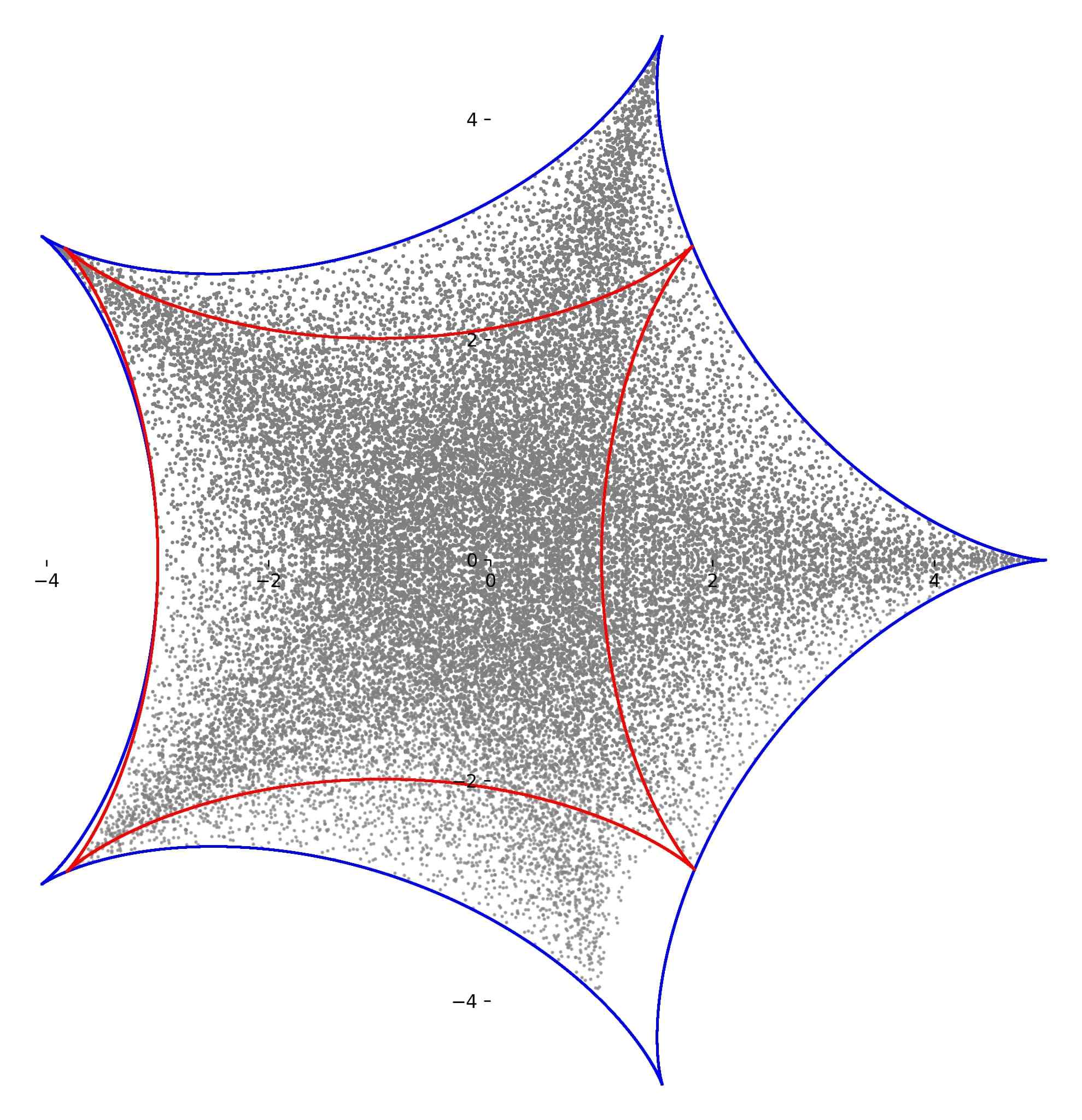}
		\caption{Up to $200 \lfloor \sqrt{n} \rfloor$}
	\end{subfigure}
	
	\caption{A 4-sided hypocycloid rolling along the inside of a 5-sided hypocycloid when $n = 11^5$ and $\omega = 37107$}
	\label{fig:AnimationStillFrames}
\end{figure}

The behavior above can be described more generally by looking at the following homomorphism: $$\varphi_m: \text{U}(1) \times \SU(m - 1) \into \SU(m), \hspace{.4in} (e^{i \theta}, U) \mapsto \diag(e^{i (m - 1) \theta}, e^{-i \theta} U),$$ where $\diag(e^{i (m - 1) \theta}, e^{-i \theta} U)$ represents the block-diagonal matrix with the $1 \times 1$ block $e^{i (m - 1) \theta}$ and the $(m-1) \times (m-1)$ block $e^{-i \theta} U$. Thus we see that $\Tr(\varphi_m(e^{i \theta}, U)) = e^{i (m - 1) \theta} + e^{-i \theta} \Tr(U)$. Since $U \in \SU(m-1)$, then $\Tr(U)$ is contained in an $(m-1)$-sided hypocycloid. Additionally, $e^{i (m - 1) \theta}$ is a element which moves counterclockwise on the unit circle as $\theta$ goes from $0$ to $2\pi$, and the multiplicative factor of $e^{-i \theta}$ causes $\Tr(U)$ to rotate clockwise as $\theta$ is increasing. Thus, if we allow $U \in \SU(m-1)$ to vary and have $\theta$ increase from $0$ to $2\pi$, the result is a filled-in $(m-1)$-sided hypocycloid rolling counterclockwise along the interior of an $m$-sided hypocycloid. 

The relationship between this homomorphism and example 2 is simple enough to see. For a given prime power $n$, a prime $d$, $\omega$ of order $d$, and $k \in (\Z/n\Z)$, recall the definition of a Gaussian period: $$\eta_{n, \omega}(k) = e^{\frac{2 \pi i k}{n}} + e^{\frac{2 \pi i \omega k}{n}} + \cdots + e^{\frac{2 \pi i \omega^{d - 1} k}{n}}.$$ If we let $\theta = \frac{2 \pi k}{(d - 1)n}$, then we have the following: $$\eta_{n, \omega}(k) = e^{i (d - 1) \theta} + e^{-i \theta} \left( e^{i \theta ((d - 1) \omega + 1)} + \cdots + e^{i \theta ((d - 1) \omega^{d - 1} + 1)} \right).$$ One can then verify that $$\prod_{j = 1}^{d - 1} e^{i \theta ((d - 1) \omega^j + 1)} = 1,$$ showing that $e^{i \theta ((d - 1) \omega + 1)}$, \ldots, $e^{i \theta ((d - 1) \omega^{d - 1} + 1)}$ are the eigenvalues of some matrix $U_k \in \SU(d - 1)$. Thus $\eta_{n, \omega}(k) = e^{i (d - 1) \theta} + e^{-i \theta} \Tr(U_k)$, which shows that $\eta_{n, \omega}(k) = \Tr(\varphi_d(e^{i \theta}, U_k))$ (i.e. that $\eta_{n, \omega}$ is the trace of some matrix in $\img(\varphi_d)$). Additionally, note that $\theta = \frac{2 \pi i k}{(d - 1)n}$ increases as $k$ increases, which explains the smooth rolling behavior of the $(d - 1)$-sided hypocycloid along the inner boundary of the $d$-sided hypocycloid. 

The discussion above proves the following proposition.

\begin{proposition}
	\label{prop:animationswithrollinghypocycloids}
	Let $n = p^a$ be a power of a prime, $d$ a prime dividing $p - 1$, and $\omega \in (\Z/n\Z)^\times$ an element of order $d$. Let $k \in \Z/n\Z$. Then the value of the $k$-th Gaussian period of modulus $n$ and generator $\omega$ is contained in a $(d - 1)$-sided hypocycloid centered at $e\left(\frac{k}{n}\right)$ and rotated by a factor of $e\left(\frac{-k}{(d - 1) n}\right)$. That is, $$\eta_{n, \omega}(k) \in \left\{e\left(\frac{k}{n}\right) + h \cdot e\left(\frac{-k}{(d - 1) n}\right) \st h \in H_{d - 1}\right\},$$ where $H_{d - 1}$ represents the filled-in $(d - 1)$-sided hypocycloid centered at the origin in the complex plane.
\end{proposition}

One might now wonder what occurs in the case where $d$ is not itself prime. Unfortunately, the geometrical shapes that one gets are much harder to describe succinctly, though the general behavior of ``a smaller shape rolling counterclockwise along the boundary'' seems to hold based on experimentation. But again, describing the ``smaller shape'' remains as elusive as describing the overall shape of the Gaussian period plots in these cases. It is our hope that someone more knowledgeable about this type of geometry might be able to offer some insight.

\subsection{Generalization of Duke--Garcia--Lutz}

\label{sec:DGLgeneralization}

We now return to viewing Gaussian periods as values of a supercharacter theory. Recall that the supercharacter theory used in Gaussian periods is constructed using a cyclic subgroup $\langle \omega \rangle \subset (\Z/n\Z)^\times$, where the supercharacters and superclasses are defined using a compatible action of $\langle \omega \rangle$ on the group $\Z/n\Z$ and its group of characters. We can generalize this setting in the following way.

Let $G = (\Z/n\Z)^m$. The irreducible characters of $G$ are simply products of irreducible characters on $\Z/n\Z$; in particular, given $\bold{x} = (x_1, \ldots, x_m) \in G$, we obtain an irreducible character $\chi_{\bold{x}}$ such that for $\bold{y} = (y_1, \ldots, y_m) \in G$, we have the following: $$\chi_{\bold{x}}(\bold{y}) = \chi_{x_1}(y_1) \cdot \chi_{x_2}(y_2) \cdots \chi_{x_m}(y_m) = e\left( \frac{x_1 y_1 + \cdots + x_m y_m}{n} \right).$$ The analogy of $(\Z/n\Z)^\times$ in this setting is the set of automorphisms $\Aut(G)$, which is isomorphic to $\GL_m(\Z/n\Z)$. Thus, we choose a matrix $A \in \GL_m(\Z/n\Z)$ of order $d$, and take the cyclic subgroup $\Gamma := \langle A \rangle$. We then define the (right) action of $\Gamma$ on $G$ to be $\gamma \cdot \bold{x} = \gamma^T \bold{x}$, where $\gamma \in \Gamma$, $\gamma^T$ is the transpose of the matrix $\gamma$, and $\bold{x} \in G$ (viewed as a column vector). Let $\sK$ be the partition of $G$ corresponding to the orbits of this action. Using the same notation, we define the (right) action of $\Gamma$ on the irreducible characters $\hat G$ to be $\gamma \cdot \chi_{\bold{x}} = \chi_{\gamma^\inv \bold{x}}$ for $\chi_{\bold{x}} \in \hat G$. Let $\sX$ be the partition of $\hat G$ corresponding to the orbits of this action. Then $(\sX, \sK)$ defines a supercharacter theory on $(\Z/n\Z)^m$. This follows from a more general statement by Brumbaugh et al. at the beginning of Section 2 in \cite{ExpSumsBrumbaugh}. 

While the authors of that paper did not state their result as a proposition, we do so here for the sake of clarity.

\begin{proposition}[From Section 2 of \cite{ExpSumsBrumbaugh}]
	Let $G = (\Z/n\Z)^m$, and let $\hat G$ be its group of characters. Let $\Gamma$ be a subgroup of $\GL_m(\Z/n\Z)$, and let $\sK$ and $\sX$ be the partitions defined in the preceding paragraphs for $G$ and $\hat G$, respectively. Then $(\sX, \sK)$ is a supercharacter theory on $G$.
\end{proposition}

\begin{remark}
	In the paper of Brumbaugh et al., they require that $\Gamma$ be a symmetric subgroup---i.e. that $\Gamma^T = \Gamma$. However, we note that this assumption is not used until after the supercharacter theory axioms were proved, so it is not necessary for the statement of the proposition.
\end{remark}

\begin{remark}
	We also note that the group action defined in \cite{ExpSumsBrumbaugh} differs slightly from the one defined above. However, this slight difference does not alter the details of the proof.
\end{remark}

Now, each irreducible character on $G$ is given by a choice of vector $\bold{x} \in (\Z/n\Z)^m$. In the case $m = 1$ (i.e. the case of Gaussian periods), we've seen that our choice of matrix $A$ simply describes scalar multiplication. In order to get the vectors $\bold{x}$ used for the irreducible characters $\chi_{\bold{x}}$, we used the set of vectors given by the orbit $\langle A \rangle \cdot 1$, where $1$ is the identity element in $(\Z/n\Z)^\times = \GL_1(\Z/n\Z)$. As noted in Proposition 2.2 of \cite{DGL}, if we instead choose any other nonzero $r \in \Z/n\Z$ for the orbit $\langle A \rangle \cdot r$, the corresponding plot of supercharacter values can be embedded within the plot corresponding to $r = 1$. This then justifies our use of the orbit $\langle A \rangle \cdot 1$ to get our characters, since it is the most general possible choice. Analogously for the case where $m > 1$, the vectors used for the irreducible characters will be the set of vectors given by the orbit $\langle A \rangle \cdot \bold{1}$, where $\bold{1} = (1, 1, \ldots, 1)^T \in (\Z/n\Z)^m$ is viewed as a column vector. 

\begin{definition} 
	For $\bold{x} \in (\Z/n\Z)^m$, we define the following cyclic supercharacter: $$\theta_{n, m, A}: (\Z/n\Z)^m \to \C, \hspace{0.5in} \theta_{n, m, A}(\bold{x}) := \sum_{j = 0}^{d - 1} e\left(\frac{A^j \bullet \bold{x}}{n}\right),$$ where $A^j \bullet \bold{x}$ represents $(A^j \bold{1}) \cdot \bold{x}$; that is, the dot product between $A^j \bold{1}$ and $\bold{x}$. We call $\img(\theta_{n, m, A})$ the \emph{cyclic supercharacter plot for $n$, $m$, and $A$}.
\end{definition}

We provide examples of the resulting cyclic supercharacter plots in Figure \ref{fig:SuperCharacterPlots}.

\begin{figure}[h!]
	\centering
	\begin{subfigure}[b]{0.32\linewidth}
		\includegraphics[width=\linewidth]{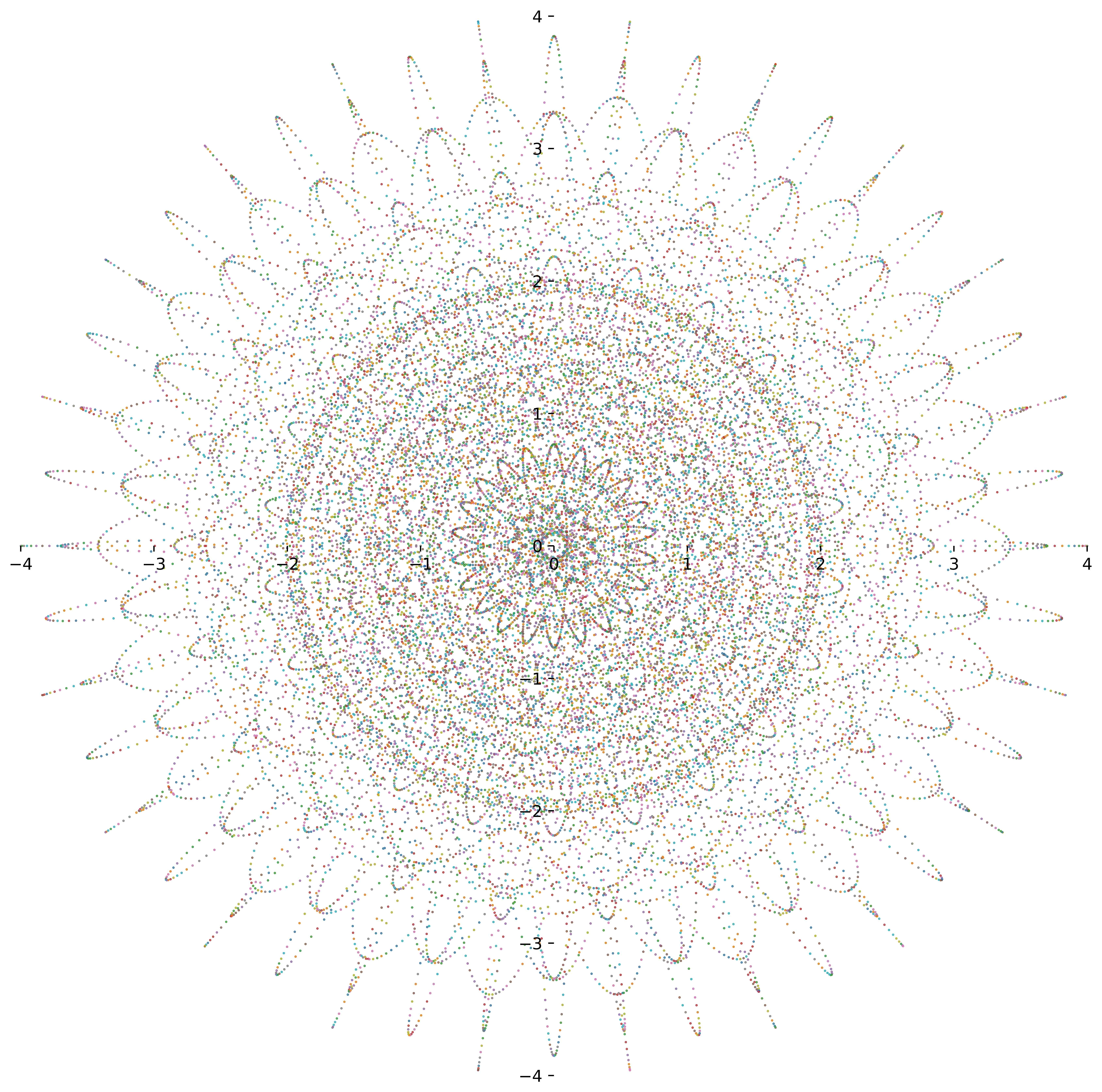}
		\caption{$n = 715$, $A = \left( \begin{smallmatrix} 0 & 2 \\ 61 & 121 \end{smallmatrix} \right)$}
	\end{subfigure}
	\begin{subfigure}[b]{0.32\linewidth}
		\includegraphics[width=\linewidth]{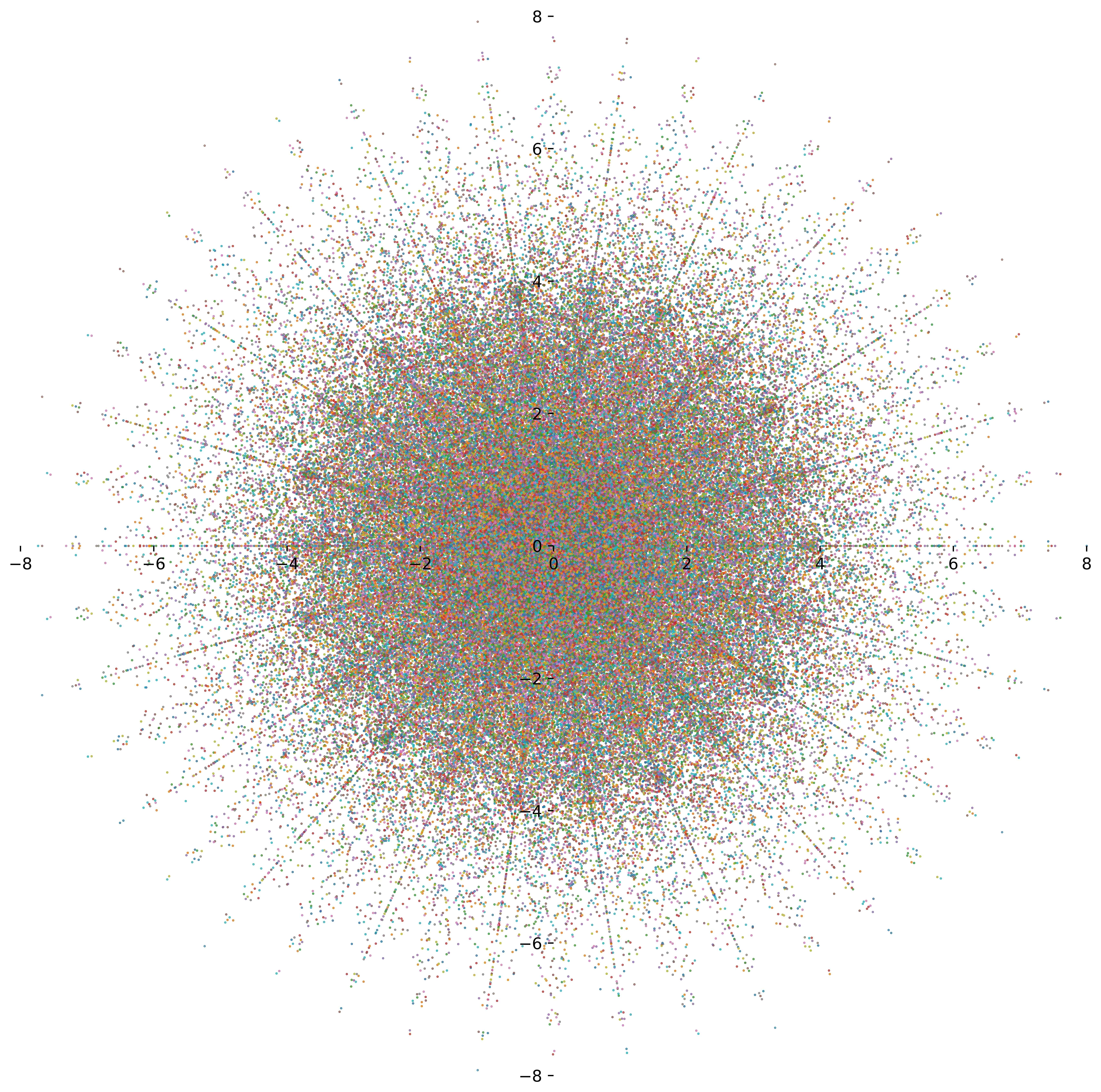}
		\caption{$n = 1155$, $A = \left( \begin{smallmatrix} 310 & 229 \\ 458 & 295 \end{smallmatrix} \right)$}
	\end{subfigure}
	\begin{subfigure}[b]{0.32\linewidth}
		\includegraphics[width=\linewidth]{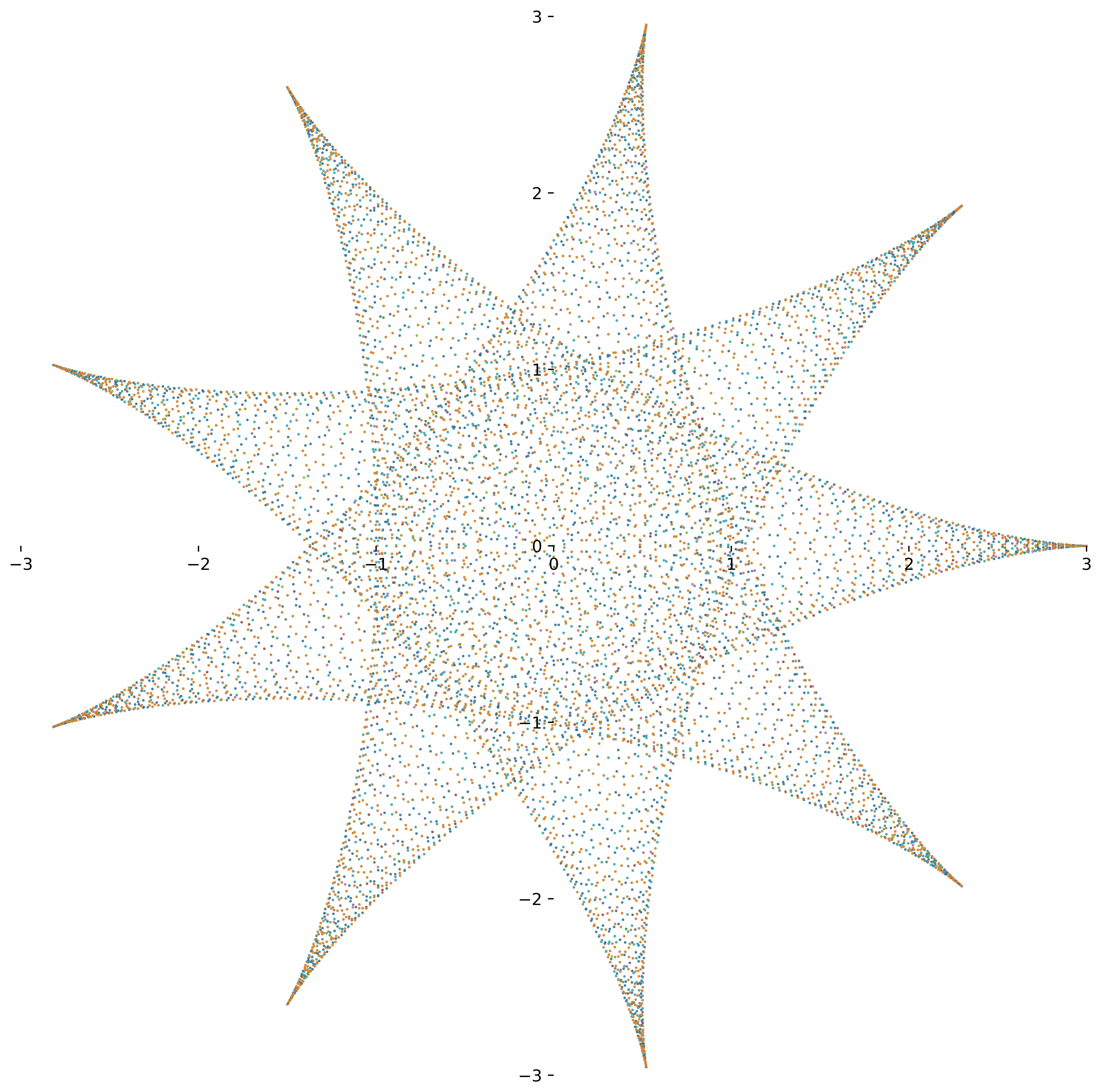}
		\caption{$n = 273$, $A = \left( \begin{smallmatrix} 51 & 224 & 140 \\ 168 & 268 & 14 \\ 203 & 203 & 170 \end{smallmatrix} \right)$}
	\end{subfigure}

	\caption{Examples of cyclic supercharacter plots for various $n$, $m$, and $A$.}
	\label{fig:SuperCharacterPlots}
\end{figure}

Our goal is to prove the following generalization of the Duke--Garcia--Lutz Theorem.

\begin{theorem}
	\label{thm:DGLgeneralize}
	Let $n \in \Z_{\geq 2}$ and $m \in \Z_{\geq 1}$. Suppose $d \mid (\# \GL_m(\Z/n\Z))$, and choose a matrix $A \in \GL_m(\Z/n\Z)$ of order $d$ such that $\Phi_d(A) = 0$ in $\Mat_m(\Z/n\Z)$, where $\Phi_d$ is the $d$-th cyclotomic polynomial. Let $\theta_{n, m, A}: (\Z/n\Z)^m \to \C$ be the cyclic supercharacter corresponding to $n$, $m$, and $A$. Then $\img(\theta_{n, m, A})$ is contained in the image of the Laurent polynomial function $g_d: \T^{\varphi(d)} \to \C$ defined by the following: $$g_d(z_1, z_2, \ldots, z_{\varphi(d)}) = \sum_{k = 0}^{d - 1} \prod_{j = 0}^{\varphi(d) - 1} z_{j+1}^{c_{jk}},$$ where the $c_{jk}$ are given by the relations $$x^k \equiv \sum_{j = 0}^{\varphi(d) - 1} c_{jk} x^j \mod \Phi_d(x).$$ Additionally, for a fixed order $d$, and as $n$ tends to infinity---assuming there exists a matrix $A \in \GL_m(\Z/n\Z)$ such that $\Phi_d(A) = 0 \mod n$---every nonempty open disk in the image of $g_d$ eventually contains points in $\img(\theta_{n, m, A})$. In other words, the image of $g_d$ is ``filled out'' by cyclic supercharacter plots as the modulus grows without bound.
\end{theorem}

Before we begin the proof of Theorem \ref{thm:DGLgeneralize}, we make some necessary detours. First, a brief remark about the areas in which this generalizes the Duke--Garcia--Lutz Theorem (which we shorten to DGL Theorem from here on).

\begin{remark}
	\label{rmk:DGLgeneralizeexplain}
	There are two directions in which this theorem generalizes the original DGL Theorem. First, it concerns the group $(\Z/n\Z)^m$ for any $m \geq 1$, rather than just the case $m = 1$. Second, it allows for composite moduli $n$, with the restriction that there is some matrix $A \in \GL_m(\Z/n\Z)$ which both has order $d$ and satisfies $\Phi_d(A) = 0 \mod n$. To see why this new condition on $n$ truly is a generalization, recall that the DGL Theorem assumes the modulus is $p^e$ for some power of an odd prime $p$. The group being studied is then $\Z/p^e\Z$, and the group of automorphisms is $$\GL_1(\Z/p^e\Z) \cong (\Z/p^e\Z)^\times.$$ Since $(\Z/p^e\Z)^\times$ is cyclic when $p$ is odd, then if $d \mid \# (\Z/p^e\Z)^\times$, there is always an element $\omega \in (\Z/p^e\Z)^\times$ of order $d$ (in general, note that there is not always a matrix $A \in \GL_m(\Z/n\Z)$ of order $d$ whenever $d \mid \# \GL_m(\Z/n\Z)$). Additionally, in the case where $m = 1$, the order $d$ divides $p - 1$ if and only if $\Phi_d(\omega) = 0 \mod p^e$. Theorem \ref{thm:DGLgeneralize} then generalizes the DGL Theorem using matrices which satisfy $\Phi_d$ rather than using matrices which have order $d$ (without necessarily satisfying $\Phi_d$).
\end{remark}

\begin{remark}
	\label{rmk:DGLgenpreamble}
	Note that Remark \ref{rmk:DGLgeneralizeexplain} hints at the possibility of another generalization of the DGL Theorem. The fact that $A \in \GL_m(\Z/n\Z)$ has order $d$ implies that minimal polynomial of $A$ must divide $x^d - 1$. When we additionally assume that $m = 1$ and $n$ is a power of an odd prime, then the minimal polynomial is always $\Phi_d(x)$. However, this is not guaranteed when $m > 1$ or when $n$ isn't a power of an odd prime, so one might then wonder if anything can be said about the behavior of $\theta_{n, m, A}$ when $\Phi_d(A) \neq 0 \mod n$. It turns out that while we can still describe the general shape of $\img(\theta_{n, m, A})$, we are no longer guaranteed the asymptotic filling out behavior. We discuss this more in Remark \ref{rmk:DGLGeneralizeRestriction} after the proof.
\end{remark}

An astute reader might have the following question after reading Theorem \ref{thm:DGLgeneralize}: How do we know that there are infinitely many $n$ such that $\GL_m(\Z/n\Z)$ contains a matrix $A$ of order $d$ such that $\Phi_d(A) = 0 \mod n$? This is a natural question, as the statement on the asymptotic behavior of $\theta_{n, m, A}$ doesn't make sense without the existence of such $n$. In the following remark, we provide a simple construction that shows there are infinitely many $n$ with the stated properties.

\begin{remark}
	\label{rmk:existenceofnfornewtheorem}
	Let $m$ and $d$ be given. Since there are infinitely many primes congruent to 1 mod $d$, then we can choose an infinite sequence of odd primes $p_1 < p_2 < p_3 < \ldots$ such that $p_i = 1 \mod d$ for every $i$. For each $i$, choose $\omega_i \in (\Z/p_i \Z)^\times$ such that the order of $\omega_i$ is $d$; such an $\omega_i$ exists because $(\Z/p_i \Z)^\times$ is cyclic of order $p_i - 1$, which is divisible by $d$. We can then choose the matrix $A_i \in \GL_m(\Z/p_i \Z)$ to be the matrix $\diag(\omega_i, \omega_i, \ldots, \omega_i).$ That is, $A_i$ comes from the diagonal embedding $$((\Z/p_i \Z)^\times)^m \into \GL_m(\Z/p_i \Z), \hspace{0.5in} \alpha \mapsto \diag(\alpha, \alpha, \ldots, \alpha).$$ Since $\Phi_d(\omega_i) = 0 \mod p_i$ for every $i$ (as mentioned in Remark \ref{rmk:DGLgeneralizeexplain}), then $\Phi_d(A_i) = 0 \mod p_i$ for every $i$. 
\end{remark}

There are, of course, more interesting examples of moduli $n$ and matrices $A$ such that $\Phi_d(A) = 0 \mod n$. The example above is used simply to show the existence of infinitely many $n$ with the desired properties.

We now work toward the proof of Theorem \ref{thm:DGLgeneralize}, which has two main parts. First, we must show that $\img(\theta_{n, m, A})$ is contained in the $\img(g_d)$, which will come directly from the fact that $\Phi_d(A) = 0 \mod n$. And second, we must show that the $\img(g_d)$ is filled out asymptotically as described. In order to describe the asymptotic behavior of these cyclic supercharacters, we need the following concepts about equidistribution. We note that these concepts are those used by Katz in \cite{Katz+2012}, though we borrow the notational conventions used by Kowalski in \cite{kowalskiequidistribution}.

\begin{definition}
	Let $Y$ a nonempty finite set. Let $s$ be a positive integer, and suppose we have a map $$\Lambda: Y \to [0,1)^s, \hspace{0.5in} \bold{y} \mapsto \Lambda(\bold{y}) := (\Lambda_1(y_1), \Lambda_2(y_2), \ldots, \Lambda_s(y_s)).$$ The \emph{discrepancy} of $(Y, \Lambda)$ is defined to be $$\sup_{B} \Bigg | \frac{ \# \{\bold{y} \in Y \st \Lambda(\bold{y}) \in B\}}{\# Y} - \vol(B) \Bigg |,$$ where the supremum is taken over all boxes $B = [a_1, b_1) \times \cdots \times [a_s, b_s) \subset [0, 1)^s$ and $\vol(B)$ denotes the volume of $B$. We say that a sequence $(Y_t, \Lambda_t)_{t = 1}^\infty$ of nonempty finite sets $Y_t$ and maps $\Lambda_t: Y_t \to [0,1)^s$ is \emph{uniformly distributed mod 1} if the discrepancy of $\Lambda_t$ goes to zero as $t \to \infty$. 
\end{definition}

\begin{remark}
	\label{rmk:correctionofdefn}
	We should note here that this definition differs from the definition given in \cite{DGL, Menagerie} (and several other places in the literature). In particular, the discrepancy in those papers is defined using the proportion $$\frac{\#(\Lambda(Y) \cap B)}{\#\Lambda(Y)}$$ rather than the proportion $$\frac{ \# \{\bold{y} \in Y \st \Lambda(\bold{y}) \in B\}}{\# Y}.$$ That is, they determine the proportion of distinct elements of the image which are in $B$. The correct definition of discrepancy, however, is interested in the number of \emph{indices} $\bold{y}$ whose image is in $B$. These two definitions are not equivalent, as $\Lambda$ is not necessarily injective.\footnote{An anonymous referee pointed out this definition error in a previous draft, and the author would like to thank the referee for this correction.}
\end{remark}

\begin{remark}
	\label{rmk:unifdistimpliesdensityandequidist}
	If $(Y_t, \Lambda_t)_{t = 1}^\infty$ is a sequence which is uniformly distributed mod 1, then it is an easy consequence that $\Lambda_t(Y_t)$ become dense in $[0,1)^s$. In fact, it implies the stronger statement that $\Lambda_t(Y_t)$ becomes \emph{equidistributed} in $[0,1)^s$---i.e. that its points become ``evenly spaced.'' While we care primarily about the density result for the purposes of proving Theorem \ref{thm:DGLgeneralize}, the result on equidistribution is interesting in its own right, and we discuss this briefly in Remark \ref{rmk:equidistributionimplication} after the proof.
\end{remark}

We now offer Weyl's criterion (stated as Lemma 1 in \cite{Menagerie}) for determining if a sequence is uniformly distributed mod 1. This will be a critical lemma in our proof of Theorem \ref{thm:DGLgeneralize}.

\begin{lemma}[Weyl's Criterion]
	\label{lem:Weyl}
	A sequence $(Y_t, \Lambda_t)_{t = 1}^\infty$ of nonempty finite sets $Y_t$ and maps $\Lambda_t: Y_t \to [0,1)^s$ is uniformly distributed mod 1 if and only if for every nonzero $\bold{v} \in \Z^s$, we have $$\lim_{t \to \infty} \frac{1}{\# Y_t} \sum_{\bold{y} \in Y_t} e(\Lambda_t(\bold{y}) \cdot \bold{v}) = 0,$$ where $\Lambda_t(\bold{y}) \cdot \bold{v}$ denotes the usual dot product.
\end{lemma}

We are now ready to prove Theorem \ref{thm:DGLgeneralize}.

\begin{proof}
	First, we show that $\img(\theta_{n, m, A}) \subset \img(g_d)$. Let $n$, $d$, and $A$ have the properties as stated in the theorem. Since we assume that $\Phi_d(A) = 0 \mod n$, then for $k \in \{0, 1, \ldots, d - 1\}$, we obtain the relations $$A^k \equiv \sum_{j = 0}^{\varphi(d) - 1} c_{jk} A^j \mod n,$$ where the $c_{jk}$ are the constants mentioned in the statement of the theorem. Then for any $\bold{x} \in (\Z/n\Z)^m$, we now have the following: $$\theta_{n, m, A}(\bold{x}) = \sum_{k = 0}^{d - 1} e\left( \dfrac{A^k \bullet \bold{x}}{n} \right) = \sum_{k = 0}^{d - 1} e\left( \dfrac{\sum_{j = 0}^{\varphi(d) - 1} c_{jk} A^j \bullet \bold{x}}{n} \right) = \sum_{k = 0}^{d - 1} \prod_{j = 0}^{\varphi(d) - 1} e\left( \dfrac{A^j \bullet \bold{x}}{n} \right)^{c_{jk}}.$$ Since $e\left( \frac{A^j \bullet \bold{x}}{n} \right)$ is contained in $\T$ for all $j$, then the fact that $\img(\theta_{n, m, A}) \subset \img(g_d)$ follows from the definition of $g_d$.
	
	We now need to show that $\img(\theta_{n, m, A})$ fills out $\img(g_d)$ as $n \to \infty$. To do this, we need to prove that the roots of unity showing up in the supercharacter sums get asymptotically close to any element $(z_1, \ldots, z_{\varphi(d)})$ in the domain of $g_d$. We show this by using Weyl's criterion to prove that the exponents of those roots of unity are uniformly distributed mod 1.
	
	We start by indexing our sets $Y_n$ and maps $\Lambda_n$ in the following way. Let $\cat{N} \subset \N$ be the set of all positive integers $n$ such that $\GL_m(\Z/n\Z)$ contains a matrix $A$ of order $d$ such that $\Phi_d(A) = 0 \mod n$. Create a sequence $(n_i)_{i = 1}^\infty$ using all the elements of $\cat{N}$, indexed so that $n_1 < n_2 < n_3 < \cdots$. Additionally, let $A_i$ denote our choice of matrix of order $d$ modulo $n_i$. We let $Y_{n_i} := (\Z/n_i \Z)^m$. Then for any $\bold{y} \in Y_{n_i}$, we define the following map: $$\Lambda_{n_i}(\bold{y}) = \left( \frac{(A_i^0 \bullet \bold{y}) \mod n_i}{n_i}, \frac{(A_i^1 \bullet \bold{y}) \mod n_i}{n_i}, \ldots, \frac{(A_i^{\varphi(d) - 1} \bullet \bold{y}) \mod n_i}{n_i} \right) \in [0, 1)^{\varphi(d)}.$$
	
	We need to show that $(Y_{n_i}, \Lambda_{n_i})_{i = 1}^\infty$ is uniformly distributed mod 1. Thus, using Lemma \ref{lem:Weyl}, we need to show that for any nonzero vector $\bold{v} \in \Z^{\varphi(d)}$, the following is true: $$\lim_{i \to \infty} \frac{1}{\# Y_{n_i}} \sum_{\bold{y} \in Y_{n_i}} e(\Lambda_{n_i}(\bold{y}) \cdot \bold{v}) = 0.$$ Note that $\# Y_{n_i} = n_i^m$ for every $i$, so we need to show that $$\lim_{i \to \infty} \frac{1}{n_i^m} \sum_{\bold{y} \in Y_{n_i}} e(\Lambda_{n_i}(\bold{y}) \cdot \bold{v}) = 0.$$ 
	
	To this end, first let us consider the vectors $A_i^j \bold{1}$ more closely. For $j \in \{0, 1, \ldots, \varphi(d) - 1\}$, we write $A_i^j =: (a_{bc, i}^j)_{1 \leq b,c \leq m}$. Then we have the following: $$A_i^j \bold{1} = A^j_i \cdot \begin{pmatrix} 1 & 1 & \cdots & 1 \end{pmatrix}^T = \begin{pmatrix} \sum_{c = 1}^m a_{1c, i}^j & \sum_{c = 1}^m a_{2c, i}^j & \cdots & \sum_{c = 1}^m a_{cm, i}^j \end{pmatrix}^T.$$ Let us write $A_i^j \bold{1} = (w^j_{1,i}, w^j_{2,i}, \ldots, w^j_{m,i})^T$; that is, $w^j_{k,i}$ is the sum of the elements in the $k$-th row of the matrix $A_i^j$. To simplify notation, we set $\bold{w}_i^j := A_i^j \bold{1} = (w^j_{1,i}, w^j_{2,i}, \ldots, w^j_{m,i})^T$.
	
	Returning to the computation at hand, let $\bold{v} = (v_0, \ldots, v_{\varphi(d) - 1})$ be any nonzero vector in $\Z^{\varphi(d)}$. Then we have the following:
	\begin{flalign*}
		\sum_{\bold{y} \in Y_{n_i}} e(\Lambda_{n_i}(\bold{y}) \cdot \bold{v}) &= \sum_{\bold{x} \in (\Z/n_i\Z)^m}  e\left( \sum_{j = 0}^{\varphi(d) - 1} \dfrac{\bold{w}^j_i \cdot \bold{x} \cdot v_j}{n_i} \right) \\
		&= \sum_{x_1, \ldots, x_m = 0}^{n_i - 1} e\left( \sum_{j = 0}^{\varphi(d) - 1} \dfrac{ (w_{1,i}^j x_1 + w_{2,i}^j x_2 + \cdots + w_{m,i}^j x_m) v_j}{n_i} \right) \\
		&= \left[ \sum_{x_1 = 0}^{n_i - 1} e\left( \sum_{j = 0}^{\varphi(d) - 1} \dfrac{w_{1,i}^j x_1 v_j}{n_i} \right) \right] \cdots \left[ \sum_{x_m = 0}^{n_i - 1} e\left( \sum_{j = 0}^{\varphi(d) - 1} \dfrac{w_{m,i}^j x_m v_j}{n_i} \right) \right].
	\end{flalign*}
	For $\ell \in \{1, \ldots, m\}$, we define $\alpha_{\ell, i} := \sum_{j = 0}^{\varphi(d) - 1} w_{\ell,i}^j \cdot v_j$. Then we have the following: 
	\begin{flalign*}
		\sum_{\bold{y} \in Y_{n_i}} e(\Lambda_{n_i}(\bold{y}) \cdot \bold{v}) &= \left[ \sum_{x_1 = 0}^{n_i - 1} e\left( \sum_{j = 0}^{\varphi(d) - 1} \dfrac{w_{1,i}^j x_1 v_j}{n_i} \right) \right] \cdots \left[ \sum_{x_m = 0}^{n_i - 1} e\left( \sum_{j = 0}^{\varphi(d) - 1} \dfrac{w_{m,i}^j x_m v_j}{n_i} \right) \right] \\
		&= \left[ \sum_{x_1 = 0}^{n_i - 1} e\left(\frac{\alpha_{1,i} x_1}{n_i} \right) \right] \cdots \left[ \sum_{x_m = 0}^{n_i - 1} e\left( \dfrac{\alpha_{m,i} x_m}{n_i} \right) \right] \\
		&= \begin{cases} n_i^m & \text{if } n_i \mid \alpha_{\ell,i} \text{ for all } \ell \in \{1, 2, \ldots, m\}, \\ 0 & \text{otherwise.} \end{cases}
	\end{flalign*}
	The final equality comes from the orthogonality of additive characters on $\Z/n\Z$. That is, $$\sum_{x = 0}^{n - 1} e\left(\dfrac{\alpha x}{n}\right) = \begin{cases} n & \text{if } n \mid \alpha, \\ 0 & \text{otherwise}. \end{cases}$$
	
	From here, we will show that there are at most finitely many $i$ such that $n_i$ divides $\alpha_{\ell, i}$ for all $\ell \in \{1, 2, \ldots, m\}$. First, define a polynomial $f_{\bold{v}}(x) = v_0 + v_1 x + \cdots + v_{\varphi(d) - 1} x^{\varphi(d) - 1}$. Using the definition of $\alpha_{\ell, i}$, it is a simple calculation to check that the following equality holds: $$f_{\bold{v}}(A_i) \begin{pmatrix} 1 & 1 & \cdots & 1 \end{pmatrix}^T = \begin{pmatrix} \alpha_{1, i} & \alpha_{2, i} & \cdots & \alpha_{m, i} \end{pmatrix}^T.$$ 
	
	Consider now the cyclotomic polynomial $\Phi_d(x)$. Since $f_{\bold{v}}(x)$ is a polynomial of degree $\varphi(d) - 1$, and since $\Phi_d(x)$ is irreducible over $\Q$ of degree $\varphi(d)$, then $\gcd(f_{\bold{v}}(x), \Phi_d(x)) = 1$ in $\Q[x]$. Thus there exist polynomials $P(x), Q(x) \in \Q[x]$ such that $P(x) f_{\bold{v}}(x) + Q(x) \Phi_d(x) = 1$. By clearing out denominators, we then obtain polynomials $R(x), S(x) \in \Z[x]$ such that $R(x) f_{\bold{v}}(x) + S(x) \Phi_d(x) = t$ for some fixed nonzero integer $t$. 
	
	Note that the above equality depends only on the choice of $d$ and the vector $\bold{v} \in \Z^{\varphi(d)}$, which we fixed at the beginning of the proof. In particular, the equality does not depend on the choices of $n_i$ nor $A_i$. Thus for any of our chosen $A_i$, we find that $R(A_i) f_{\bold{v}}(A_i) + S(A_i) \Phi_d(A_i) = t \cdot I_m$, where $I_m$ is the $m \times m$ identity matrix. Since $\Phi_d(A_i) = 0 \mod n_i$, we then find that $R(A_i) f_{\bold{v}}(A_i) \equiv t \cdot I_m \mod n_i$. After multiplying both sides of this congruence by the matrix $\begin{pmatrix} 1 & 1 & \cdots & 1 \end{pmatrix}^T$, we get the following result: $$R(A_i) f_{\bold{v}}(A_i) \begin{pmatrix} 1 & \cdots & 1 \end{pmatrix}^T \equiv \begin{pmatrix} t & \cdots & t \end{pmatrix}^T \mod n_i.$$
	
	Now, $R(A_i)$ is an $m \times m$ matrix, so write $R(A_i) = (\rho_{bc, i})_{1 \leq b,c \leq m}$. Using the fact that $$f_{\bold{v}}(A_i) \begin{pmatrix} 1 & 1 & \cdots & 1 \end{pmatrix}^T = \begin{pmatrix} \alpha_{1, i} & \alpha_{2, i} & \cdots & \alpha_{m, i} \end{pmatrix}^T,$$ we can rewrite $R(A_i) f_{\bold{v}}(A_i) \begin{pmatrix} 1 & \cdots & 1 \end{pmatrix}^T$ as the following: $$\begin{pmatrix} \rho_{11, i} & \cdots & \rho_{1m, i} \\ \vdots & \ddots & \vdots \\ \rho_{m1, i} & \cdots & \rho_{mm, i} \end{pmatrix} \begin{pmatrix} \alpha_{1,i} & \cdots & \alpha_{m,i} \end{pmatrix}^T = \begin{pmatrix} \displaystyle{\sum_{\ell = 1}^m \rho_{1 \ell, i} \alpha_{\ell, i} } & \displaystyle{\sum_{\ell = 1}^m \rho_{2 \ell, i} \alpha_{\ell, i}} & \cdots & \displaystyle{\sum_{\ell = 1}^m \rho_{m \ell, i} \alpha_{\ell, i}}\end{pmatrix}^T.$$ Putting this all together, we end up with the following equivalence relation: $$\begin{pmatrix} \displaystyle{\sum_{\ell = 1}^m \rho_{1 \ell, i} \alpha_{\ell, i} } & \displaystyle{\sum_{\ell = 1}^m \rho_{2 \ell, i} \alpha_{\ell, i}} & \cdots & \displaystyle{\sum_{\ell = 1}^m \rho_{m \ell, i} \alpha_{\ell, i}}\end{pmatrix}^T \equiv \begin{pmatrix} t & \cdots & t \end{pmatrix}^T \mod n_i.$$ 
	
	If we now assume that $n_i$ divides $\alpha_{\ell, i}$ for every $\ell$, then $\sum_{\ell = 1}^m \rho_{b \ell, i} \alpha_{\ell, i}$ is divisible by $n_i$ for every $b \in \{1, \ldots, m\}$. Since $\sum_{\ell = 1}^m \rho_{b \ell, i} \alpha_{\ell, i} \equiv t \mod n_i$ by the equivalence relation above, then $n_i$ must also divide $t$. However, as stated previously, $t$ is a nonzero integer which is fixed for all choices of $n_i$ and $A_i$. Since there are at most finitely many $i$ such that $n_i$ divides $t$, then there can be at most finitely many $i$ such that $n_i$ divides $\alpha_{\ell, i}$ for all $\ell$.
	
	We can then choose $N$ to be the largest integer such that $n_N$ divides $\alpha_{\ell, N}$ for all $\ell$, letting $N = 0$ if there is no such integer. Then the previous paragraph implies that $$\sum_{\bold{y} \in Y_{n_i}} e(\Lambda_{n_i}(\bold{y}) \cdot \bold{v}) = 0$$ for every $i > N$. Hence, we have shown that $$\lim_{i \to \infty} \frac{1}{\# Y_{n_i} } \sum_{\bold{y} \in Y_{n_i}} e(\Lambda_{n_i}(\bold{y}) \cdot \bold{v}) = 0.$$ Thus the sequence $(Y_{n_i}, \Lambda_{n_i})_{i = 1}^\infty$ is uniformly distributed mod 1 by Weyl's criterion, finishing our proof of Theorem \ref{thm:DGLgeneralize}.
\end{proof}

\begin{remark}
	\label{rmk:equidistributionimplication}
	As mentioned in Remark \ref{rmk:unifdistimpliesdensityandequidist}, we showed not only that $\Lambda_{n_i}(Y_{n_i})$ becomes dense in $[0,1)^{\varphi(d)}$, but that it becomes equidistributed as well. More precisely, Theorem \ref{thm:DGLgeneralize} shows that the sequence $(Y_{n_i}, \theta_{n_i, m, A_i})_{i = 1}^\infty$ becomes equidistributed in $\img(g_d)$ with respect to the pushforward measure $(g_d)_{*}\lambda$ of the Haar measure $\lambda$ on $\T^{\varphi(d)}$. These concepts and more are explored in much more depth in the aforementioned papers by Untrau and Kowalski--Untrau \cite{untrau2021equidistribution, untraukowalski2023ultrashort}.
\end{remark}

We now return to our discussion in Remark \ref{rmk:DGLgenpreamble}. Let $n$, $m$, and $d$ be chosen as in Theorem \ref{thm:DGLgeneralize}. Let $A \in \GL_m(\Z/n\Z)$ have order $d$, but this time assume that $\Phi_d(A) \neq 0 \mod n$. In Remark \ref{rmk:DGLGeneralizeRestriction} below, we discuss the implications of these assumptions, and we show by counterexample that (an appropriate generalization of) Theorem \ref{thm:DGLgeneralize} is not true in this situation.

\begin{remark}
	First, note that $A^d - I = 0 \mod n$ since $A$ has order $d$ in $\GL_m(\Z/n\Z)$. Also, recall that $x^d - 1 \in \Z[x]$ decomposes as $$x^d - 1 = \prod_{k \mid d} \Phi_k(x).$$ Thus the divisors of $x^d - 1$ are of the form $\Phi_{k_1} \Phi_{k_2} \cdots \Phi_{k_\ell}$ for $k_i$ dividing $d$. Let $R \subset \Z[x]$ be the set of polynomials $f(x)$ dividing $x^d - 1$ such that $f(A) = 0 \mod n$. At the very least, we know that $x^d - 1 \in R$, though it may not be the polynomial in $R$ with the smallest degree. Let $\mu(x) \in R$ be the monic polynomial with minimal degree among the polynomials in $R$. If $\mu(x)$ is irreducible over $\Z$, then $\mu = \Phi_d$ (since $A$ has order $d$). However, since we are exploring the case where $\Phi_d(A) \neq 0 \mod n$, then we assume that $\mu$ is reducible over $\Z$.
	
	It turns out that Theorem \ref{thm:DGLgeneralize} does not generalize to situations where $\mu$ is reducible. At least, the asymptotic filling out behavior no longer holds---however, we can still describe the general shape of the supercharacter values. For example, given $\mu(x) \mid (x^d - 1)$, we can define an analogous Laurent polynomial $g_\mu: \T^{\deg(\mu)} \to \C$ given by the following: $$g_\mu(z_1, \ldots, z_{\deg(\mu)}) = \sum_{k = 0}^{d - 1} \prod_{j = 0}^{\deg(\mu) - 1} z_{j+1}^{b_{jk}},$$ where the $b_{jk}$ are given by the relations $$x^k \equiv \sum_{j = 0}^{\deg(\mu) - 1} b_{jk} x^j \mod \mu(x).$$ Then, using the same reasoning as in the proof of Theorem \ref{thm:DGLgeneralize}, we see that $\img(\theta_{n, m, A}) \subset \img(g_\mu)$. However, it is no longer necessarily true that the $\img(g_\mu)$ is filled out as $n \to \infty$. In fact, we provide a counterexample to this in the following paragraphs.
	
	Recall that Weyl's criterion gives an equivalent condition for $\img(g_\mu)$ to be filled out asymptotically. This criterion says that $\img(g_\mu)$ is filled out asymptotically if for any nonzero $\bold{v} \in \Z^{\deg(\mu)}$, we have $$\lim_{i \to \infty} \frac{1}{\# Y_{n_i}} \sum_{\bold{y} \in Y_{n_i}} e(\Lambda_{n_i}(\bold{y}) \cdot \bold{v}) = 0.$$ In the setting of Theorem \ref{thm:DGLgeneralize}, we additionally showed in our proof that Weyl's criterion holds if and only if $f_{\bold{v}}(A_i) \bold{1} = 0 \mod n_i$ for at most finitely many $i$. Thus finding a counterexample is equivalent to finding $n_i$, $m$, $d$, and $A_i$ such that for some nonzero $\bold{v} \in \Z^{\deg(\mu)}$, we have $f_{\bold{v}}(A_i) \bold{1} = 0 \mod n_i$ for infinitely many $i$. 
	
	To this end, consider the case where $\mu(x) = \Phi_3(x) \Phi_5(x)$; where $m = 6$ and $d = 15$; where $n_i$ are the integers greater than 6 which are not divisible by 2, 3, or 5; and where $A_i$ is fixed for all $i$ to be the companion matrix to $\Phi_3 \Phi_5$. Since $m = 6$ and $n_i \geq 7$, then it is guaranteed that $15 \mid (\# \GL_m(\Z/n_i\Z))$; to see this, consult the first bullet point in Section \ref{sec:strategies} and note that $3 \mid (p^2 - 1)$ and $5 \mid (p^4 - 1)$ for all primes $p \geq 7$. Additionally, if $A \in \GL_6(\Z)$ is the companion matrix to $\Phi_3 \Phi_5$, then $A$ has order 15. One can then verify computationally that for all $k \in \{0, 1, \ldots, 14\}$, the matrix $A^k$ has integer entries in the set $\{0, \pm 1, \pm 2, \pm 3, \pm 4\}.$ Thus, as long as $n_i$ is not divisible by 2, 3, or 5, the matrix $A$ has order 15 in the group $\GL_6(\Z/n_i\Z)$ (in particular, no smaller power of $A$ reduces to the identity). 
	
	We can then choose the vector $\bold{v} = (1, 1, 1, 1, 1, 0) \in \Z^{\deg(\mu)}$ so that $$f_{\bold{v}}(x) = 1 + x + x^2 + x^3 + x^4 = \Phi_5(x).$$ One can compute directly that $$f_{\bold{v}}(A) \bold{1} = \begin{pmatrix} 0 & 0 & 0 & 0 & 0 & 0 \end{pmatrix}^T,$$ where the multiplication is happening over $\Z$. Since $f_{\bold{v}}(A) \bold{1} = 0$ over $\Z$, then $f_{\bold{v}}(A) \bold{1} = 0 \mod n_i$ for all $n_i$. Thus $\img(g_\mu)$ does \emph{not} get filled out asymptotically. Thus, without any added assumptions, Theorem \ref{thm:DGLgeneralize} seems to be the most general version possible of the original DGL Theorem.
	
	\label{rmk:DGLGeneralizeRestriction}
\end{remark}

To conclude this section, we include some examples of the phenomenon described in Theorem \ref{thm:DGLgeneralize} in Figure \ref{fig:DGLgeneralize}. 

\begin{figure}[h!]
	\centering
	\begin{subfigure}[b]{0.33\linewidth}
		\includegraphics[width=\linewidth]{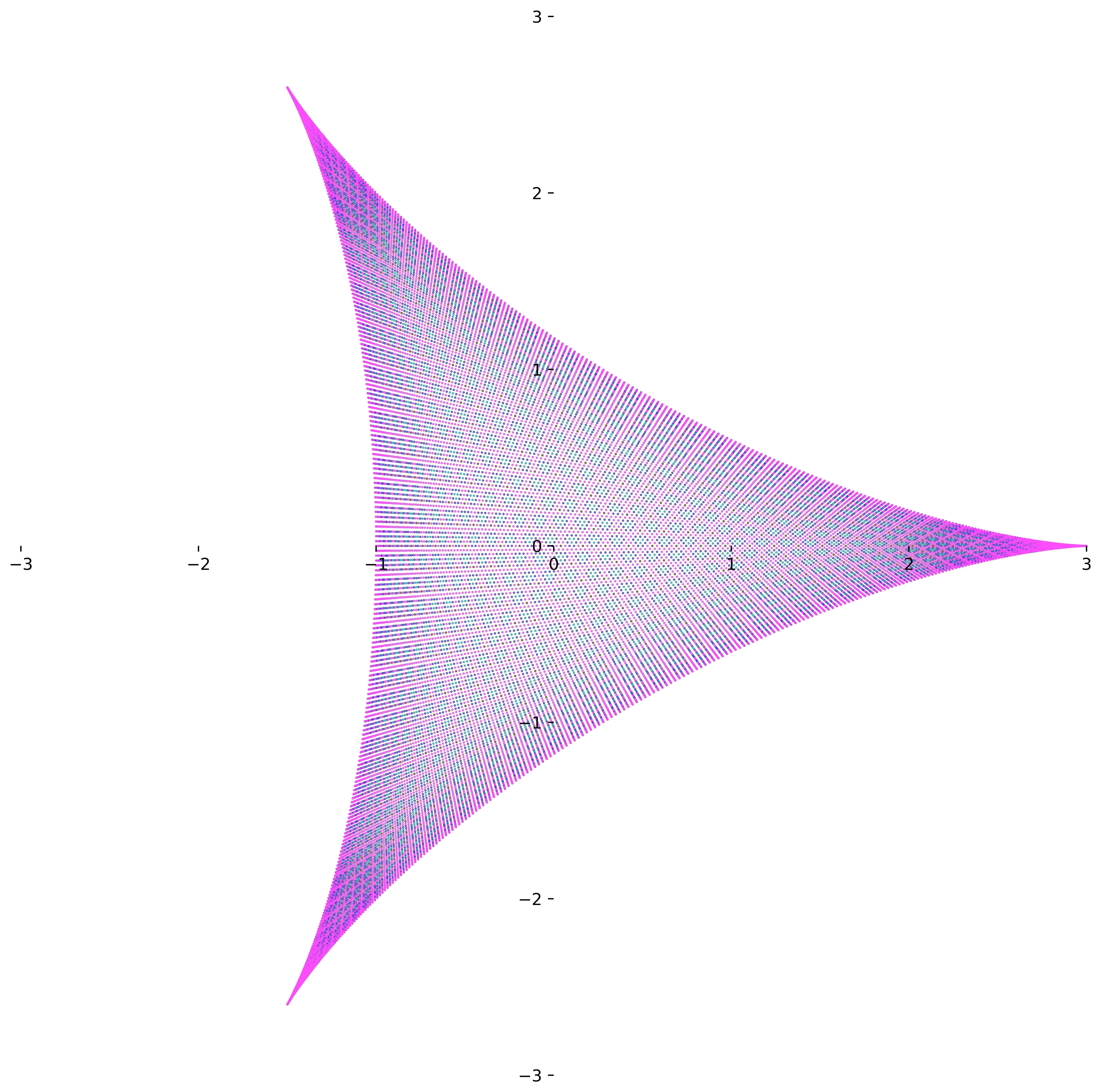}
		\caption{$A = \left( \begin{smallmatrix} 0 & 1 \\ 454 & 454 \end{smallmatrix} \right)$, $n = 455$, $d = 3$}
	\end{subfigure}
	\begin{subfigure}[b]{0.33\linewidth}
		\includegraphics[width=\linewidth]{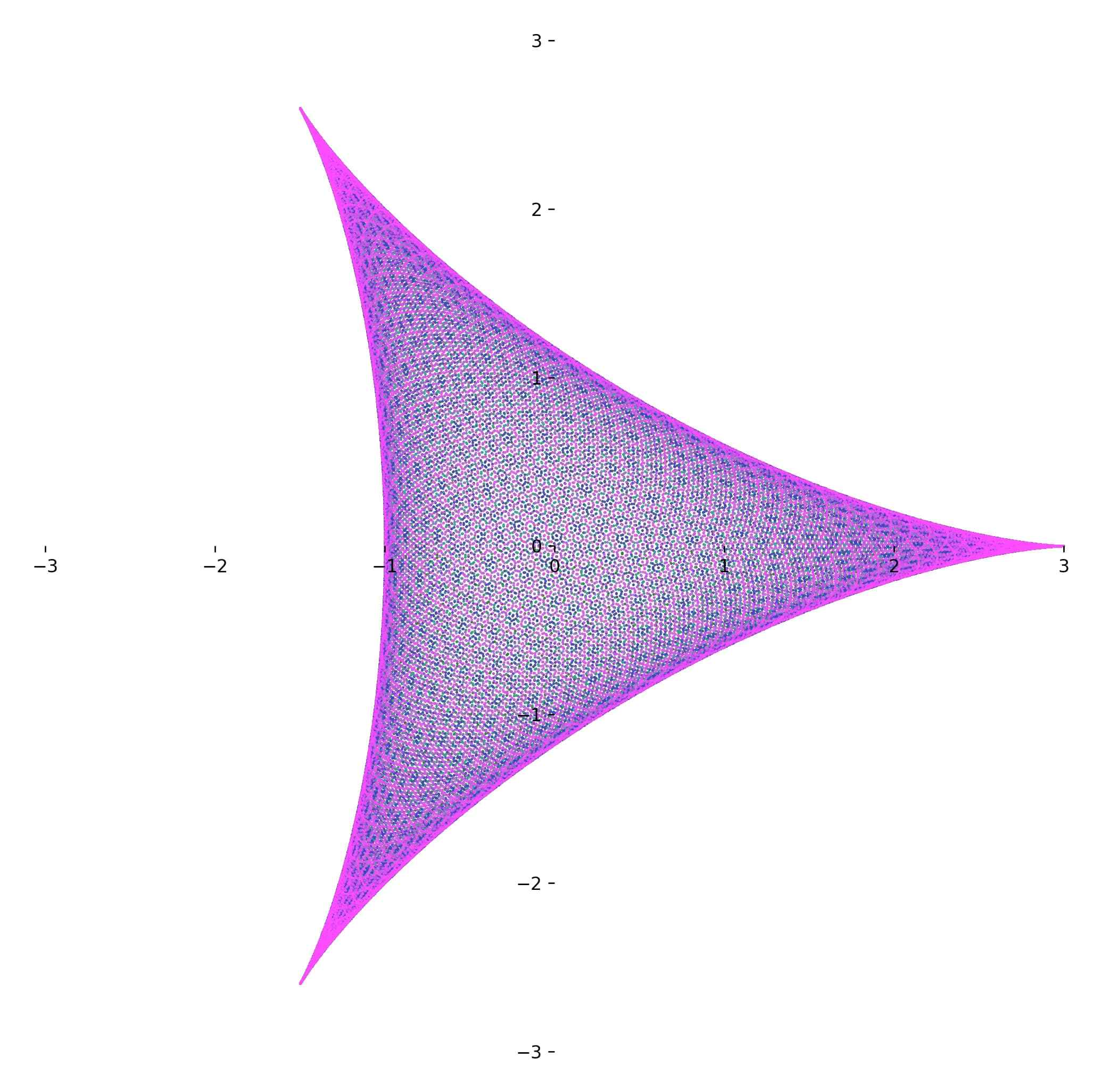}
		\caption{$A = 21823$, $n = 7^2 13^2 19$, $d = 3$}
	\end{subfigure}
	\begin{subfigure}[b]{0.32\linewidth}
		\includegraphics[width=\linewidth]{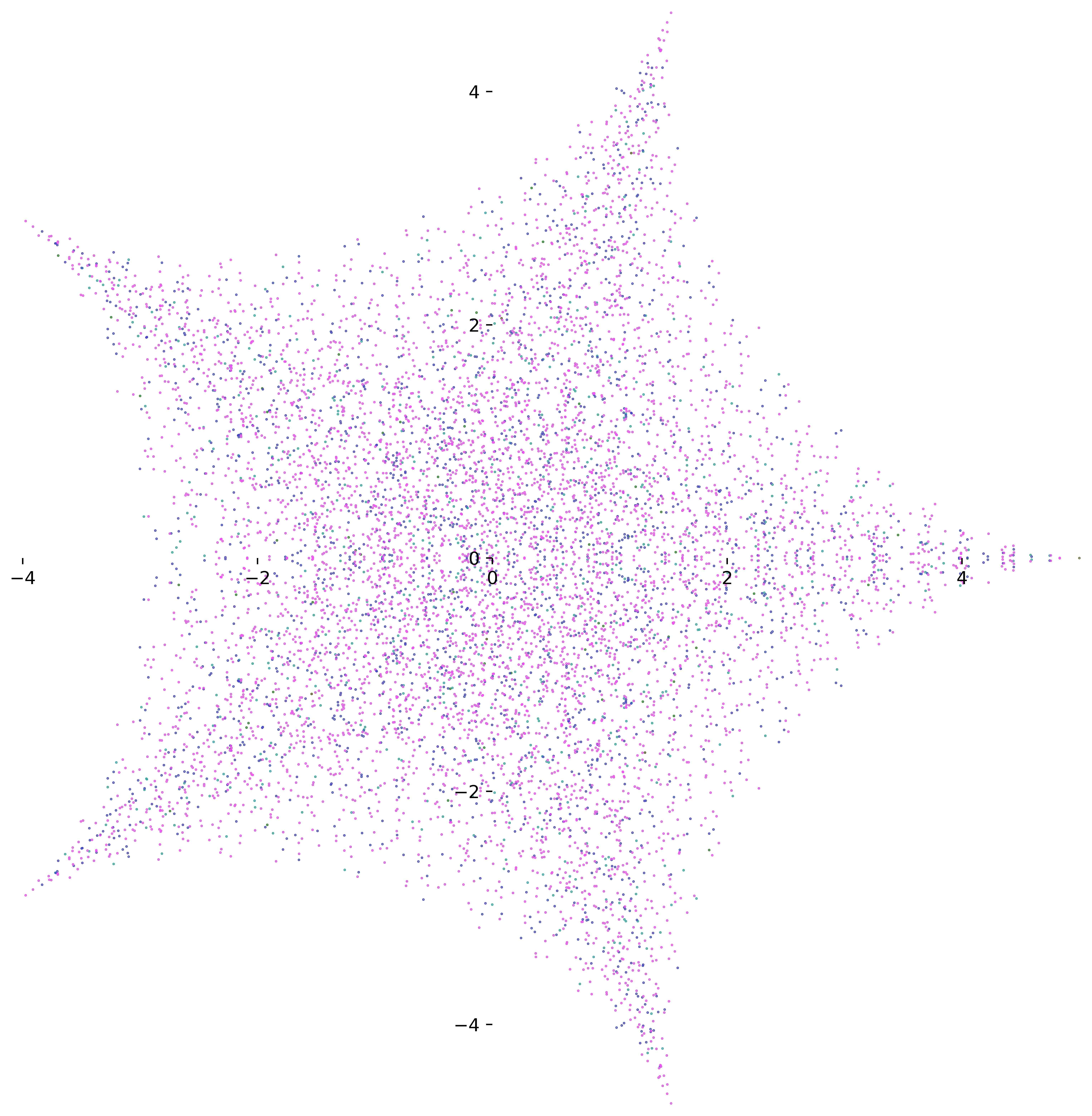}
		\caption{$A = \left( \begin{smallmatrix} 0 & 1 \\ 94 & 166 \end{smallmatrix} \right)$, $n = 209$, $d = 5$}
	\end{subfigure}
	
	\caption{Examples of Theorem \ref{thm:DGLgeneralize}}
	\label{fig:DGLgeneralize}
\end{figure}

\section{Observations and Generalizations: Class Field Theory Perspective}

\label{sec:classfield}
	
We now wish to step away from supercharacter theory and return to Gaussian periods, this time viewing them from the perspective of class field theory. As a reminder, we stated in section 2.2 that the ray class field for $\Q$ of modulus $(n) \subset \Z$ is the cyclotomic field $\Q(\mu_n)$, where $\mu_n \subset \C^\times$ is the subset of $n$-th roots of unity. The ray class group of modulus $(n)$ is isomorphic to $(\Z/n\Z)^\times$, and we can choose an element $\omega \in (\Z/n\Z)^\times$ to get a cyclic subgroup $\langle \omega \rangle \subset (\Z/n\Z)^\times$. From this perspective, Gaussian periods correspond to elements that generate the subfield $\Q(\mu_n)^{\langle \omega \rangle} \subset \Q(\mu_n)$ fixed by the action of this subgroup. 

Now, this is the whole story for the rational field $\Q$, but we would like to generalize this to other base fields. Ostensibly, this story is simple enough to recreate. Take a number field $K$, an ideal $\m \subset \O_K$, its ray class field $K[\m]$ of modulus $\m$, and its ray class group $\Cl_K(\m)$. Then choose an element $\omega \in \Cl_K(\m)$, and sum over the Galois action of $\langle \omega \rangle$ to generate elements of the subfield $K[\m]^{\langle \omega \rangle}$. However, in order to recreate this story \emph{explicitly}, one needs to have explicit elements of the ray class field $K[\m]$, which as stated previously is a problem to which we have very few answers. 

As of the writing of this paper, there are two main classes of fields other than $\Q$ for which we have explicit descriptions of their ray class fields. The first is quadratic imaginary fields (fields of the form $\Q(\sqrt{-D})$ for positive square-free $D$), where the theory of complex multiplication provides an answer. The second is the much more recent case of totally real fields, in which Dasgupta and Kakde showed in \cite{TotReal,dasgupta2023brumerstark} that the ray class fields can be generated using Brumer-Stark units. 

For the rest of this section, we will explore the case of quadratic imaginary fields, starting with a (very) brief overview of elliptic curves and the theory of complex multiplication. For further reading on elliptic curves and complex multiplication, we recommend \cite{SilvermanArithmetic} and \cite[Chapter II]{SilvermanAdvanced}.

\subsection{Elliptic Curves and Complex Multiplication}

\label{sec:classfieldellipticcurvesandCM}

Our goal is to describe the construction of ray class fields of quadratic imaginary fields, along with some useful related theorems. For this section, we assume basic knowledge about elliptic curves and start with the following definition.

\begin{definition}
	Let $E$ be an elliptic curve defined over $\C$. Then $E$ is said to have \emph{complex multiplication} (CM) if $\End(E)$ is isomorphic to a quadratic imaginary order $\O$. If $\O$ is contained in the quadratic imaginary field $K$, then we say that $E$ has CM by $K$.
\end{definition}

Elliptic curves with CM have many special properties. For this paper, however, we focus only on the relationship between CM and abelian extensions of quadratic imaginary fields.

\begin{proposition}[Proposition II.1.2 in \cite{SilvermanAdvanced}]
	Let $K = \Q(\sqrt{-D})$ be a quadratic imaginary field. Let $\cat{E}(K)$ denote the set of elliptic curves with CM by $K$, up to isomorphism. Then $\cat{E}(K)$ is finite, and there exists a bijection between $\cat{E}(K)$ and the ideal class group $\Cl_K(1)$.
\end{proposition}

Recall that the Hilbert class field $K[1]$ has degree $[K[1]: K] = \# \Cl_K(1)$ over $K$. We will be using the fact that there is a bijection between $\Cl_K(1)$ and $\cat{E}(K)$ to obtain generators of the Hilbert class field, but before we can do that, we must first describe a different characterization of elliptic curves over $\C$.

\begin{theorem}[Uniformization Theorem]
	Let $E$ be an elliptic curve over $\C$. Then there exists a $\Z$-lattice $\Lambda \subset \C$, unique up to homothety, such that $E$ is isomorphic to $\C/\Lambda$ via the complex analytic isomorphism $$\phi: \C/\Lambda \to E, \hspace{.5in} \phi(z) = (\wp(z; \Lambda), \wp'(z; \Lambda)),$$ where $\wp$ is the Weierstrass $\wp$-function defined to be the following: $$\wp(z; \Lambda) = \frac{1}{z^2} + \sum_{\lambda \in \Lambda \sm \{0\}} \left( \frac{1}{(z - \lambda)^2} - \frac{1}{\lambda^2} \right).$$
\end{theorem}

The Weierstrass $\wp$-function will be important to us later for computational purposes, but for now we wish to focus on the fact that choosing an elliptic curve (up to isomorphism) is equivalent to choosing a lattice $\Lambda \subset \C$ (up to homothety). We are now ready to describe how to obtain the generators of the Hilbert class fields of quadratic imaginary fields.

\begin{theorem}[Theorems II.4.1 and II.4.3 in \cite{SilvermanAdvanced}]
	Let $K$ be a quadratic imaginary field. Let $E_1$, \ldots, $E_\ell$ be representatives of all the isomorphism classes of elliptic curves with CM by $K$. Let $\Lambda_1$, \ldots, $\Lambda_\ell$ be the lattices in $\C$ such that $E_i \cong \C/\Lambda_i$, and write $\Lambda_i = \Z + \Z\tau_i$, where $\tau_i$ is in the upper half-plane. There exists a weight 0 modular function $j$ such that the Hilbert class field $K[1]$ is achieved by adjoining the values $j(\tau_1)$, \ldots, $j(\tau_\ell)$. In fact, the $j(\tau_i)$ are all algebraic conjugates, so $K[1] = K(j(\tau_i))$ for any choice of $i \in \{1, \ldots, \ell\}$. 
\end{theorem}

\begin{remark}
	It is common to abuse notation for the function $j$ by allowing elliptic curves as inputs rather than elements of the upper half-plane. That is, if $E$ is an elliptic curve which is isomorphic to $\C/\Lambda$ for some lattice $\Lambda \subset \C$, then we understand $j(E)$ to mean $j(\tau)$, where $\Lambda = \Z + \Z\tau$ with $\tau$ in the upper half-plane. In this way, we can write the result of the theorem above as saying that $K[1] = K(j(E))$, where $E$ is any elliptic curve with CM by $K$.
\end{remark}

We are now able to construct the Hilbert class fields of quadratic imaginary fields. However, to construct abelian extensions that allow ramification at certain primes, we require the following definitions.

\begin{definition}
	Let $E$ be an elliptic curve with CM by a quadratic imaginary field $K$. Let $\m \subset \O_K$ be an ideal. We define the \emph{$\m$-torsion subgroup of $E$} to be $$E[\m] = \{t \in E \st [\alpha] t = 0 \text{ for every } \alpha \in \m\},$$ where $0$ represents the identity element of $E$ and $[\alpha]$ represents the normalized action of $\alpha$ as an element of $\End(E)$. In the special case that $\m = (m)$ for $m \in \Z$, then $E[m] \cong \Z/m\Z \times \Z/m\Z$ due to the fact that $E \cong \C/\Lambda$.
	
	\label{def:torsionsbgp}
\end{definition}

\begin{definition}
	Let $E$ be an elliptic curve with CM by a quadratic imaginary field $K$. Suppose $E$ is defined by the equation $y^2 = x^3 + Ax + B$. Then a \emph{Weber function} is a finite map $h: E \to E/\Aut(E)$. For our purposes, we follow the convention of \cite[II.5.5.1]{SilvermanAdvanced} and use the following Weber function: $$h(x,y) = \begin{cases} x & AB \neq 0, \\ x^2 & B = 0, \\ x^3 & A = 0. \end{cases}$$
\end{definition}

\begin{remark}
	The two special cases of $A = 0$ and $B = 0$ correspond to the cases where $E$ has CM by $\Q(\sqrt{-3})$ and $\Q(\sqrt{-1})$, respectively. These are the only quadratic imaginary fields which contain roots of unity other than $-1$ and $1$ (i.e. where $\Aut(E)$ is strictly larger than $\{\pm 1\}$), and it is for this reason that their Weber functions are different. However, in most cases, the Weber function is simply $x$-coordinate projection.
\end{remark}

We now use all of this to construct the ray class fields of quadratic imaginary fields.

\begin{theorem}[Theorem II.5.6 of \cite{SilvermanAdvanced}]
	Let $K$ be a quadratic imaginary field with ring of integers $\O_K$, let $\m \subset \O_K$ be an ideal, let $E$ be an elliptic curve with CM by $K$, and let $h: E \to E/\Aut(E)$ be a Weber function. Then the field $$K(j(E), h(E[\m]))$$ is the ray class field of $K$ of modulus $\m$. 
\end{theorem}

\begin{remark}
	\label{rmk:analogofrootsofunity}
	The result of the above theorem can be stated in the following way. Start with a quadratic imaginary field $K$, a modulus $\m \subset \O_K$, and an elliptic curve $E$ with CM by $K$. In order to obtain the ray class field $K[\m]$, one must first adjoin $j(E)$ to $K$, followed by adjoining the $x$-coordinates (or the squares of $x$-coordinates for $K = \Q(\sqrt{-1})$, the cubes for $K = \Q(\sqrt{-3})$) of the $\m$-torsion points of $E$. Thus the analogy of roots of unity for $\Q$ are certain $j$-values and coordinates of elliptic curve torsion points for quadratic imaginary $K$.
\end{remark}

To continue the last point of Remark \ref{rmk:analogofrootsofunity}, we would like to note that the roots of unity are quite simple to describe geometrically, as they are simply points lying on the unit circle in $\C$. However, the $j$-values and $\m$-torsion points of elliptic curves are more complicated to describe geometrically. 

In fact, to the author's knowledge, it seems that there hasn't been much study of elliptic curve torsion points from a graphical perspective. One reason for this might be that elliptic curves inherently live in a four-dimensional $\R$-vector space, which are notoriously difficult to represent graphically (to say the least). However, our study will focus almost exclusively on the $x$-coordinates, so we thought it might be of interest to generate images of the $x$- and $y$-coordinates of elliptic curve torsion points in the complex plane. 

To this end, we offer examples of these images in Figure \ref{fig:TorsionPts}. The images (A) through (F) have no special properties other than their coloring, which is simply the coloring that Python automatically applies to scatter plots (it cycles through a list of colors). The images (G) through (I) take inspiration from \cite{Starscapes}, in which we size the torsion points inversely based on their additive order in the torsion group; that is, the smaller the order of the torsion point, the larger the dot. We discuss our computational methods for these images and others (including how we chose the elliptic curves with CM) after the proof of Proposition \ref{prop:GalgroupofrayclassfieldoverHilbert} in the next section. 

\begin{figure}[h!]
	\centering
	\begin{subfigure}[b]{0.32\linewidth}
		\includegraphics[width=\linewidth]{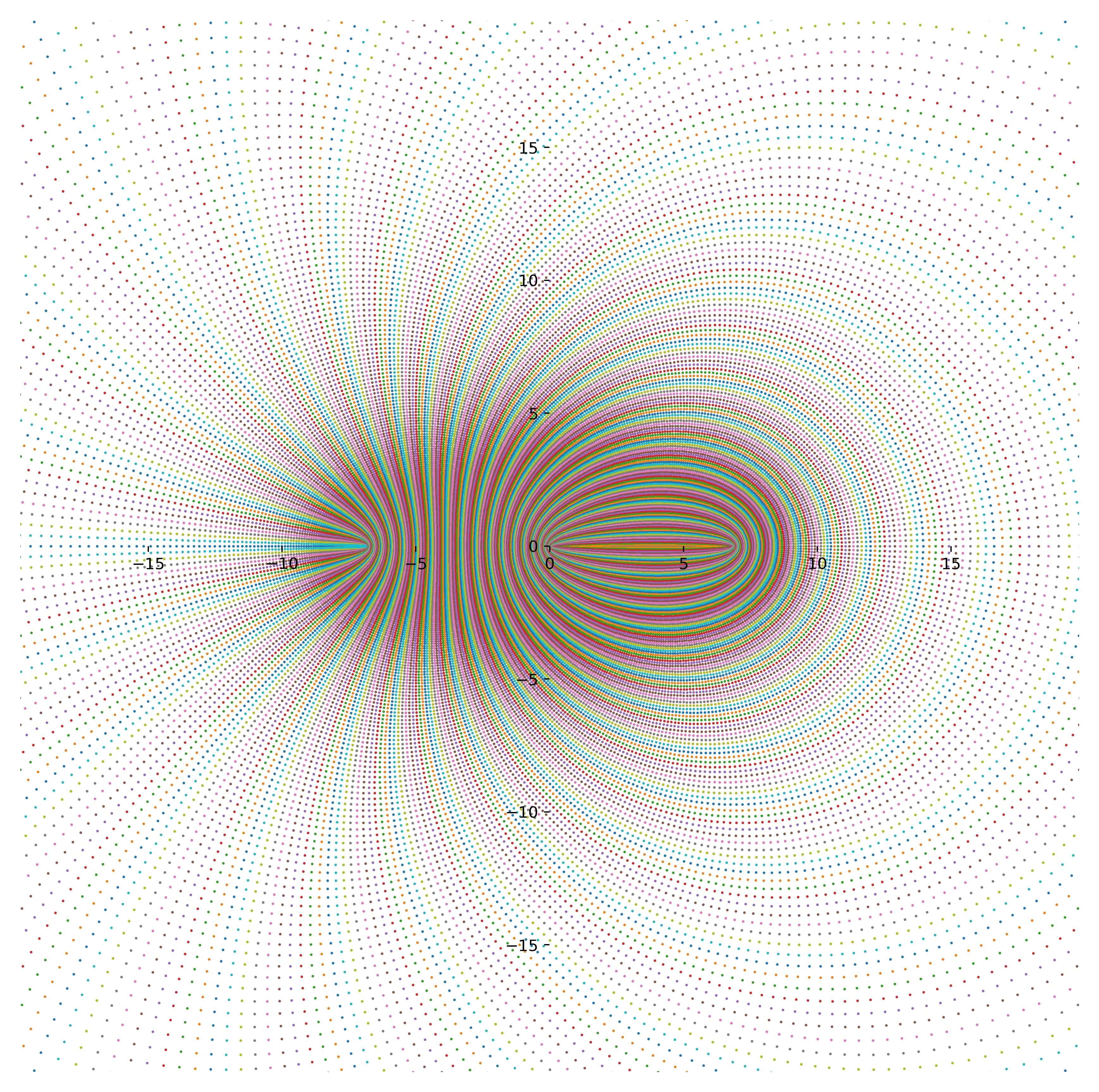}
		\caption{$x$-coordinates, $K = \Q(\sqrt{-1})$}
	\end{subfigure}
	\begin{subfigure}[b]{0.32\linewidth}
		\includegraphics[width=\linewidth]{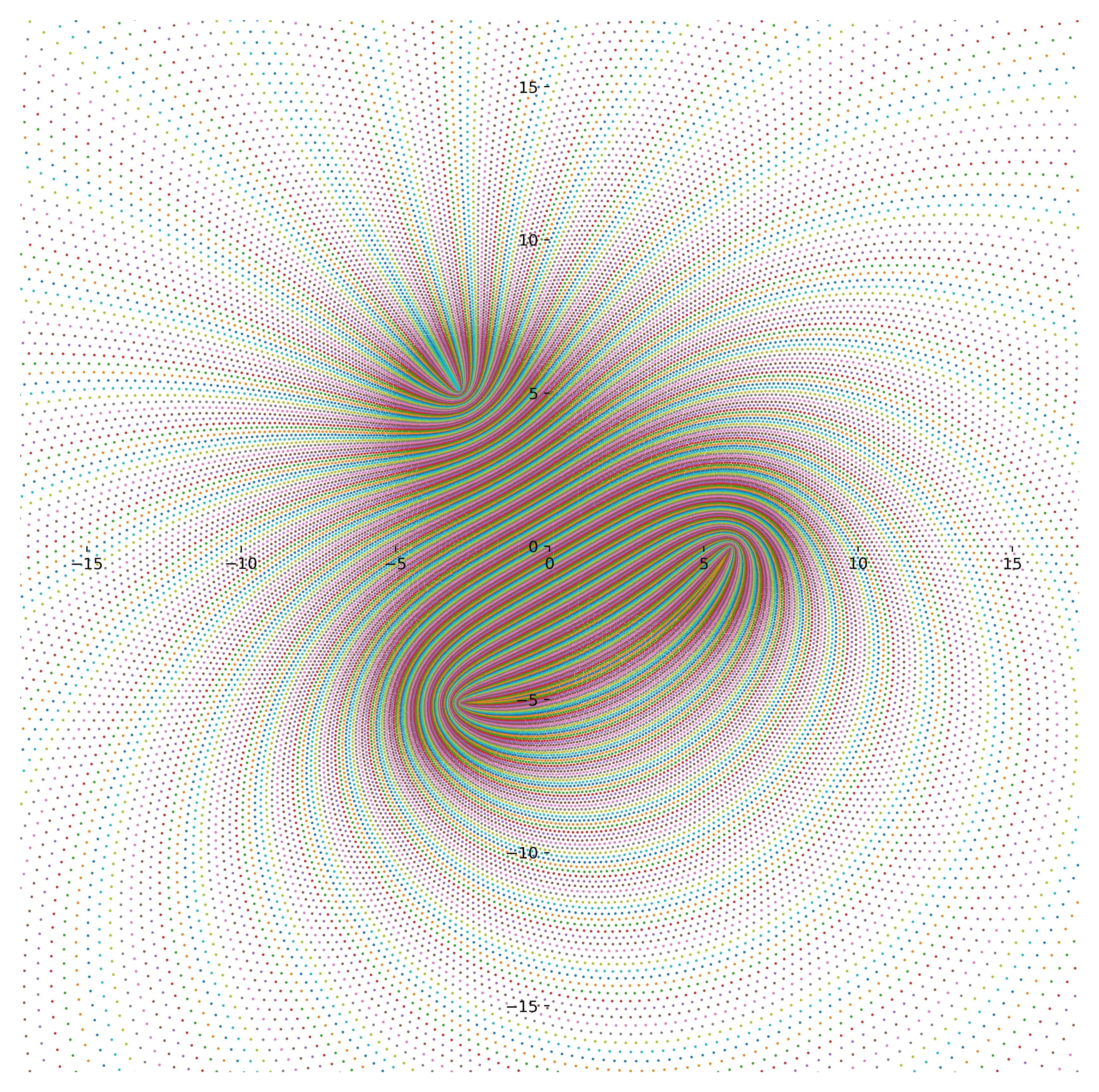}
		\caption{$x$-coordinates, $K = \Q(\sqrt{-3})$}
	\end{subfigure}
	\begin{subfigure}[b]{0.32\linewidth}
		\includegraphics[width=\linewidth]{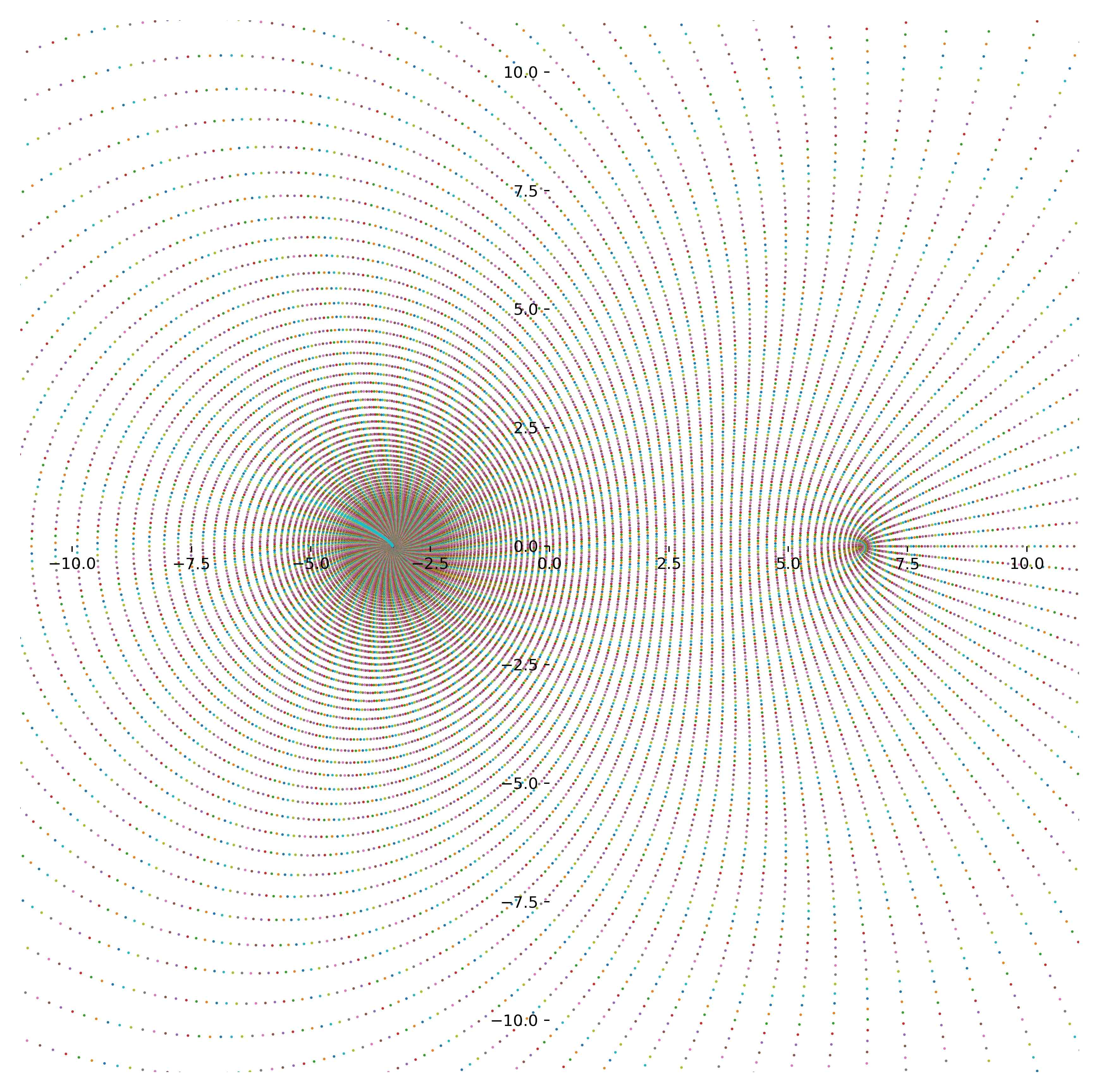}
		\caption{$x$-coordinates, $K = \Q(\sqrt{-43})$}
	\end{subfigure}

	\begin{subfigure}[b]{0.32\linewidth}
		\includegraphics[width=\linewidth]{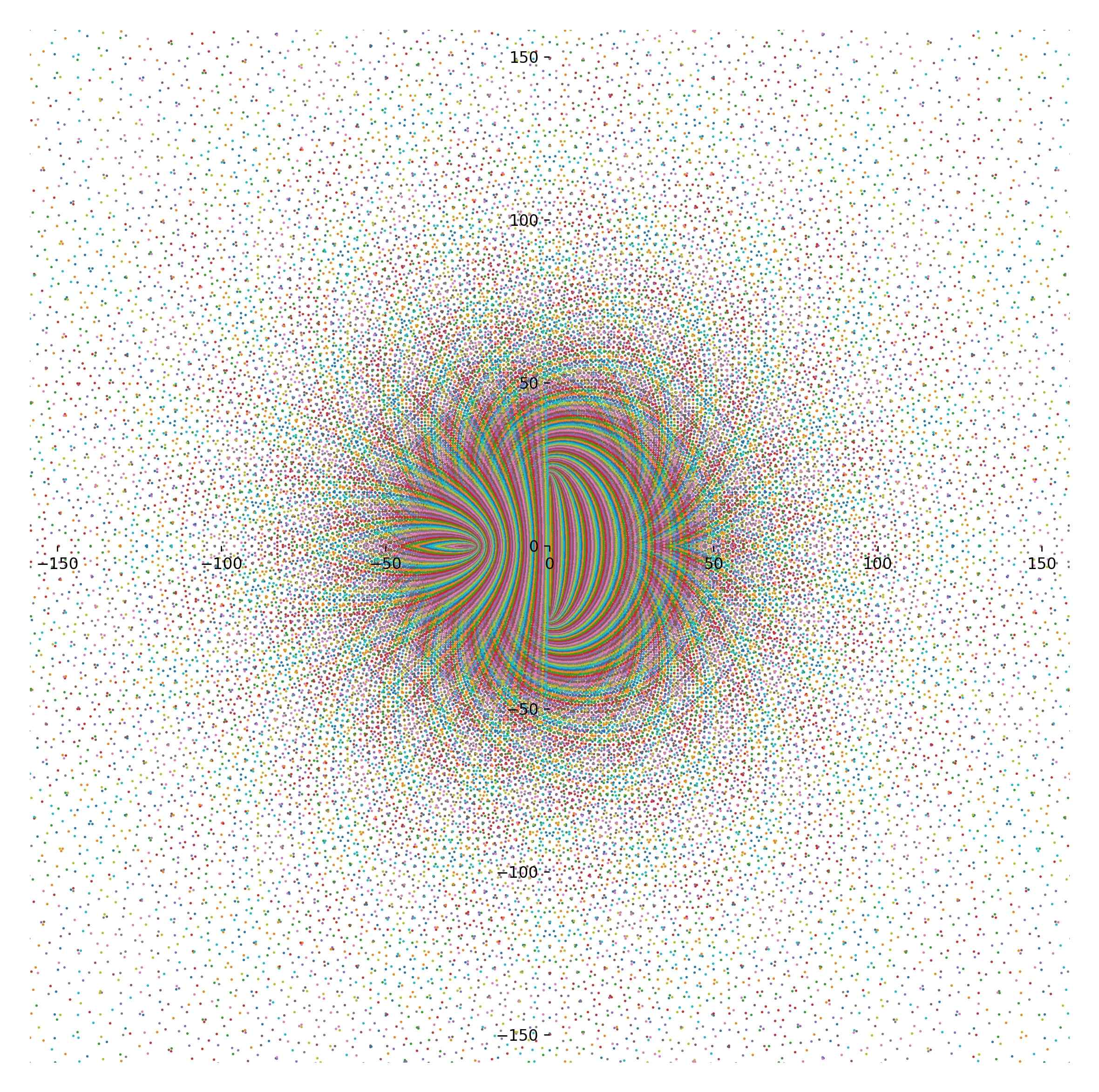}
		\caption{$y$-coordinates, $K = \Q(\sqrt{-1})$}
	\end{subfigure}
	\begin{subfigure}[b]{0.32\linewidth}
		\includegraphics[width=\linewidth]{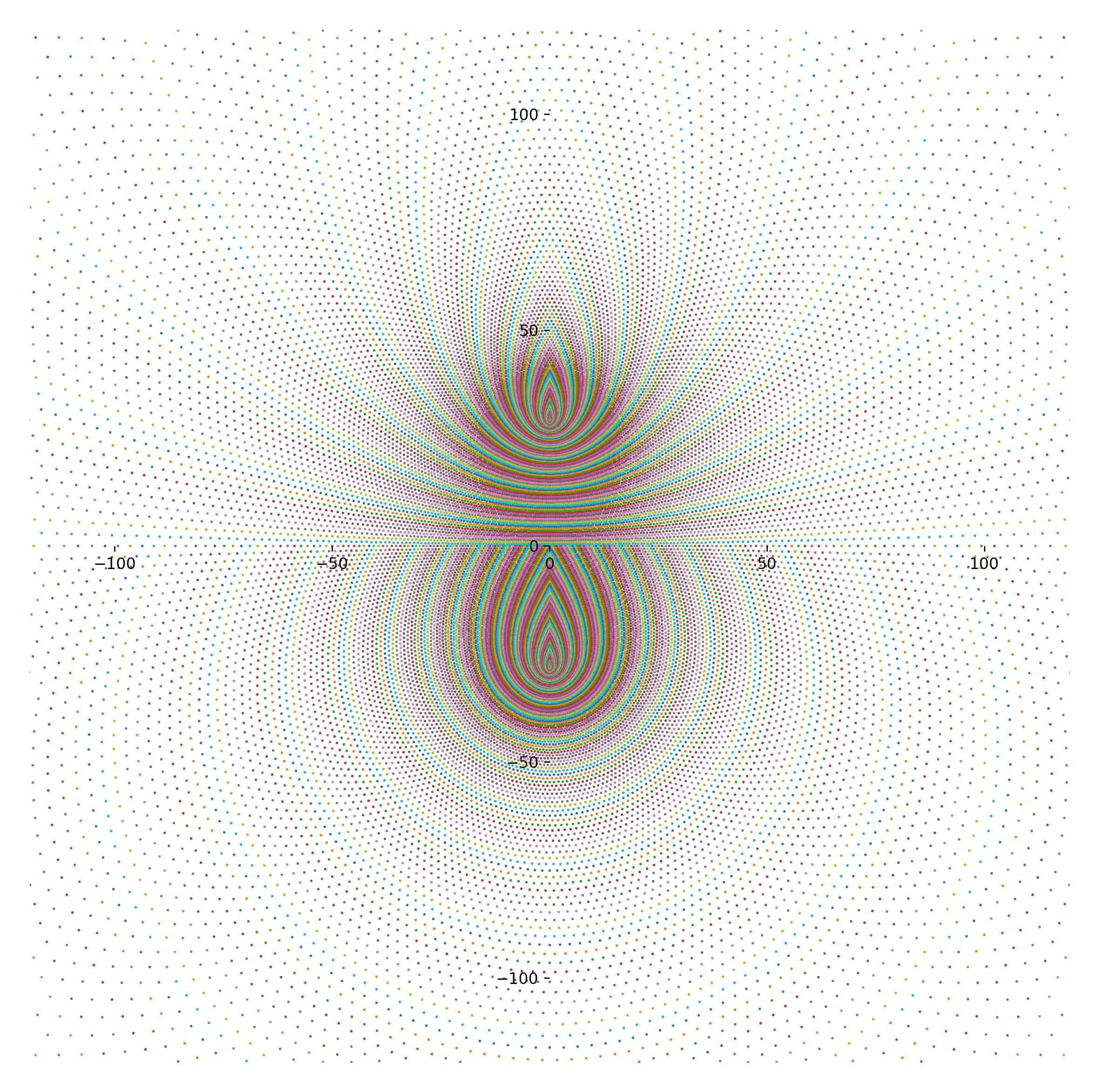}
		\caption{$y$-coordinates, $K = \Q(\sqrt{-3})$}
	\end{subfigure}
	\begin{subfigure}[b]{0.32\linewidth}
		\includegraphics[width=\linewidth]{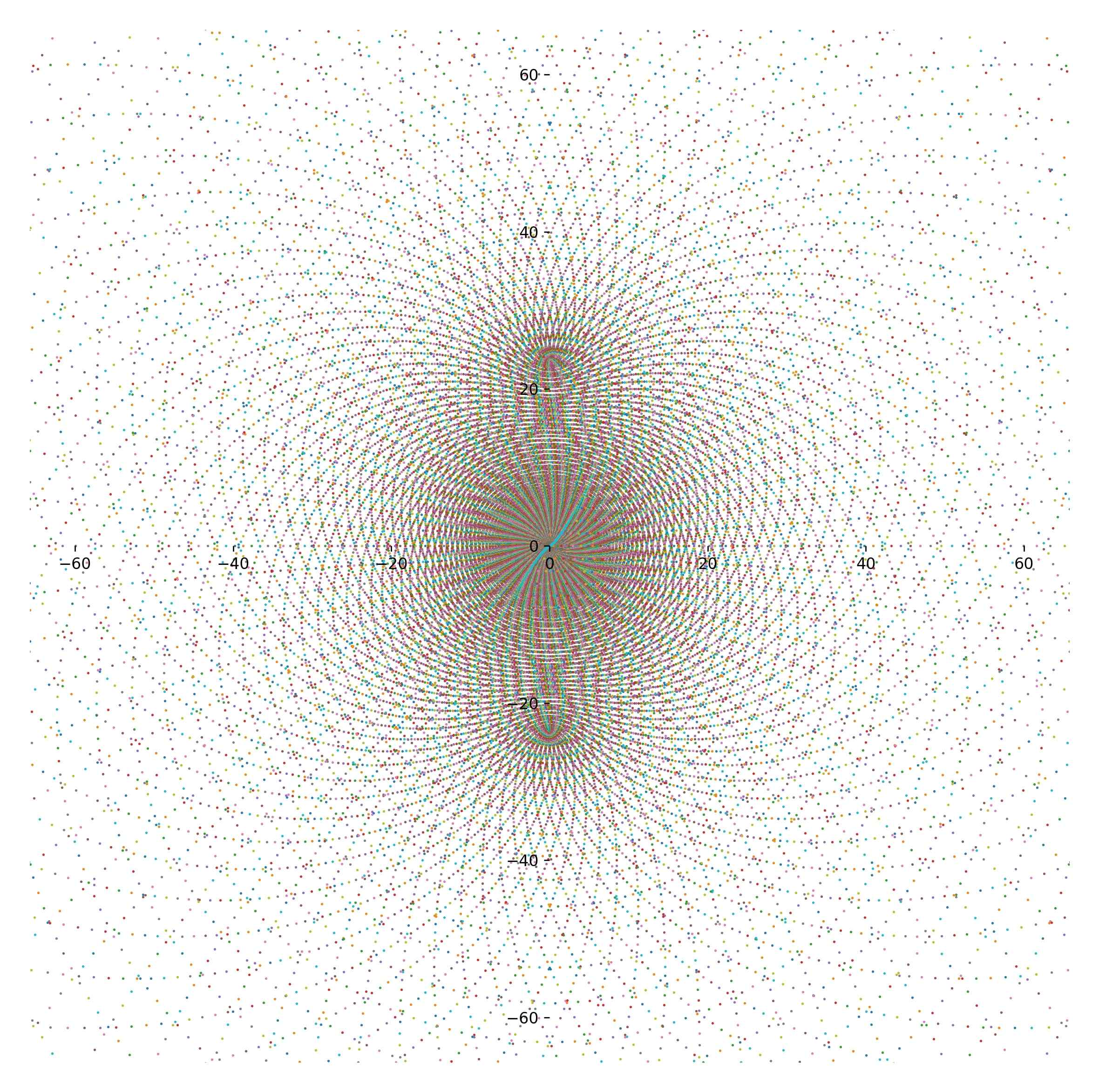}
		\caption{$y$-coordinates, $K = \Q(\sqrt{-43})$}
	\end{subfigure}

	\begin{subfigure}[b]{0.32\linewidth}
		\includegraphics[width=\linewidth]{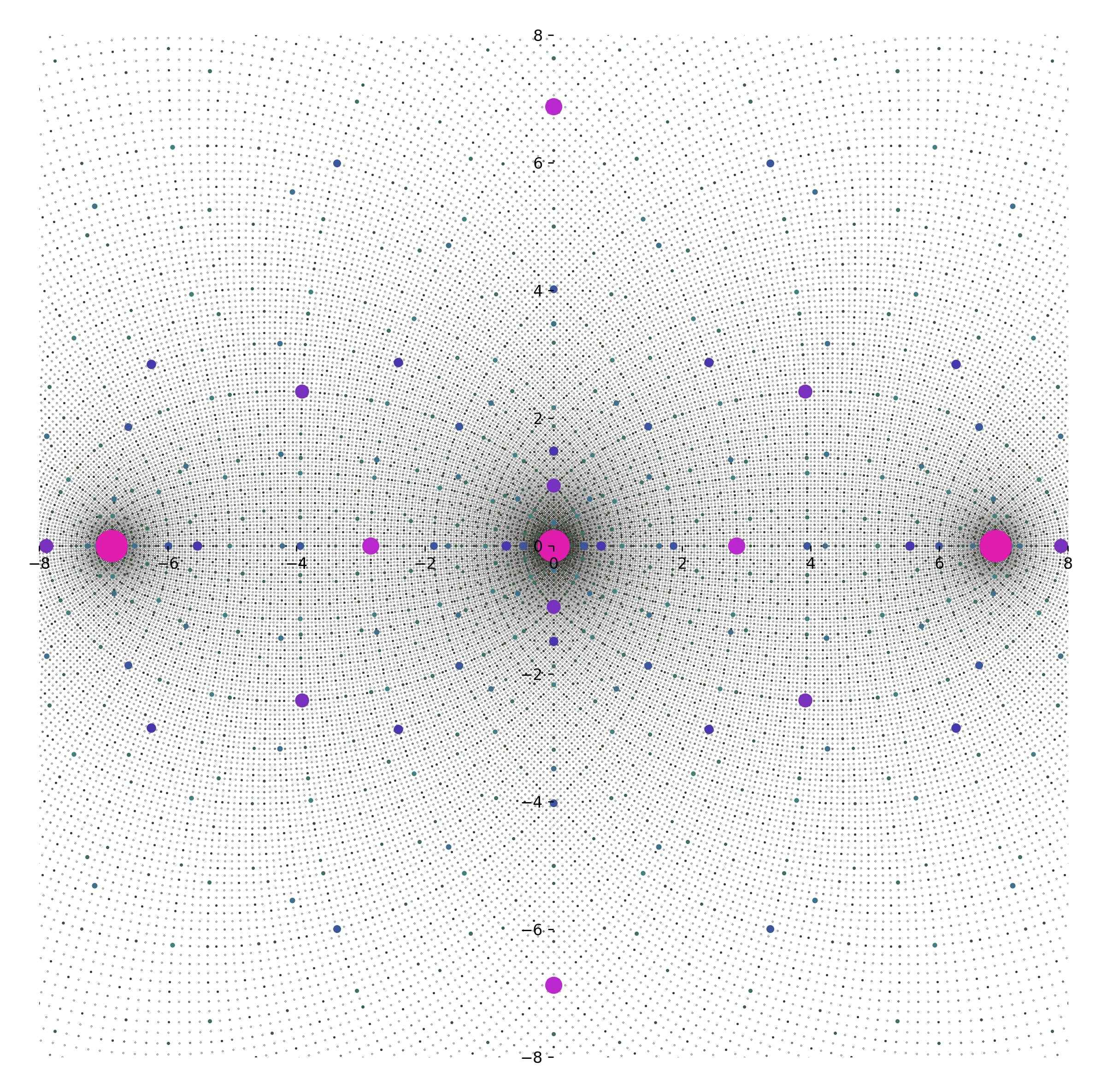}
		\caption{$x$-coordinates, $K = \Q(\sqrt{-1})$}
	\end{subfigure}
	\begin{subfigure}[b]{0.32\linewidth}
		\includegraphics[width=\linewidth]{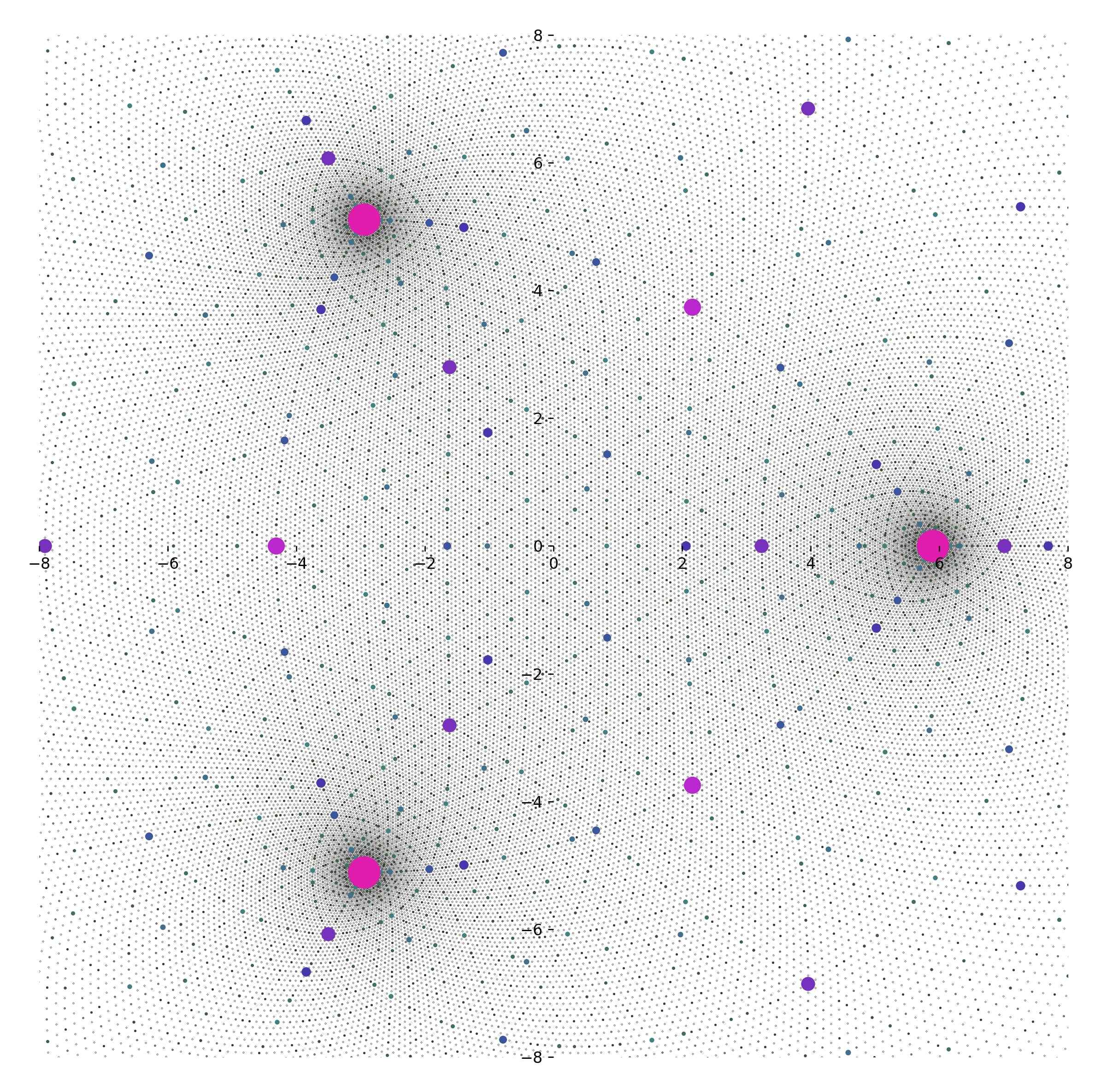}
		\caption{$x$-coordinates, $K = \Q(\sqrt{-3})$}
	\end{subfigure}
	\begin{subfigure}[b]{0.32\linewidth}
		\includegraphics[width=\linewidth]{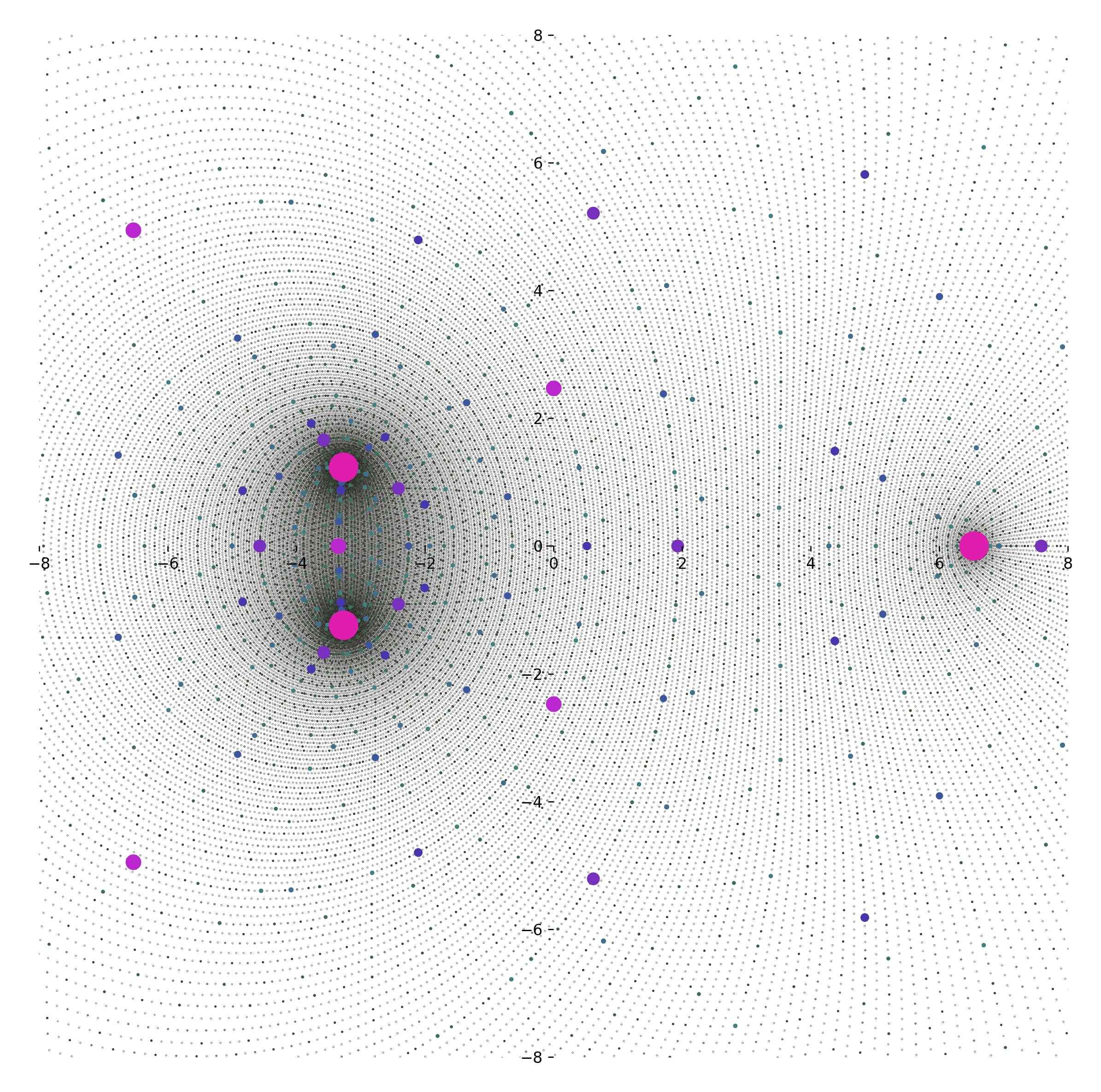}
		\caption{$x$-coordinates, $K = \Q(\sqrt{-7})$}
	\end{subfigure}
	
	%Changed ``Coordinates 400-torsion'' to ``Coordinates of 400-torsion''
	
	\caption{Coordinates of 400-torsion points of elliptic curves $E \cong \C/\O_K$}
	\label{fig:TorsionPts}
\end{figure}

\subsection{The Galois Group and Its Action}

Our goal now is to compute the ray class group of modulus $\m$ for the quadratic imaginary field $K$ and to describe its Galois action on the ray class field. 

Recall that the ideal class group $\Cl_K(1)$ is contained in $\Cl_K(\m)$ for any modulus $\m$. Thus, in order to compute the full ray class group of modulus $\m$, we also need to compute the ideal class group for any quadratic imaginary field $K$. This part of the computation can be rather complex, and its difficulty depends on the base field. That is, computing the full group $\Cl_K(\m) \cong \Gal(K[\m]/K)$ can be quite difficult. However, computing the subgroup $\Gal(K[\m]/K[1])$ proves to be more straightforward using class field theory. The following proposition describes the structure of this group. 

\begin{proposition}
	\label{prop:GalgroupofrayclassfieldoverHilbert}
	Let $K = \Q(\sqrt{-D})$ be a quadratic imaginary field and $\m \subset \O_K$ an ideal. Let $K[\m]$ be the ray class field for $K$ of modulus $\m$, and let $K[1]$ be its Hilbert class field. Let $\rho$ be a primitive third root of unity. Then the field extension $K[1] \subset K[\m]$ has Galois group $$\Gal(K[\m]/K[1]) \cong \left( \O_K/\m \O_K \right)^\times/\O_K^\times,$$ where $\O_K^\times$ is the cyclic multiplicative group $\{\pm 1, \pm i\}$ when $D = 1$, $\{\pm 1, \pm \rho, \pm \bar \rho\}$ when $D = 3$, and $\{\pm 1\}$ otherwise.
\end{proposition}

It is worth noting that, while the result of this proposition seems to be known by those who use class field theory regularly, it is only scarcely mentioned explicitly in the literature. Because of this, we have decided to include a proof of this fact.

\begin{proof}
	First, using class field theory, we know that $\Gal(K[\m]/K) \cong \A_K^\times/U_\m K^\times$, where $\A_K^\times$ is the group of id{\`e}les and $$U_\m = \C^\times \cdot \prod_{v < \infty, v \nmid \m} \O_{K,v}^\times \cdot \prod_{v < \infty, v \mid \m} (1 + \p_v^{\ord_v(\m)} \O_{K,v}),$$ where $\O_{K,v}$ is the completion of $\O_K$ with respect to the place $v$, and where $\p_v$ is the prime ideal corresponding to the place $v$. Similarly, we know that $\Gal(K[1]/K) \cong \A_K^\times/U_1 K^\times$. 
	
	By Galois theory, the group $\Gal(K[\m]/K[1])$ is a subgroup of $\Gal(K[\m]/K)$, and we have the isomorphism $$\Gal(K[1]/K) \cong \Gal(K[\m]/K)/\Gal(K[\m]/K[1]).$$ This isomorphism is induced by the surjective homomorphism $$\varphi: \Gal(K[\m]/K) \onto \Gal(K[1]/K), \hspace{0.5in} \sigma \mapsto \sigma|_{K[1]},$$ whose kernel is isomorphic to $\Gal(K[\m]/K[1])$.
	
	This homomorphism can be described id{\`e}lically in the following way: $$\psi: \A_K^\times/U_\m K^\times \to \A_K^\times/U_1 K^\times, \hspace{.3in} (a_v)_v + U_\m K^\times \mapsto (a_v)_v + U_1 K^\times.$$ Note that $U_\m \subset U_1$, so this map is surjective. The elements of $\ker(\psi)$ are cosets of $U_\m K^\times$ which are contained in $U_1 K^\times$, and so $\ker(\psi) \cong U_1 K^\times/U_\m K^\times$. Additionally, the isomorphisms $\A_K^\times/U_\m K^\times \cong \Gal(K[\m]/K)$ and $\A_K^\times/U_1 K^\times \cong \Gal(K[1]/K)$ are given by the Artin map. By a slight abuse of notation, we denote both Artin maps by $\phi$. Then we have the following diagram: 
	\begin{equation*}
		\begin{tikzcd}
			\A_K^\times/U_\m K^\times \arrow{r}{\psi} \arrow[swap]{d}{\phi} & \A_K^\times/U_1 K^\times \arrow{d}{\phi}\\
			\Gal(K[\m]/K) \arrow{r}{\varphi} & \Gal(K[1]/K)
		\end{tikzcd}
	\end{equation*}
	Note that the elements of $\phi(U_1 K^\times/U_\m K^\times)$ are precisely the elements of $\Gal(K[\m]/K)$ which act nontrivially on $K[\m]$, but trivially on $K[1]$. In other words, we have that $$\phi(U_1 K^\times/U_\m K^\times) = \ker(\varphi).$$ This shows that the diagram above is commutative, which induces the following isomorphism: $$\Gal(K[\m]/K[1]) \cong U_1 K^\times/U_\m K^\times,$$ where this uses the fact that $\ker(\varphi) \cong \Gal(K[\m]/K[1])$.
	
	In the interest of computing the group $U_1 K^\times/U_\m K^\times$ explicitly, we let $G = U_1$ and $H = U_\m K^\times$. Then $U_1 K^\times/U_\m K^\times = GH/H$, and by the second isomorphism theorem we have that $$GH/H \cong G/(G \cap H).$$ Note that $G \cap H = U_1 \cap (U_\m K^\times)$. An element of $U_\m K^\times$ is of the form $(a_v \alpha)_v$ for $(a_v)_v \in U_\m$ and $\alpha \in K^\times$. If this element is also in $U_1 = \C^\times \cdot \prod_{v < \infty} \O_{K,v}^\times$, then $\alpha$ must be in $\O_{K,v}^\times$ for every place $v$. This happens if and only if $\ord_v(\alpha) = 0$ for every place $v$, so $\alpha$ must be an element of $\O_K^\times$. Then, since $U_\m \subset U_1$, we have that $U_1 \cap (U_\m K^\times) = U_\m \O_K^\times$. Thus we have the following: $$U_1 K^\times/U_\m K^\times \cong U_1/(U_1 \cap (U_\m K^\times)) = U_1/U_\m \O_K^\times.$$ 
	
	Using the definitions of $U_1$ and $U_\m$, we then have the following: $$U_1/U_\m \O_K^\times \cong \left( \prod_{v < \infty, v \mid \m} \O_{K,v}^\times \right)/ \left( \prod_{v < \infty, v \mid \m} (1 + \p_v^{\ord_v(\m)} \O_{K,v}) \right) \O_K^\times.$$ By the Chinese Remainder Theorem, we need only compute $\O_{K,v}^\times/(1 + \p_v^{\ord_v(\m)} \O_{K,v})$ for each place $v \mid \m$. 
	
	For a fixed place $v$ and a fixed $k \in \N$, we construct a map $$f: \O_{K,v}^\times \to (\O_{K,v}/\p_v^k \O_{K,v})^\times$$ sending a unit $u \mapsto u \mod \p_v^k$. This is a homomorphism, and we claim it is surjective. 
	
	For an element $u \in (\O_{K,v}/\p_v^k \O_{K,v})^\times$, we have $u = a_0 + a_1 \pi + \cdots + a_{k-1} \pi^{k-1}$, where $\pi$ is a uniformizer of $\p_v$, $a_i \in \O_{K,v}/\p_v$ for all $i$, and $a_0 \neq 0$ ($a_0$ cannot be 0 because it is a unit). Then, since $a_0 \neq 0$, there exists an element $\tilde u \in \O_{K,v}^\times$ such that $\tilde u = a_0 + a_1 \pi + \cdots + a_{k-1} \pi^{k-1} + \cdots$ which maps to $u$. 
	
	Now, note that the kernel of $f$ is any element $u \in \O_{K,v}^\times$ such that $u \equiv 1 \mod \p_v^k$. That is, $u \in (1 + \p_v^k \O_{K,v})$. Thus $\ker(f) =  (1 + \p_v^k \O_{K,v})$, so $$\O_{K,v}^\times/(1 + \p_v^k \O_{K,v}) \cong (\O_{K,v}/\p_v^k \O_{K,v})^\times.$$ By a well-known fact about completions, we also have that $$(\O_{K,v}/\p_v^k \O_{K,v})^\times \cong (\O_K/\p_v^k \O_K)^\times.$$ Thus we have the following: $$U_1/U_\m \O_K^\times \cong \left( \prod_{v < \infty, v \mid \m} (\O_K/\p_v^{\ord_v(\m)}\O_K)^\times \right) / \O_K^\times \cong (\O_K/\m\O_K)^\times/\O_K^\times,$$ where the last equality comes from reversing the direction of the Chinese Remainder Theorem.
\end{proof}

Now that we know the structure of $\Gal(K[\m]/K[1])$, we need to determine its action on $K[\m]$. Since $K[\m] = K(j(E), h(E[\m]))$ and $K[1] = K(j(E))$, then the action of $\Gal(K[\m]/K[1])$ on $K[\m]$ permutes only the elliptic curve torsion points and leaves the $j$-values fixed. In other words, we want to determine the action of $(\O_K/\m \O_K)^\times/\O_K^\times$ on the set $h(E[\m])$.

We first need an elliptic curve with CM by $K$. Recall that choosing an elliptic curve is equivalent to choosing a lattice in $\C$ by the Uniformization Theorem, so the simplest way to choose an elliptic curve with CM by $K$ is to choose the lattice $\Lambda = \O_K$. This lattice will automatically be closed under multiplication by elements of $\O_K$, so $\C/\O_K$ will have CM by $K$. 

For our purposes, however, we need to be more explicit. For $K = \Q(\sqrt{-D})$, we use the lattice $\Lambda := \Z + \Z\alpha$, where $\alpha \in \O_K$ is chosen in the fundamental domain for $\text{SL}_2(\Z)$ to be the following: $$\alpha = \begin{cases} \sqrt{-D} & D \equiv 1, 2 \mod 4, \\ \frac{-1 + \sqrt{-D}}{2} & D \equiv 3 \mod 4. \end{cases}$$ Note that in the second case, we could have equivalently chosen $\alpha = \frac{1 + \sqrt{-D}}{2}$. We chose this definition so that $\alpha$ is a third root of unity in the case $D = 3$.

Thus we now have an elliptic curve $E \cong \C/\Lambda$ with CM by $K$. Note that a point $(x,y)$ on our elliptic curve can be found via the Uniformization Theorem by taking a complex number $z \in \C/\Lambda$ and using the Weierstrass $\wp$-function to get $x = \wp(z; \Lambda)$ and $y = \wp'(z; \Lambda)$. Because of this, we will use $\C/\Lambda$ for the majority of our calculations, as it is much simpler in this setting to find torsion points, to add two elements, and to compute the Galois action on $E$ when compared to viewing $E$ as a subset of $\mathbb{P}^2$, even though the latter viewpoint is the one needed in the end. Additionally, we will restrict our attention to the case in which $\m = (m)$ is an ideal generated by an integer. The analogy of Gaussian periods we define in the next section allows for any integral ideal $\m$, but we believe it is more intuitive and insightful to discuss only the $\m = (m)$ case for the time being. 

To determine the $m$-torsion points of $E$, we must find the complex numbers $z$ such that $m z \in \Lambda$. If we restrict $z$ to the fundamental parallelogram, then an $m$-torsion point $z$ must be of the form $z = \frac{1}{m} (a + b\alpha)$, where $a, b \in \{0, 1, \ldots, m - 1\}$. This gives all $m^2$ of the $m$-torsion points, matching the remark made at the end of Definition \ref{def:torsionsbgp}. Note that $z$ can also be viewed as an element of $\O_K/m\O_K$. 

We now define the Galois action. Let $\beta \in (\O_K/m\O_K)^\times/\O_K^\times$, so we can write $\beta = c + d\alpha$ for $c,d \in \Z/m\Z$. Let $\rho \in h(E[m])$; that is, $\rho$ is an element of $K[m]$. Since $\rho$ is a power of the $x$-coordinate of an $m$-torsion point of $E$, then there exist $a, b \in \Z/m\Z$ and a power $e \in \{1, 2, 3\}$ (depending on the base field) such that $z = \frac{1}{m}(a + b \alpha)$ and $\rho = \wp(z; \Lambda)^e$. The action of $\beta$ on $\rho$ is given by $$[\beta] \cdot \rho = \wp(\beta z; \Lambda)^e,$$ where $\beta z$ is the standard multiplication of complex numbers. 

\begin{remark}
	For computational purposes, it is best if we view $z$ and $\beta$ as matrices in $\Mat_2(\Z/m\Z)$, where we represent $\alpha$ with the companion matrix of its minimal polynomial. More explicitly, we use the matrix $$C_\alpha = \begin{cases} \begin{pmatrix} 1 & 0 \\ 0 & -D \end{pmatrix} & D \equiv 1, 2 \mod 4, \\ \begin{pmatrix} 1 & -\frac{D + 1}{2} \\ 0 & -1 \end{pmatrix} & D \equiv 3 \mod 4. \end{cases}$$ to represent $\alpha$. Then we can view $\O_K/m\O_K$ as a subring embedded into $\Mat_2(\Z/m\Z)$ via the map sending $\alpha \mapsto C_\alpha$ and $1 \mapsto I$. Thus an element $\gamma \in \O_K/m\O_K$ is given by $\gamma = a I + b C_\alpha$ for some $a, b \in \Z/m\Z$, and $\gamma$ is an element of $(\O_K/m\O_K)^\times$ if $\det(\gamma)$ is relatively prime to $m$.
\end{remark}

\subsection{Computation and Examples}

Our goal here is to explicitly describe the analogy of Gaussian periods for quadratic imaginary fields and to provide some examples of these computations. We start with a definition.

\begin{definition}
	\label{def:RCFP}
	Let $K$ be a quadratic imaginary field, and choose $\alpha \in K$ so that $\O_K = \Z[\alpha]$ as described in the previous section. Let $\Lambda = \Z + \Z \alpha$ and $E$ be the elliptic curve isomorphic to $\C/\Lambda$. Choose $A \in (\O_K/ \m \O_K)^\times/\O_K^\times$, and let $d$ denote the multiplicative order of $A$. Let $z \in \O_K/\m\O_K$, represented as an element of $\C/\Lambda$, and let $\wp(z) := \wp(z; \Lambda)$ be the Weierstrass $\wp$-function. Then we define the following map: $$\eta_{K, \m, A}: \O_K/\m\O_K \to \C, \hspace{0.5in} \eta_{K, \m, A}(z) = \sum_{j = 0}^{d - 1} \wp(A^j z).$$ We call $\eta_{K, \m, A}(z)$ a \emph{ray class field period (RCFP) of modulus $\m$ and generator $A$}. Additionally, we call $\img(\eta_{K, \m, A})$ the \emph{ray class field period plot (RCFP plot) of modulus $\m$ and generator $A$}. 
\end{definition}

\begin{remark}
	In the above definition, one can choose $A \in (\O_K/\m\O_K)^\times/\O_K^\times$ by first choosing an element $B \in (\O_K/\m \O_K)^\times$ and computing the cyclic subgroup $\langle B \rangle$. If $\langle B \rangle \cap \O_K^\times$ is not the trivial group, then one can choose $A$ by taking $B$ to an appropriate power.
\end{remark}

In Figure \ref{fig:RCFPExamples}, we provide some examples of RCFP plots for various choices of fields $K$, ideals $\m = (m)$, and elements $A$. Note that in these examples, we use fields with class number 1 to avoid the issue of computing the Hilbert class field and ideal class group as discussed above. 

\begin{figure}[h!]
	\centering
	\begin{subfigure}[b]{0.4\linewidth}
		\includegraphics[width=\linewidth]{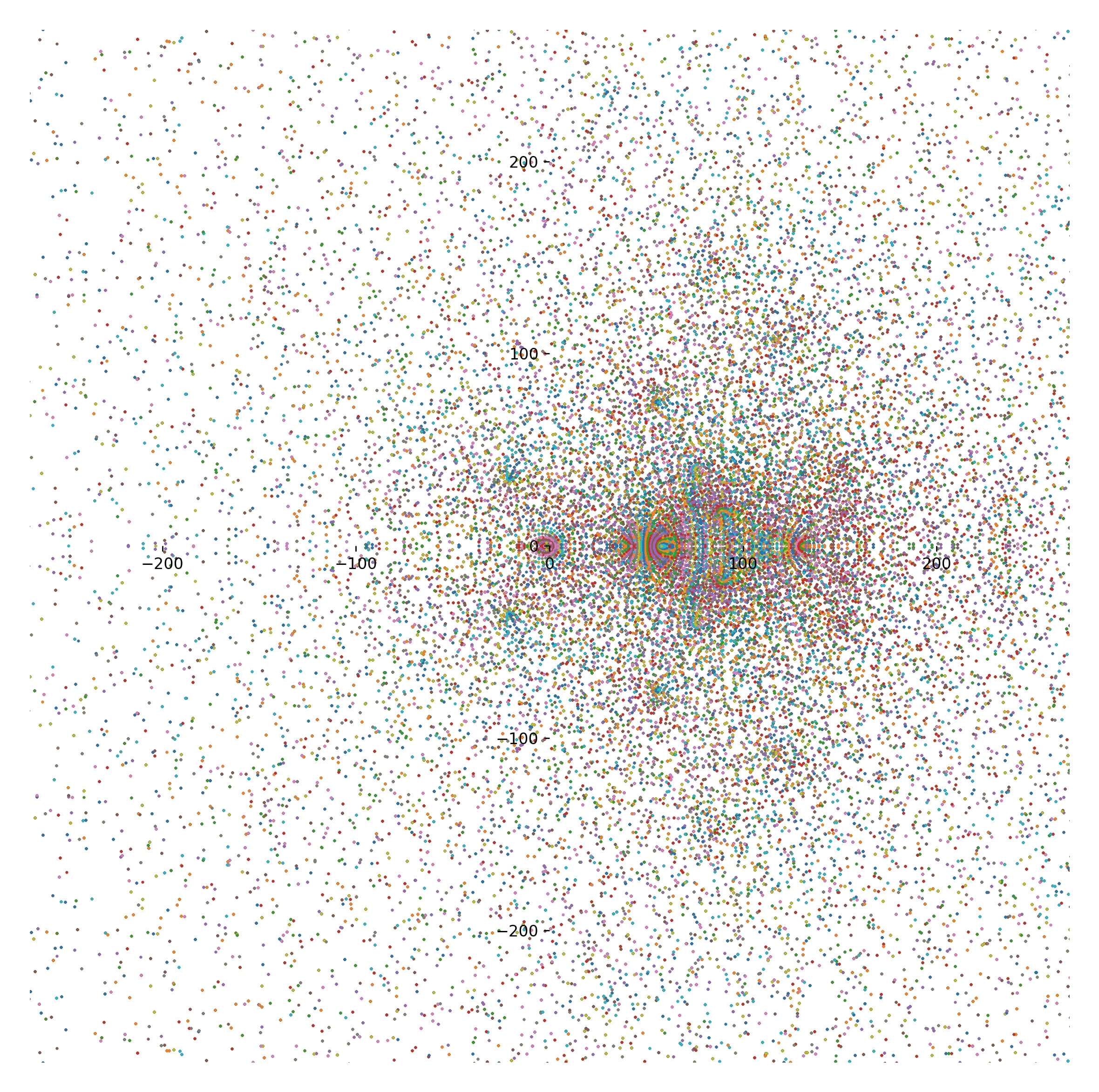}
		\caption{$D = 1$, $m = 1155$, $A = 631$}
	\end{subfigure}
	\begin{subfigure}[b]{0.4\linewidth}
		\includegraphics[width=\linewidth]{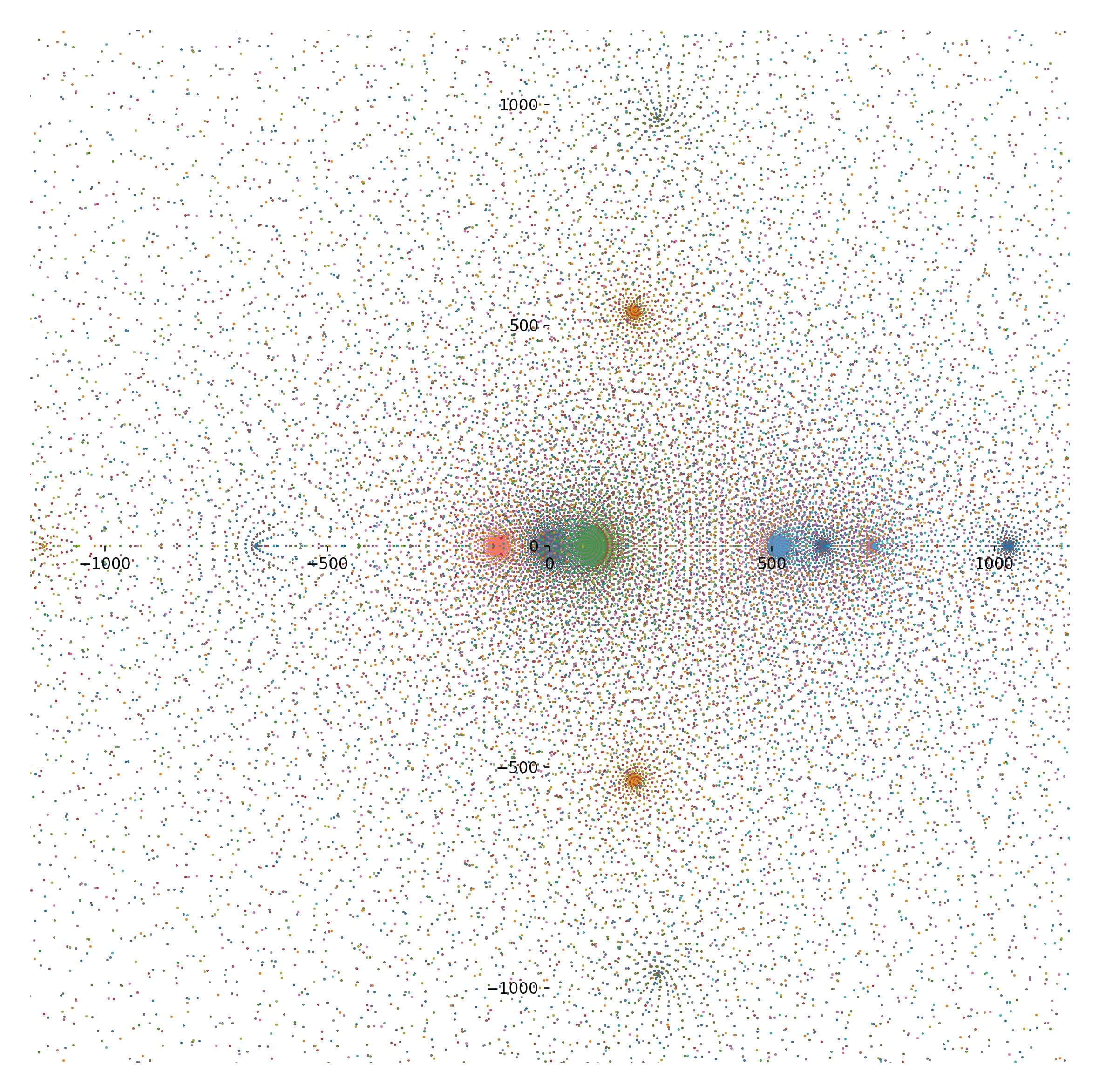}
		\caption{$D = 3$, $m = 1265$, $A = 759 + 254\alpha$}
	\end{subfigure}
	
	\begin{subfigure}[b]{0.4\linewidth}
		\includegraphics[width=\linewidth]{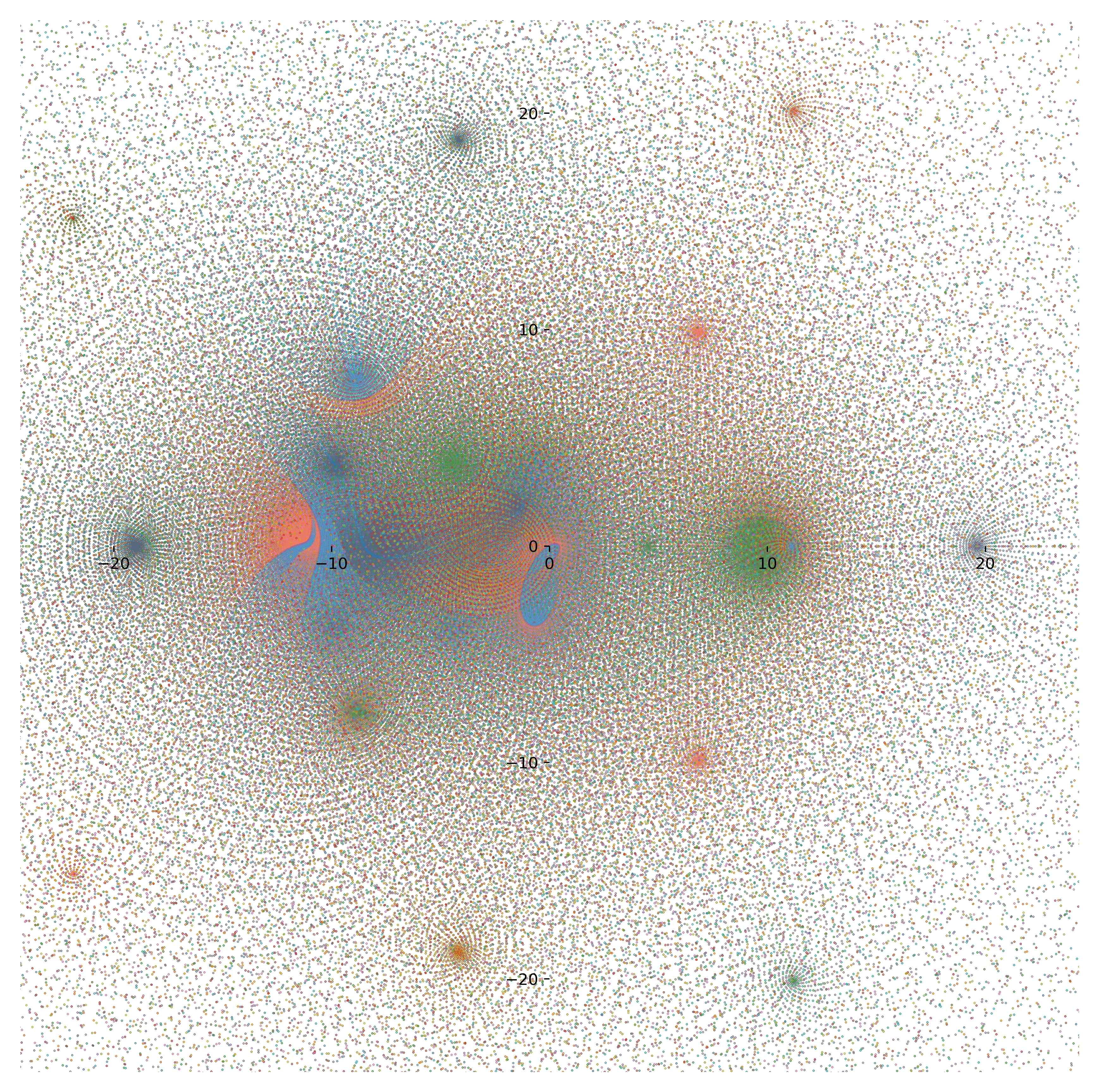}
		\caption{$D = 7$, $m = 1265$, $A = 1 + 253\alpha$}
	\end{subfigure}
	\begin{subfigure}[b]{0.4\linewidth}
		\includegraphics[width=\linewidth]{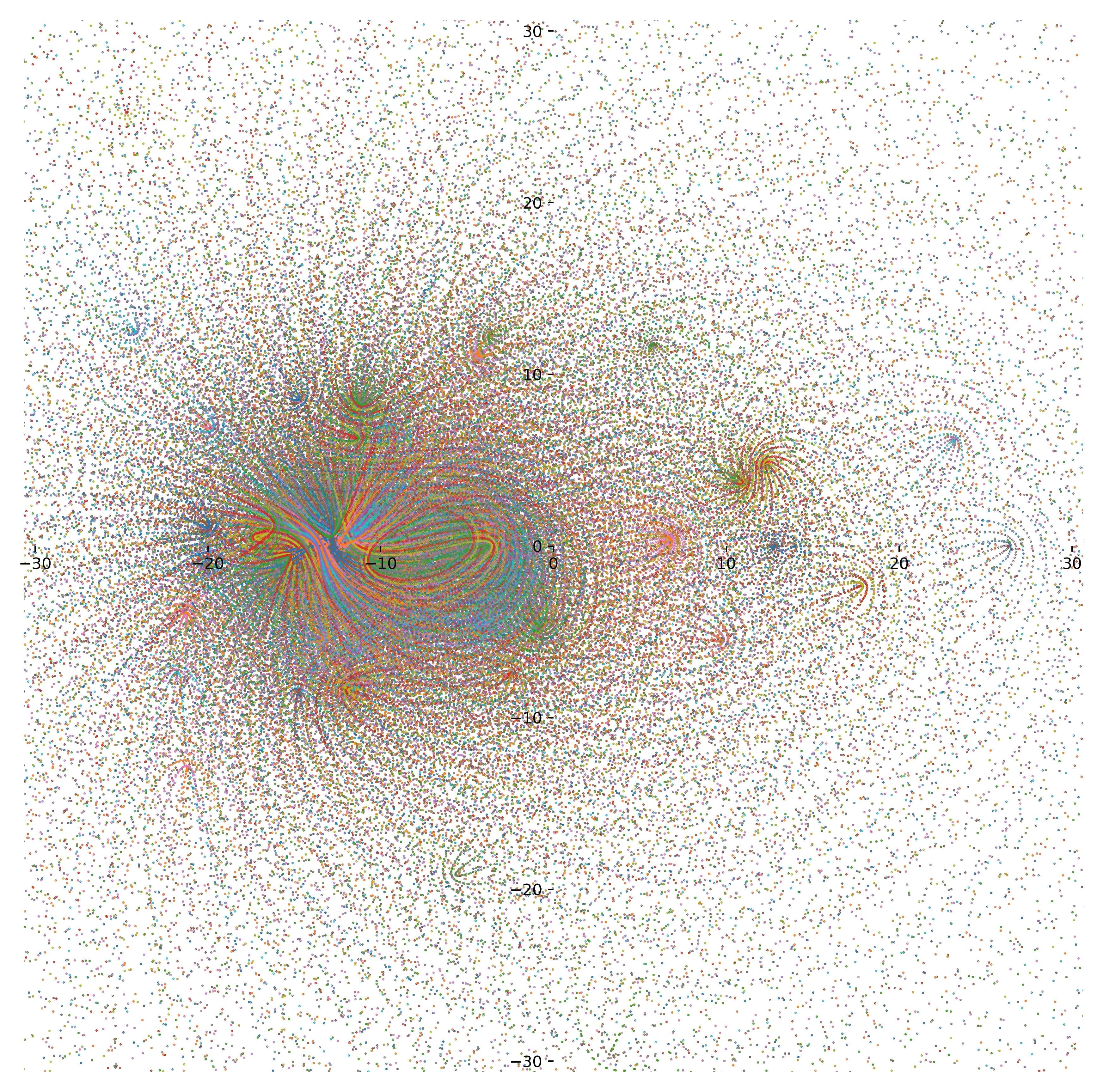}
		\caption{$D = 19$, $m = 1155$, $A = 1079 + 1078\alpha$}
	\end{subfigure}
	
	\caption{RCFP plots for the element $A$ and modulus $(m)$, where $K = \Q(\sqrt{-D})$}
	\label{fig:RCFPExamples}
\end{figure}

It should be noted that the scales for the real and imaginary axes vary in these images. The reason for doing this is because the values of $\wp$ vary largely for different choices of $K$, $\m$, and $z$. In particular, as the norm of the modulus $\m$ increases, the values of the RCFPs get further from the origin (the points ``go off to infinity'' in a sense). This problem did not exist in the Gaussian period setting since all roots of unity lie on the unit circle, so they were bounded in size. Because of this issue, we would like to rescale the RCFPs in some ``nice'' way so that the RCFP plots are contained within some bounded region.

One natural idea is to map RCFPs to the unit disc so the images are all uniform in size, and the ``nice'' property we would like is for our mapping to be conformal. Unfortunately, there are no conformal mappings from $\C$ to the unit disc, so we decided to make do with a rescaling of points that was less nice. We note that we lack knowledge in this area, however, so there could very well be a rescaling map that handles this issue better than the one we describe below.

If $w \in \img(\eta_{K, \m, A})$, we rescale using the map $$w \mapsto \frac{w}{|w| + \sqrt[4]{|\Nm(\m)|}},$$ where $\Nm$ is the norm map from $K$ to $\Q$. In the case where $\m = (m)$, this map is $w \mapsto \frac{w}{|w| + \sqrt{m}}$. Additionally, we would like to comment about the choice of using the 4th root. This was an ad hoc choice that seemed to generate images whose patterns were the easiest to see (at least for the cases where $D \neq 1, 3$) because the points were close enough together while also not being too close. We provide some examples of RCFP plots using this map in Figure \ref{fig:RCFPExamplestoDisc}. 

\begin{figure}[h!]
	\centering
	\begin{subfigure}[b]{0.4\linewidth}
		\includegraphics[width=\linewidth]{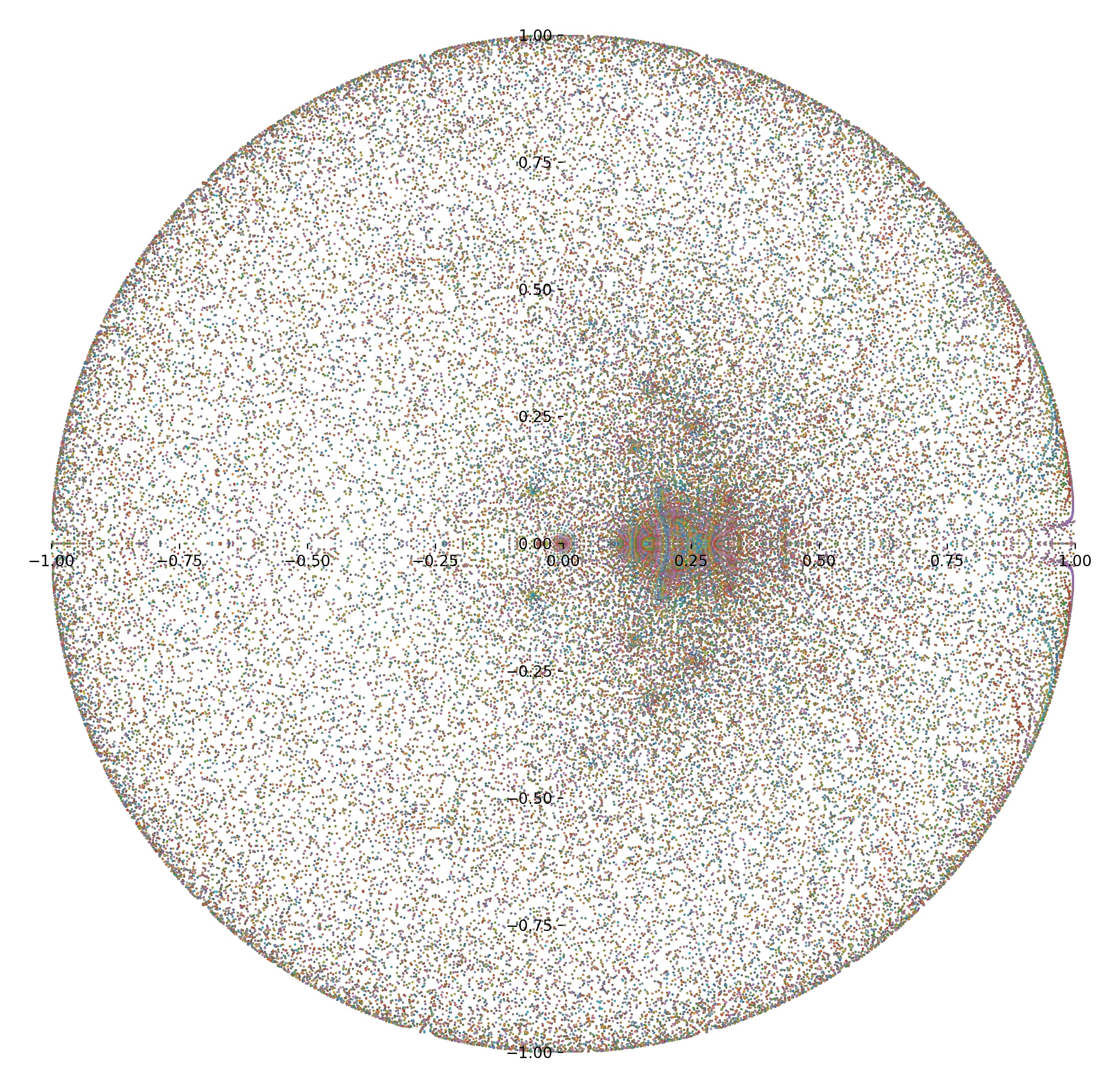}
		\caption{$D = 1$, $m = 1155$, $A = 631$}
	\end{subfigure}
	\begin{subfigure}[b]{0.4\linewidth}
		\includegraphics[width=\linewidth]{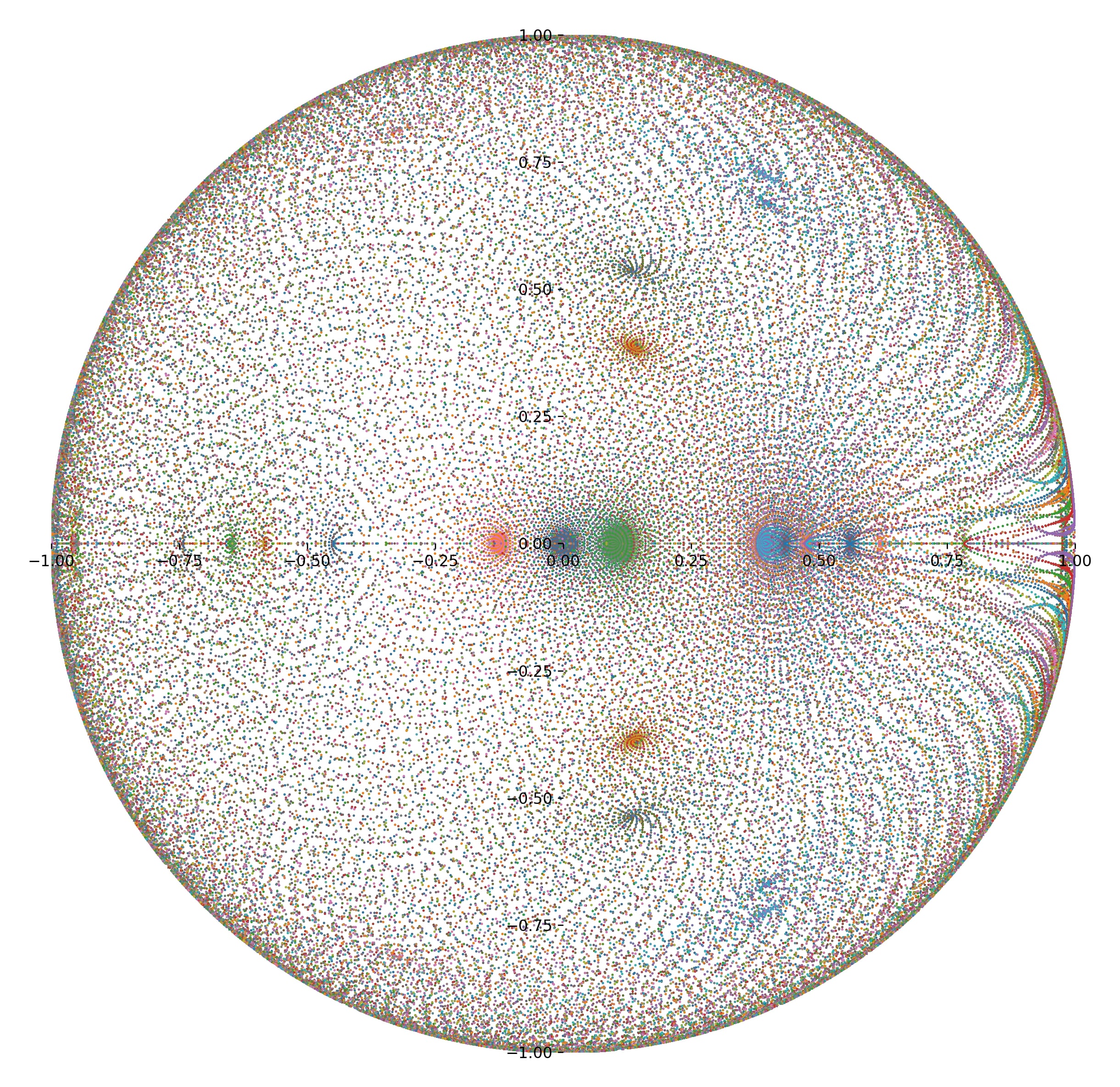}
		\caption{$D = 3$, $m = 1265$, $A = 759 + 254\alpha$}
	\end{subfigure}
	
	\begin{subfigure}[b]{0.4\linewidth}
		\includegraphics[width=\linewidth]{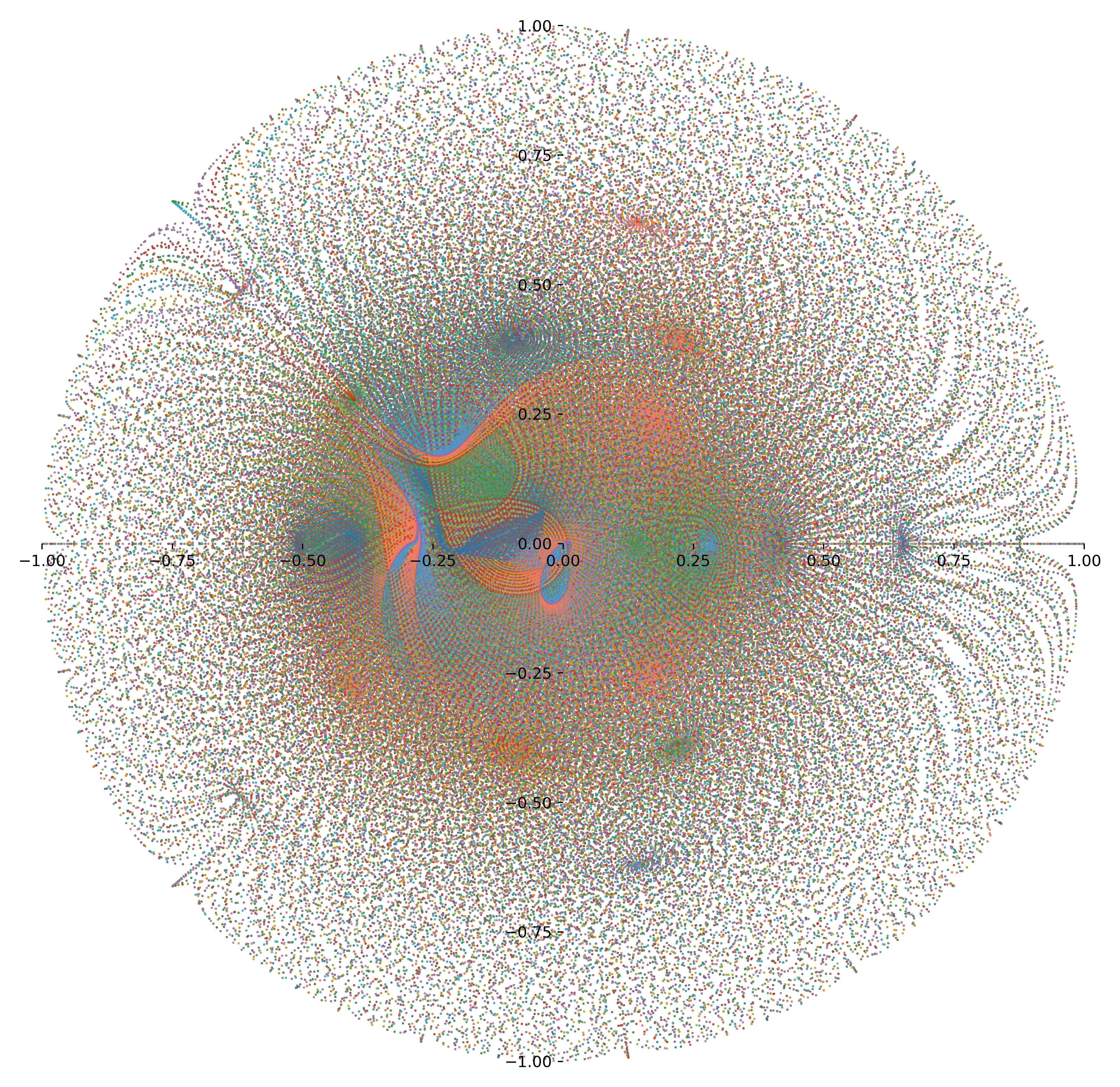}
		\caption{$D = 7$, $m = 1265$, $A = 1 + 253\alpha$}
	\end{subfigure}
	\begin{subfigure}[b]{0.4\linewidth}
		\includegraphics[width=\linewidth]{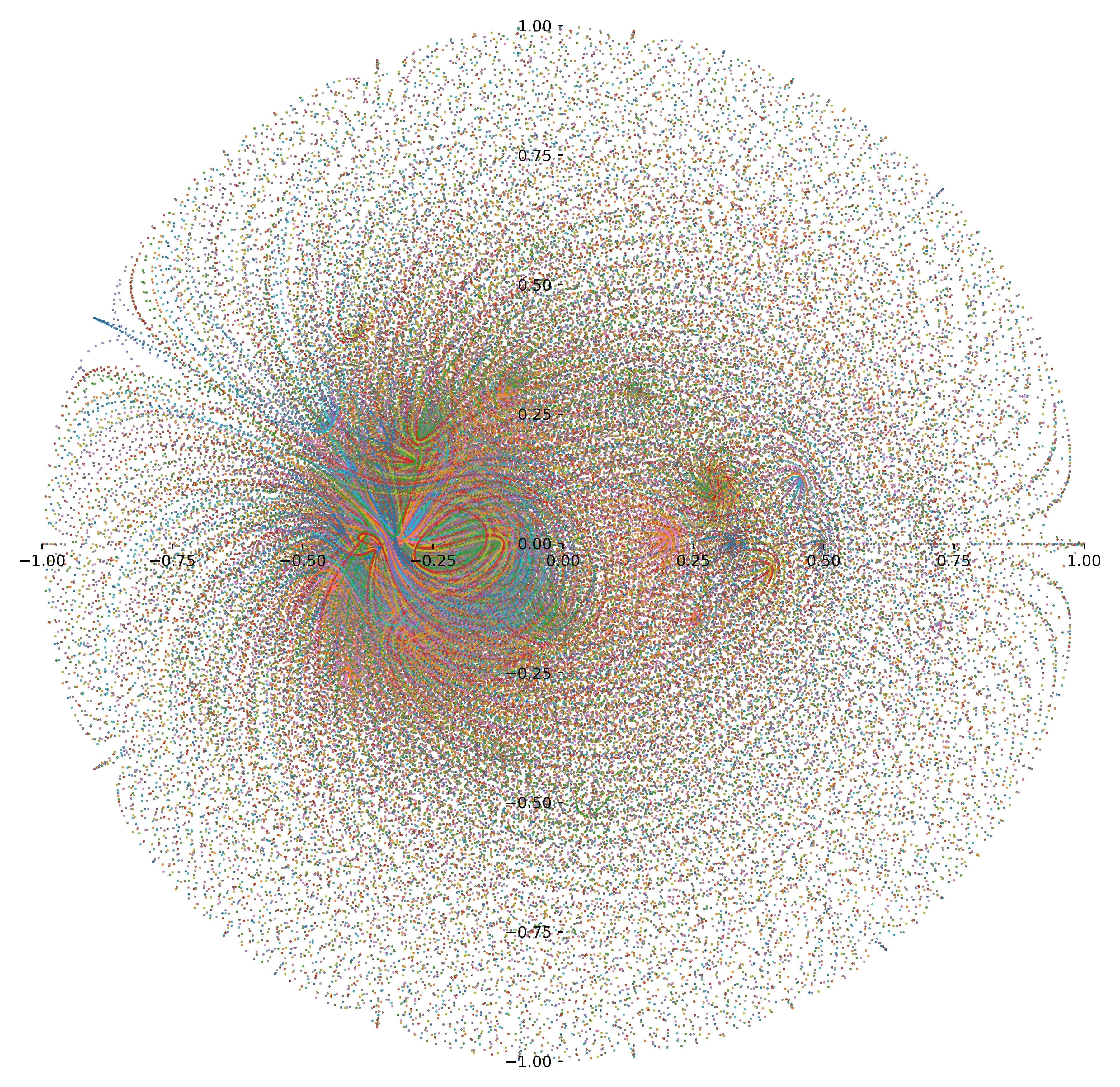}
		\caption{$D = 19$, $m = 1155$, $A = 1079 + 1078\alpha$}
	\end{subfigure}

	\caption{RCFP plots which have been rescaled to the unit disc, using the element $A$ and the modulus $(m)$, where $K = \Q(\sqrt{-D})$}
	\label{fig:RCFPExamplestoDisc}
\end{figure}

\subsection{Observations}

\label{sec:Observations}

There are many observations that can be made about these RCFP plots, though these patterns seem difficult to explain for reasons that will be discussed in the next section. 

We begin by exploring what happens in the analogous situation to the DGL Theorem \ref{thm:DGL} and its generalization Theorem \ref{thm:DGLgeneralize}. We choose a quadratic imaginary field $K = \Q(\sqrt{-D})$ and a modulus $m = p^e$, where $p$ is a rational prime and $e$ is some power. We then choose an element $A$ of the Galois group such that the multiplicative order $d$ of $A \mod p^e$ divides the size of the residue field of the prime $\p \subset \O_K$ lying over $p$. This means that $d \mid (p^2 - 1)$ when $p$ is inert or $d \mid (p - 1)$ when $p$ is split or ramified. We then want to determine the behavior of these RCFP plots as the size of $m$ increases. We provide examples of this situation in Figure \ref{fig:DGLRCFPAnalogue}. Note that for these examples, we use the field $K = \Q(\sqrt{-7})$ and we fix the order of $A$ to be 3. 

\begin{figure}[h!]
	\centering
	\begin{subfigure}[b]{0.32\linewidth}
		\includegraphics[width=\linewidth]{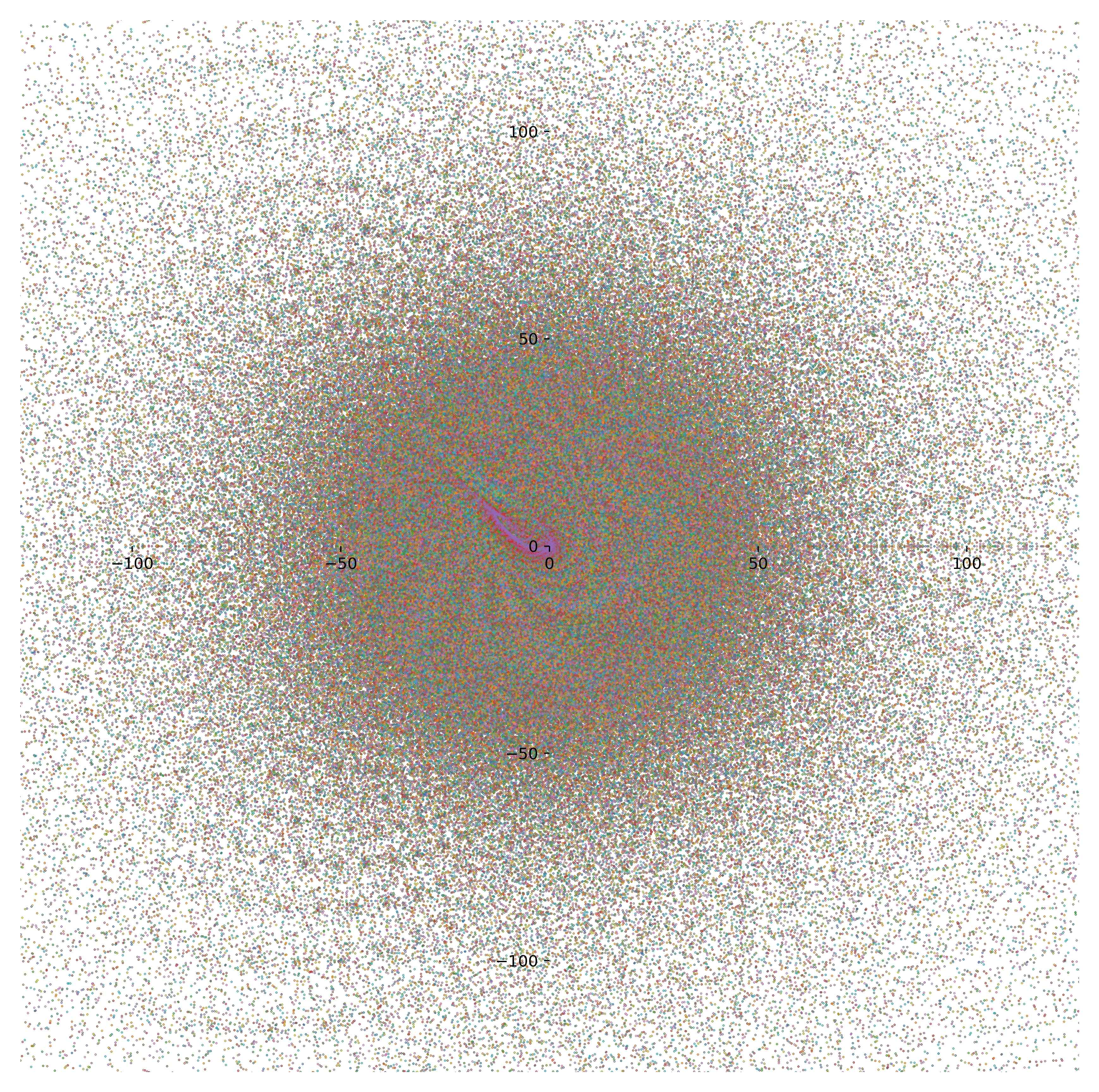}
		\caption{$m = 5^4$, $A = 211 + 423\alpha$}
	\end{subfigure}
	\begin{subfigure}[b]{0.32\linewidth}
		\includegraphics[width=\linewidth]{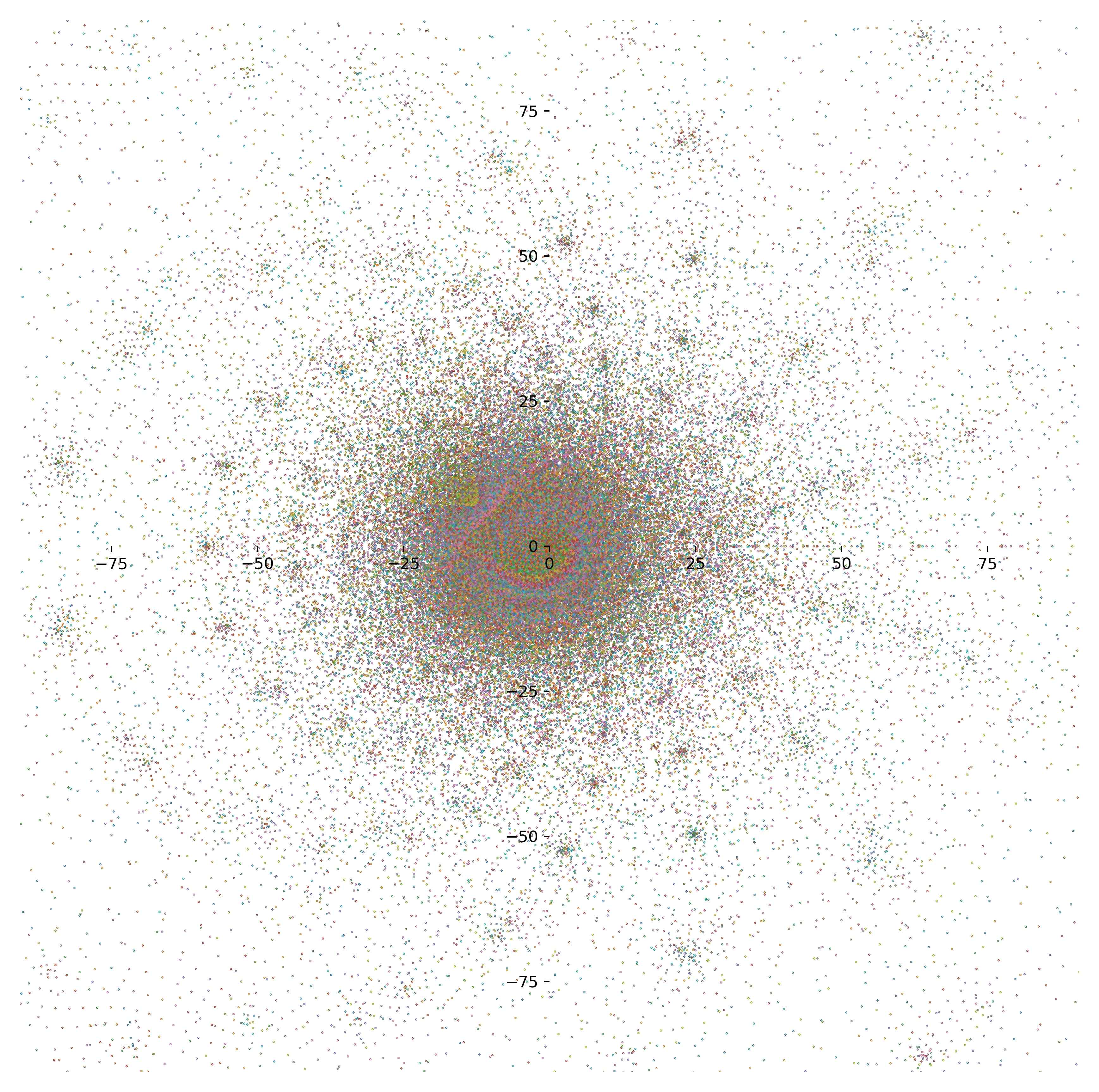}
		\caption{$m = 919$, $A = 52$}
	\end{subfigure}
	\begin{subfigure}[b]{0.32\linewidth}
		\includegraphics[width=\linewidth]{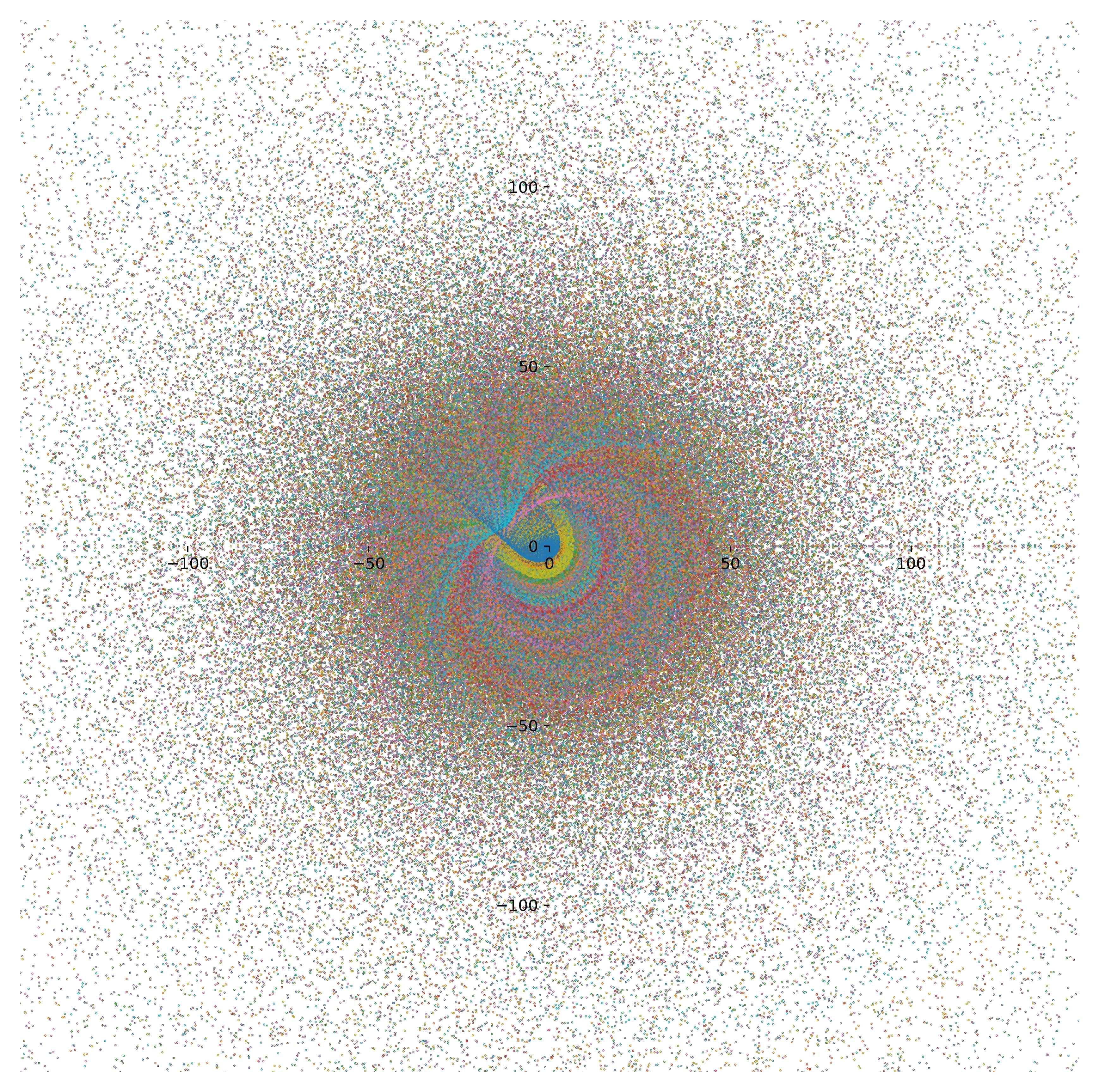}
		\caption{$m = 7^4$, $A = 1353$}
	\end{subfigure}
	
	\caption{RCFP plots for $D = 7$, modulus $(m)$, and $A$ of order $d = 3$}
	\label{fig:DGLRCFPAnalogue}
\end{figure}

One notes that there appear to be certain areas where points seem to accumulate more densely, and this pattern seems to be more pronounced in the case when $p$ is inert (when $p = 5$ for these examples) or split ($p = 919$) compared to the ramified case ($p = 7$). Additionally, these accumulation areas appear to show up at regular intervals, though the number of such areas doesn't seem to have any obvious correlation to the primes in the modulus. The fact that the accumulation points seem unrelated to the modulus is something we might expect if we view this as an analogue to the DGL Theorem (recall that the most important property in the DGL Theorem was the element's order, not the modulus itself).

For the next situation, we again choose a modulus $m = p^e$, but instead of choosing $A$ such that $d$ doesn't divide $p$, we instead choose $A$ so that $d = p^b$ for some $b < e$. Note that in the Gaussian period setting, this situation leads to a Gaussian period plot that is a circle of radius $p^b$, along with a point at the origin. The situation for RCFP plots is a bit more interesting. We provide examples of this in Figure \ref{fig:RCFPSpirals}, where we again use the field $K = \Q(\sqrt{-7})$. Also, note that we used the map mentioned at the end of Section 4.3, since we believe the patterns become more visible under this rescaling.

\begin{figure}[h!]
	\centering
	\begin{subfigure}[b]{0.4\linewidth}
		\includegraphics[width=\linewidth]{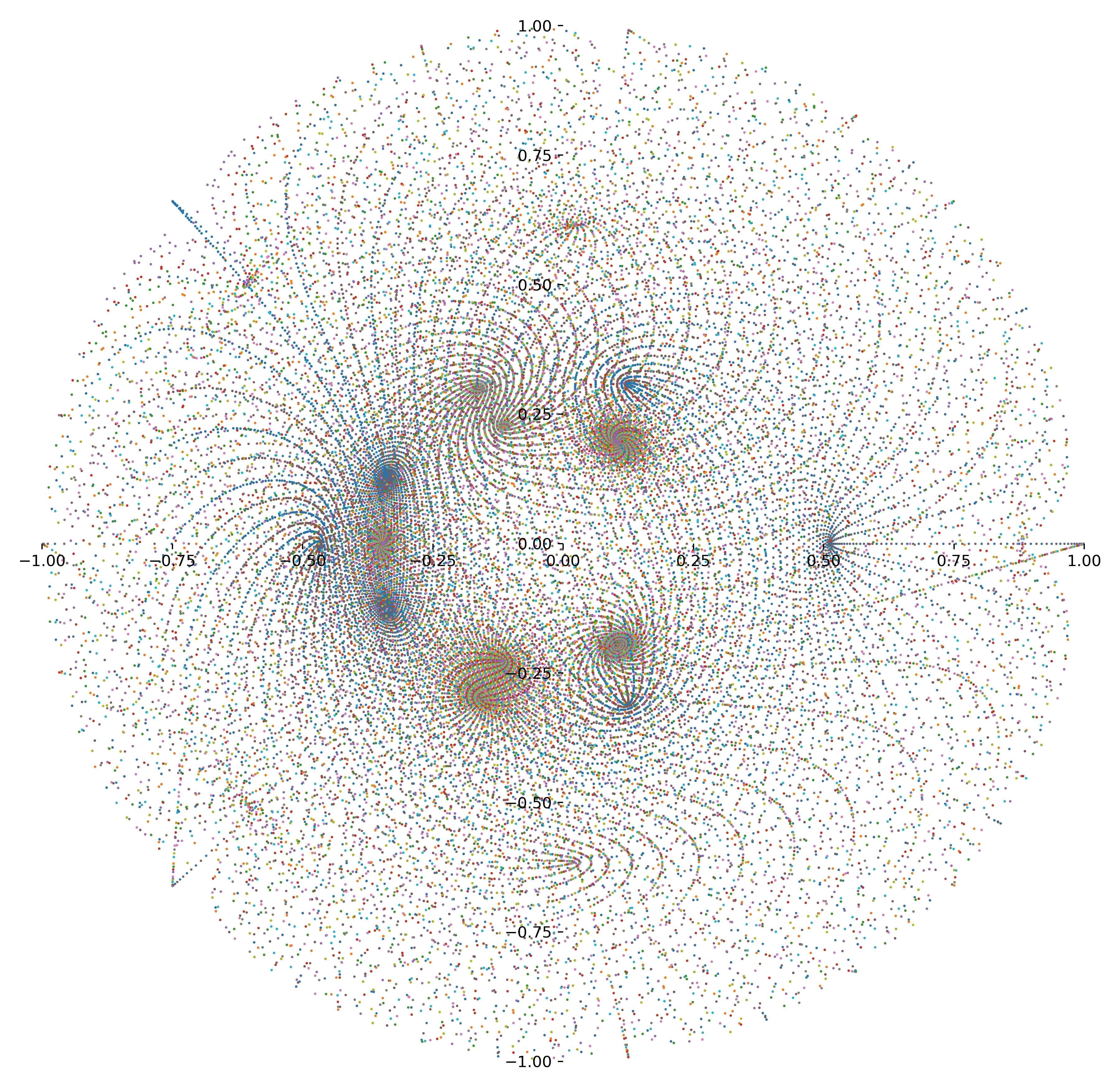}
		\caption{$m = 5^4$, $A = 126 + 125\alpha$}
		\label{fig:RCFPSpiralsA}
	\end{subfigure}
	\begin{subfigure}[b]{0.4\linewidth}
		\includegraphics[width=\linewidth]{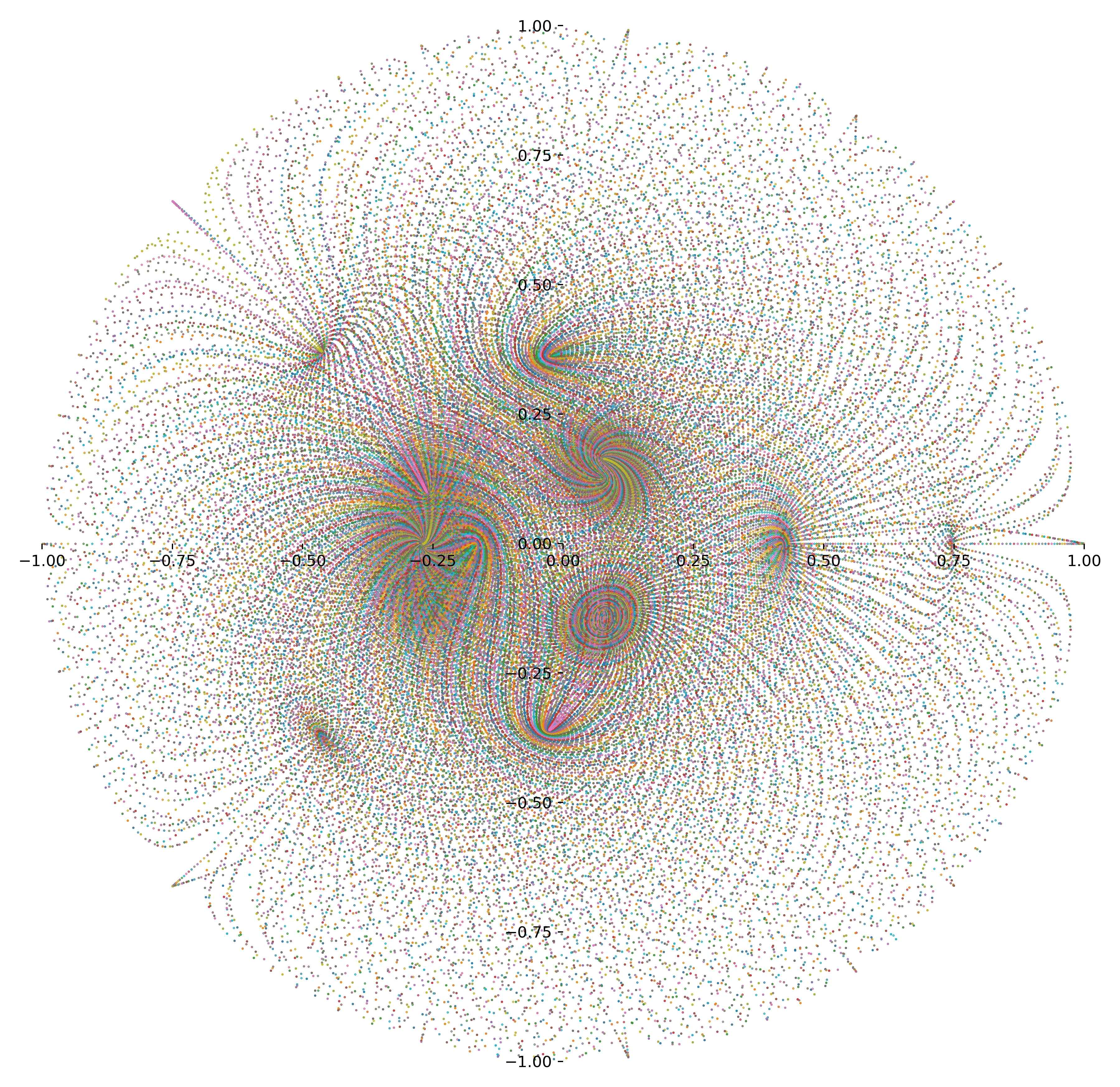}
		\caption{$m = 3^6$, $A = 1 + 243\alpha$}
		\label{fig:RCFPSpiralsB}
	\end{subfigure}

	\begin{subfigure}[b]{0.4\linewidth}
		\includegraphics[width=\linewidth]{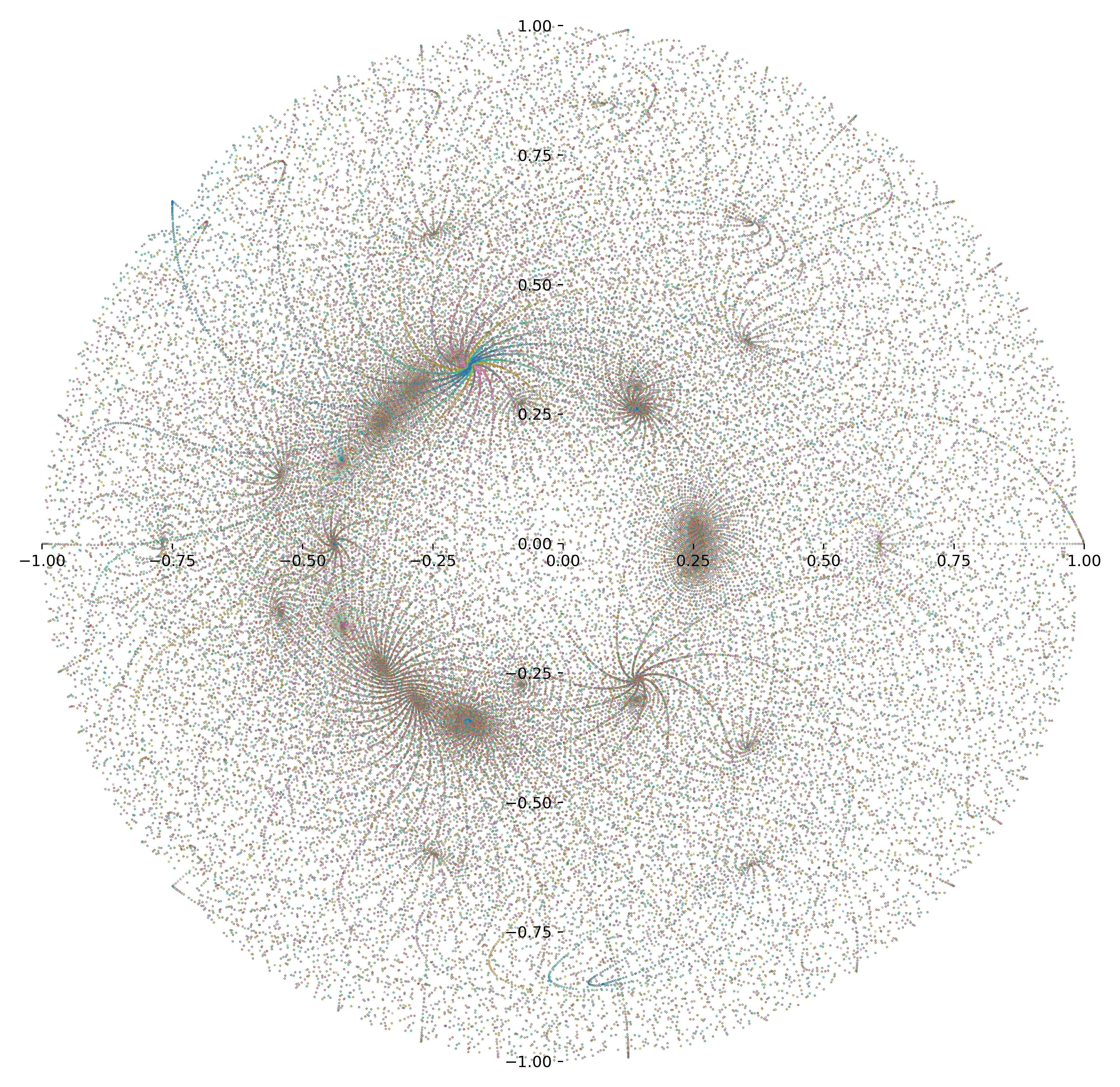}
		\caption{$m = 11^3$, $A = 1 + 121 \alpha$}
		\label{fig:RCFPSpiralsC}
	\end{subfigure}
	\begin{subfigure}[b]{0.4\linewidth}
		\includegraphics[width=\linewidth]{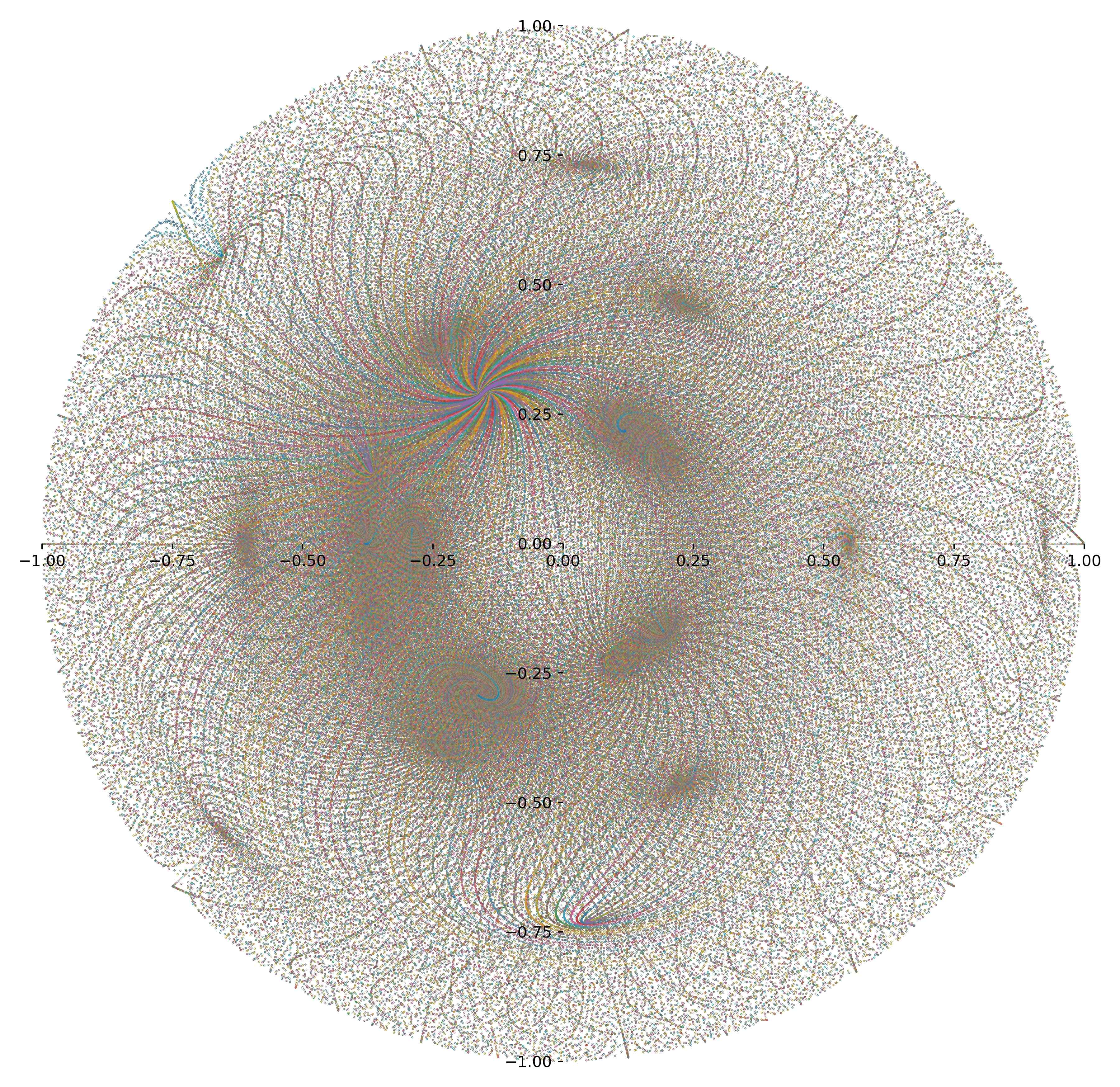}
		\caption{$m = 7^4$, $A = 1 + 343\alpha$}
		\label{fig:RCFPSpiralsD}
	\end{subfigure}
		
	\caption{RCFP plots for $D = 7$, modulus $(m)$ for $m = p^a$, and $A$ of order $d = p$}
	\label{fig:RCFPSpirals}
\end{figure}

The images in this situation are rather striking, as they tend to have easily identifiable lines. In fact, similar to the images in Figure \ref{fig:DGLRCFPAnalogue}, there again seem to be areas where points seem to accumulate more densely; however, in this situation, these accumulation areas often have a much more obvious spiral-like patterns appearing. We've taken to calling these areas ``galaxies.''

Another observation to make is that---unlike the analogous situation of the DGL Theorem in Figure \ref{fig:DGLRCFPAnalogue}---there does seem to be a correlation between the number of these galaxies and the primes in the modulus. For example, one notes that when $m = 5^4$ in Figure \ref{fig:RCFPSpiralsA}, one can make out five red, green, and purple galaxies close to the origin, as well as five primarily blue galaxies which are further out. This pattern seems to continue further away from the origin, though it's difficult to make out the finer details. Figure \ref{fig:RCFPSpiralsB} provides another example of this phenomenon when $m = 3^6$, where there are three primarily blue, red, and green galaxies close to the origin.

Another notable feature here is that if $A$ and $B$ are two elements of the Galois group with the same multiplicative order $p^b$, then their respective RCFP plots are not necessarily the same. This is different than the Gaussian periods case, where we have an identical Gaussian period plot for any $\omega$ whose order mod $p^e$ is $p^b$. This is because $(\Z/p^e\Z)^\times$ is cyclic, while $(\O_K/p^e\O_K)^\times$ is not cyclic when $p$ splits or when $e > 1$. In fact, we have the following isomorphism of abelian groups: $$(\O_K/p^e \O_K)^\times \cong \begin{cases} \F_{p^2}^\times \times (\Z/p^{e - 1}\Z \times \Z/p^{e - 1}\Z) & \text{$p$ is inert,} \\ (\Z/p^e\Z)^\times \times (\Z/p^e \Z)^\times & \text{$p$ splits,} \\ \F_p^\times \times (\Z/p^e\Z \times \Z/p^{e - 1}\Z) & \text{$p$ is ramified.} \end{cases}$$ In Figure \ref{fig:RCFPSpirals5}, we provide examples of three distinct RCFP plots, even though $K = \Q(\sqrt{-7})$, $m = 5^4$, and $A$ has order $5$ for all three cases.

\begin{figure}[h!]
	\centering
	\begin{subfigure}[b]{0.32\linewidth}
		\includegraphics[width=\linewidth]{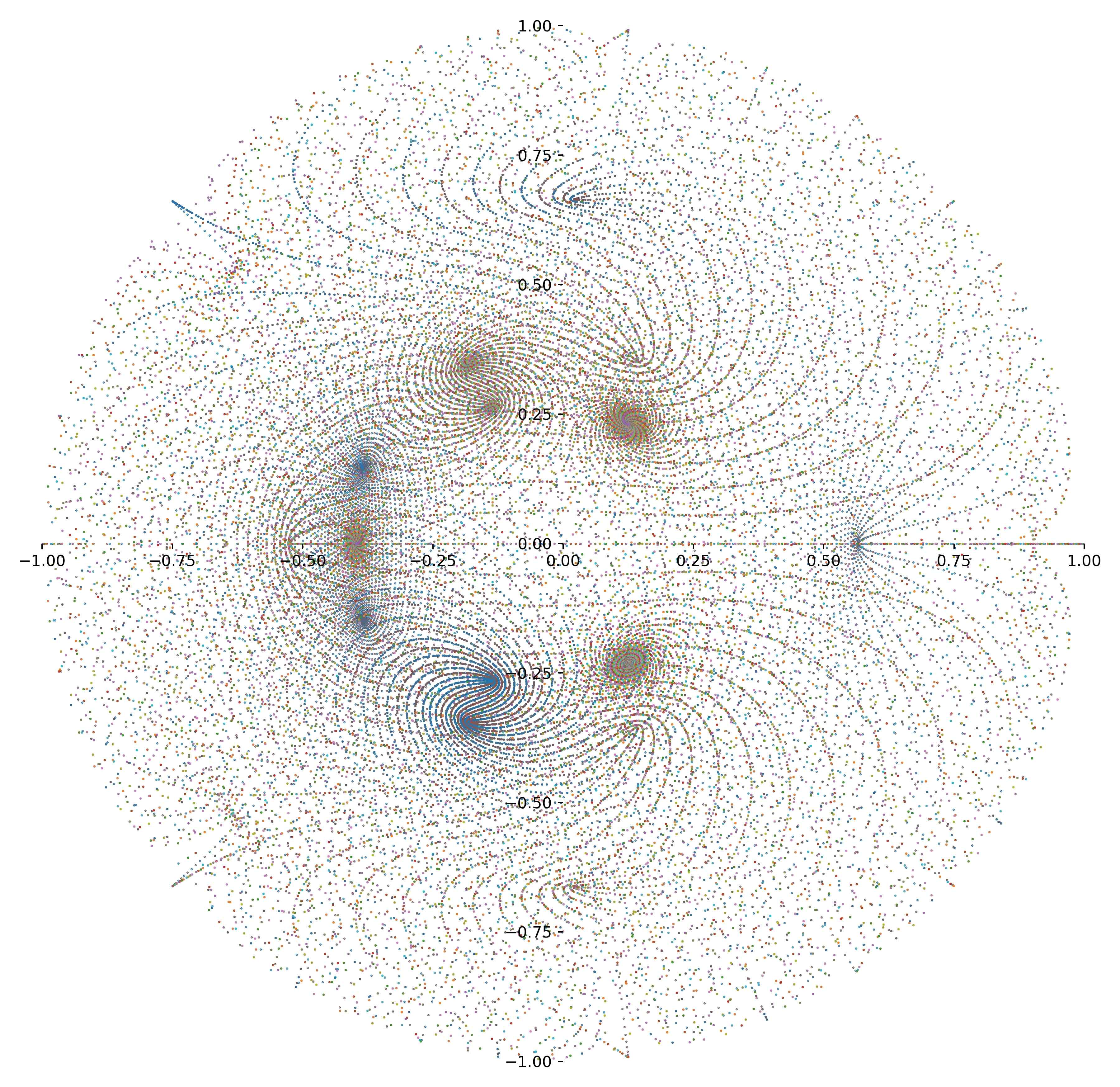}
		\caption{$m = 5^4$, $A = 376 + 250\alpha$}
		\label{fig:RCFPSpirals5A}
	\end{subfigure}
	\begin{subfigure}[b]{0.32\linewidth}
		\includegraphics[width=\linewidth]{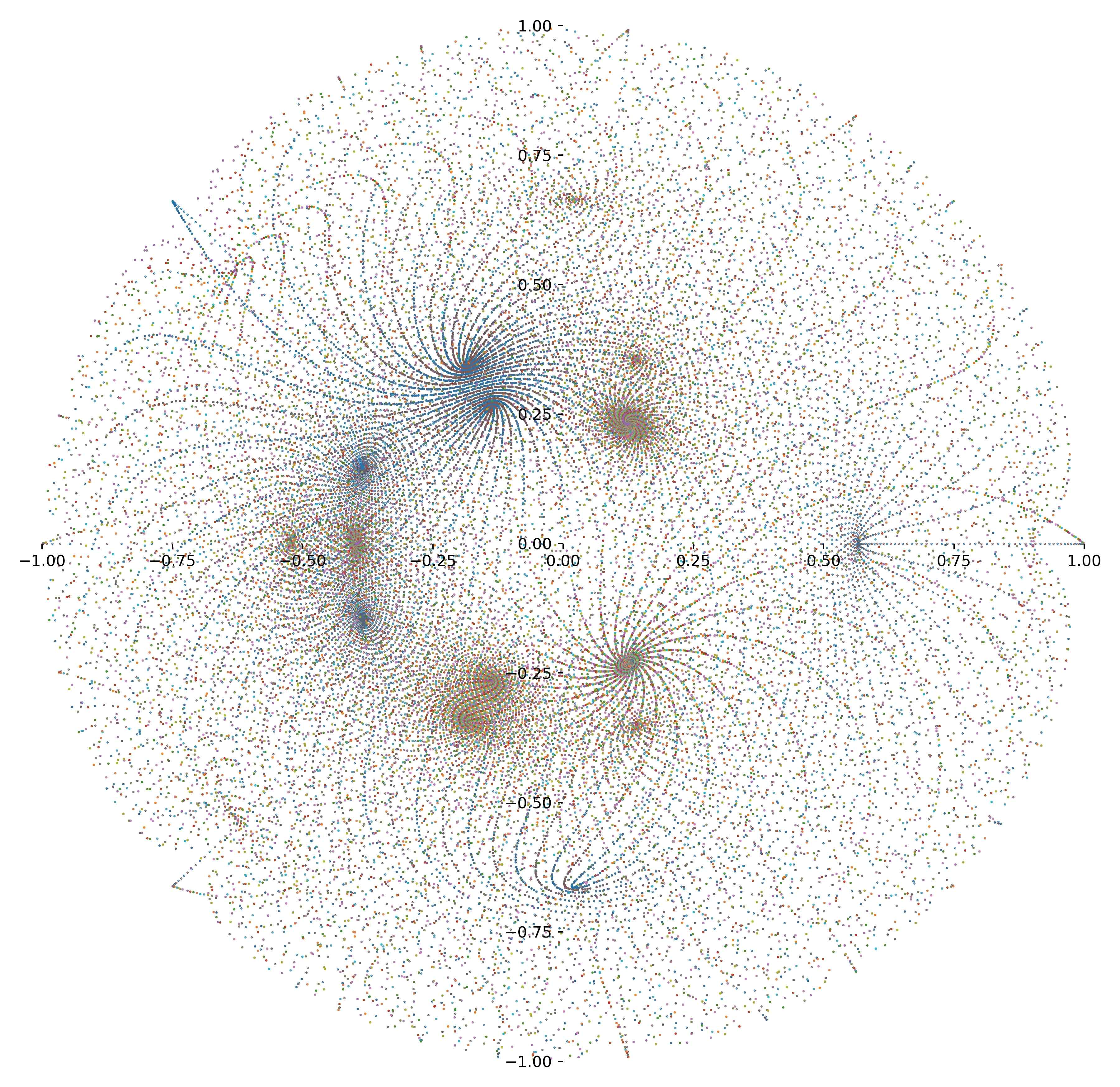}
		\caption{$m = 5^4$, $A = 1 + 125\alpha$}
		\label{fig:RCFPSpirals5B}
	\end{subfigure}
	\begin{subfigure}[b]{0.32\linewidth}
		\includegraphics[width=\linewidth]{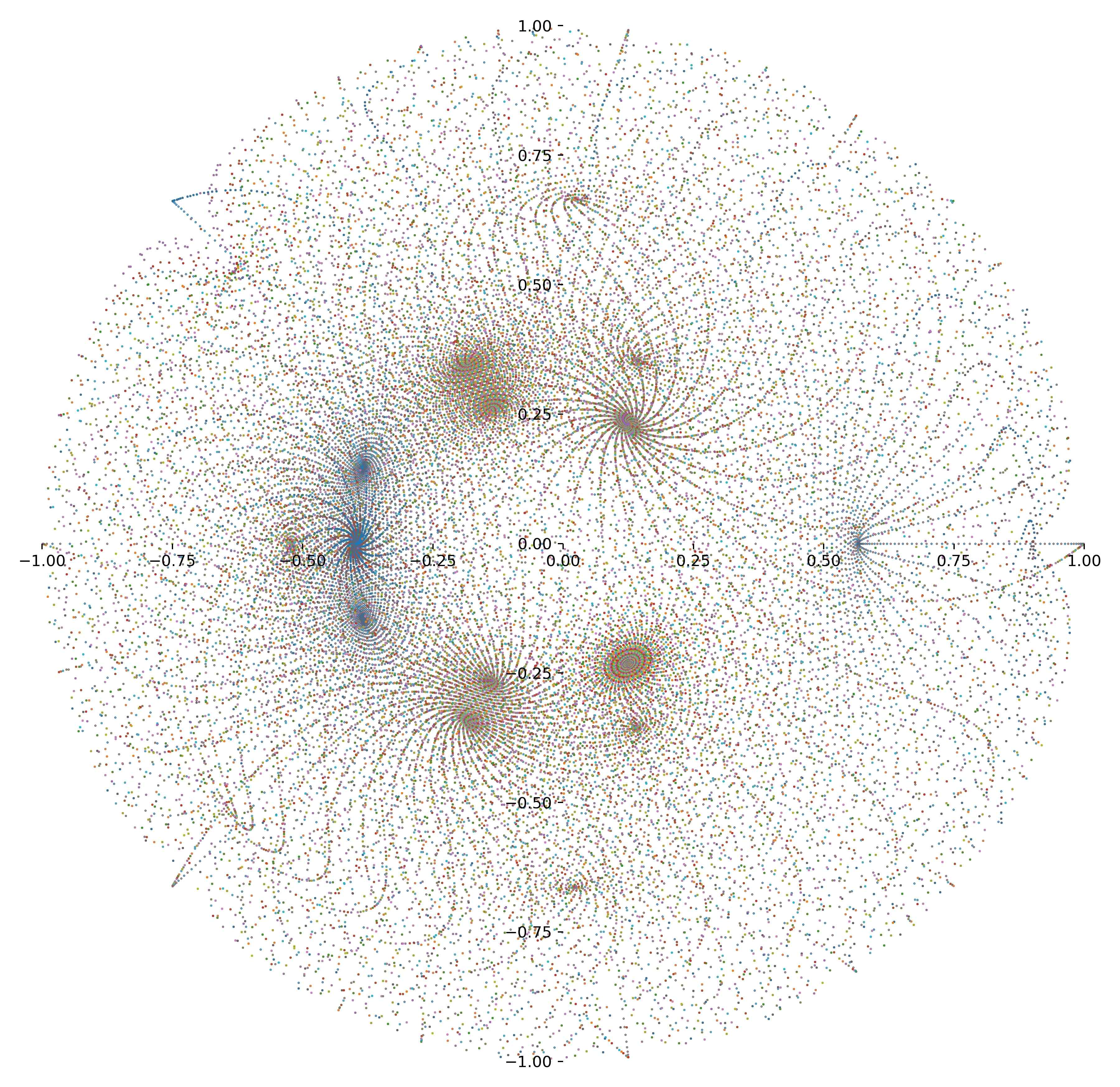}
		\caption{$m = 5^4$, $A = 501 + 250\alpha$}
		\label{fig:RCFPSpirals5C}
	\end{subfigure}
	
	\caption{RCFP plots for $D = 7$, modulus $(m)$ for $m = 5^4$, and $A$ of order $d = 5$}
	\label{fig:RCFPSpirals5}
\end{figure}

Note in particular that the differences between these plots are more substantial than a simple change of coloring. For example, note the behavior of the points along the positive real axis; for Figure \ref{fig:RCFPSpirals5A}, the string of points stays entirely on the real axis, while in Figures \ref{fig:RCFPSpirals5B} and \ref{fig:RCFPSpirals5C}, the string of points veers off toward either the upper or lower half-plane respectively.

\subsection{Obstacles} 

\label{sec:Obstacles}

So far in this section, we have avoided a discussion of RCFP plots in the same rigorous mathematical way that we used with Gaussian periods. The reason for this is due almost entirely to the relative complexity the Galois action on ray class fields of quadratic imaginary fields compared to computing the Galois action on cyclotomic fields, which we will now explain.

For the cyclotomic case, the action of the Galois group can be described concretely as a function on roots of unity. In particular, an element $\sigma_\omega \in \Gal(\Q(\zeta_n)/\Q)$ corresponding to $\omega \in (\Z/n\Z)^\times$ can be described as a function on roots of unity that maps $\zeta_n^a \mapsto \zeta_n^{\omega a}$. This function is easily described, and its arithmetic dynamics is simple enough to study. 

However, for the quadratic imaginary case, consider an element $\sigma_A \in \Gal(K[\m]/K[1])$ corresponding to an element $A \in (\O_K/\m \O_K)^\times/\O_K^\times$. If $x$ is the $x$-coordinate of an $\m$-torsion point of an elliptic curve with CM by $K$ (assuming for this discussion that $K$ is not $\Q(i)$ or $\Q(\zeta_3)$), then $x = \wp(z_x;\O_K)$ for some $z_x \in \C/\O_K$, where $\wp$ is the Weierstrass $\wp$-function. We've described the action of $A$ on $x$ to be the multiplication $A z_x$ in $\C/\O_K$, which is then mapped back to $\C$ via $\wp$. That is, the action of $A$ on $x$ is given by $A \cdot x = \wp(Az_x; \O_K)$. The question then arises whether or not we can describe this action in a ``nice'' way---i.e. where the arithmetic dynamics can be studied in a way similar to the Gaussian period case. Unfortunately, this appears to be much more complicated in the RCFP case, perhaps even prohibitively so. 

For $A \in (\O_K/\m \O_K)^\times/\O_K^\times$, we define the function $$f_A: \C \to \C, \hspace{0.5in} x \mapsto \wp(A z_x; \O_K).$$ Then $f_A$ describes the Galois action of $A$. Also, note that $f_A$ is an even elliptic function relative the lattice $\O_K$. Thus, by the well-known result that even elliptic functions can be written as rational functions in $\wp$ (see, for example, \cite[Theorem VI.3.2]{SilvermanArithmetic}), the function $f_A(x)$ is actually rational in $x = \wp(z_x;\O_K)$. 

Recall that RCFPs were defined as $$\eta_{K, \m, A}: \O_K/\m\O_K \to \C, \hspace{0.5in} \eta_{K, \m, A}(z_x) = \sum_{j = 0}^{r - 1} \wp(A^j z_x; \O_K),$$ and note that $\wp(A^j z_x; \O_K) = (f_A \circ \cdots \circ f_A)(x)$ (the composition of $f_A$ with itself $j$ times). Let us write $f_A^{(j)}(x)$ to represent this $j$-fold self-composition. Then we can rewrite RCFPs in the following way: $$\eta_{K, \m, A}: h(E[\m]) \to \C, \hspace{.5in} \eta_{K, \m, A}(x) = \sum_{j = 0}^{r - 1} f_A^{(j)}(x).$$ Thus the study of RCFP plots is the study of the arithmetic dynamics of the rational functions $f_A$. 

Now, studying the arithmetic dynamics of $f_A$ requires computing $f_A(x)$ explicitly. If $A$ is an integer, then $f_A(x)$ can be built from the standard division polynomials for the elliptic curve isomorphic to $\C/\O_K$ (i.e. polynomials whose roots are the $A$-torsion points). These are defined recursively and depend on the choice of base field, but they are relatively well-known (see \cite[Exercise III.3.7]{SilvermanArithmetic}). For general $A \in (\O_K/\m\O_K)^\times$, however, computing $f_A$ becomes more difficult, since computing the generalized division polynomials is more complex. That said, algorithms do exist for computing these polynomials. To the author's knowledge, Satoh in 2004 in \cite{DivisionPoly} was the first to describe such an algorithm, followed by an optimization of {K\"{u}\c{c}\"{u}ksakall\i} in 2015 in \cite{ComputeDivisionPoly}. 

Regardless, these division polynomials remain difficult to describe in a general way, and the arithmetic dynamics of $f_A$ even more so. It is the author's hope that this paper might spark some interest in studying the rational functions $f_A$ more explicitly, as we believe that studying these functions will prove to be worthwhile mathematically. 

\section{Questions and Strategies}

As we have seen, there are many striking and beautiful visual properties of Gaussian periods and their analogues. While this and previous papers have explored and explained some of these properties, there are many unanswered questions left. With this in mind, we conclude this article with some open questions that arose while studying these objects, along with some strategies and techniques to aid further study. While some of these questions and strategies have already been mentioned in previous sections, we thought it convenient to consolidate them here.

\subsection{Questions}

\begin{enumerate}
	\item What happens to the values of the supercharacter theories and RCFPs when using the action of a non-cyclic subgroup of the automorphism group? 
	
	\item Can anything be said about Gaussian periods in general when restricting $\eta_{n, \omega}$ to $k \in (\Z/n\Z)^\times$ (that is, only the values relatively prime to $n$)? As we mention in question (3), we note this question has already been partially answered by Untrau. However, Untrau's result works only for cases where $n$ is a power of an odd prime and does not address the case for general $n$.
	
	\item More generally, can anything interesting be said when restricting to various subsets of $\Z/n\Z$? Untrau shows in Theorem A of \cite{untrau2021equidistribution} that the equidistribution result seen in Theorems \ref{thm:DGL} and \ref{thm:DGLgeneralize} holds even when restricting to elements of $(\Z/p^e\Z)^\times$ (among other situations). Are there other conditions on subsets of $\Z/n\Z$ which retain this result?
	
	\item Similarly, can anything interesting be said about RCFPs when restricting $\eta_{K, \m, A}$ to $z \in (\O_K/\m \O_K)^\times$? What about for other subsets of $\O_K/\m \O_K$?
	
	\item Can the behavior of supercharacter theory values be succinctly described for composite moduli $n = p_1^{e_1} \cdots p_\ell^{e_\ell}$ when the order of the matrix $A \mod p_i^{e_i}$ is different for different choices of $i$? Compare this to Theorem \ref{thm:DGLgeneralize}, where the fact that $\Phi_d(A) = 0 \mod n$ forces $A \mod p_i^{e_i}$ to have multiplicative order $d$ for all $i$.
	
	\item Let $g_d$ be the Laurent polynomial described in Theorem \ref{thm:DGL}. Can the geometry of $\img(g_d)$ be described succinctly when $d$ is not a prime power? On this note, can something be said about $\img(g_d)$ as the set of traces of some subgroup of $\SU(d)$ when $d$ is not a prime power?
	
	\item Is it possible to study explicitly the rational functions $f_A$ described in Section \ref{sec:Obstacles} in a general way? If so, what can be said about the arithmetic dynamics of $f_A$?
	
	\item Why do the patterns mentioned in Section \ref{sec:Observations} show up in RCFP plots? This is related to the previous question, though not necessarily reliant upon it.
\end{enumerate}

\subsection{Strategies}

\label{sec:strategies}

There are a few strategies and heuristics that we have found helpful in our studies, and we offer a brief overview of what we think to be the important ones.

From a theoretical perspective, we have found it the most fruitful to start by finding the prime factorization of the modulus $n$ or $\m$. We then use the Chinese Remainder Theorem to study the behavior of $\omega$ or $A$ modulo prime powers, followed by trying to connect this more local behavior back to the overall behavior mod $n$ or $\m$. In fact, Theorem \ref{thm:DGLgeneralize} originated partly with the observation that the DGL Theorem was true for more general $n = p_1^{e_1} \cdots p_\ell^{e_\ell}$, as long as $\omega \mod p_i^{e_i}$ had multiplicative order $d$ for every $i$. The reader might view \cite[\S 3]{CompositeModuli} to see how this sort of strategy has been utilized elsewhere.
	
From a computational perspective, computing power and time are the main issues that one is constantly brushing up against. An astute reader might have already noticed that our choices for moduli in our examples are often fairly small, especially outside of a few examples of Gaussian period plots where some moduli were seven or eight digits long. This is, of course, almost entirely due to the quickly increasing number of computations needed as the modulus grows larger. For example, generating cyclic supercharacter plots for $G = (\Z/n\Z)^m$ is dominated by the $n^m$ computations of the supercharacter values. Even for modest choices of $n$ and even if $m = 2$, the amount of time needed to generate these plots can get out of hand quite quickly. 

Because of this, we are often quite limited in the choices of modulus with which we are allowed to experiment. Thus, it is often good practice to construct $n$ oneself by choosing the primes in its factorization. This allows one to guarantee the existence of elements in the automorphism group which have the desired properties. 

In particular, one property that one might wish to control is an element's multiplicative order, which is often an element's most important attribute when studying Gaussian periods and their analogues. Thus, in order to facilitate the work of an interested reader, we offer a few basic observations that might be useful if one hasn't worked much with these groups before.

\begin{itemize}
	\item All the automorphism groups arising in the study of Gaussian periods and the cyclic supercharacters described in Section \ref{sec:DGLgeneralization} are isomorphic to $\GL_m(\Z/n\Z)$ for some $n \geq 2$ and $m \geq 1$. If $n = p_1^{e_1} \cdots p_\ell^{e_\ell}$ is the factorization into prime powers, then one can use the Chinese Remainder Theorem to decompose this group as $$\GL_m(\Z/n\Z) \cong \GL_m(\Z/p_1^{e_1}\Z) \times \cdots \times \GL_m(\Z/p_\ell^{e_\ell}\Z).$$ Additionally, if $p^e$ is some prime power, then we have that $$\# \GL_m(\Z/p^e\Z) = p^{(e - 1)m^2}(p^m - 1)(p^m - p) \cdots (p^m - p^{m - 1}).$$ Thus is $d$ is the order of some element $A \in \GL_m(\Z/n\Z)$, then $d$ must divide $$\prod_{i = 1}^\ell p_i^{(e_i - 1)m^2}(p_i^m - 1) \cdots (p_i^m - p_i^{m - 1}).$$
	
	\item For RCFPs, we are instead looking at elements $A \in (\O_K/p^e\O_K)^\times/\O_K^\times$. If $d$ is the multiplicative order of $A$, then the possibilities for $d$ are determined by the isomorphism presented toward the end of Section \ref{sec:Observations}. Also, if one wishes to determine whether a prime $p \in \Z$ ramifies, splits, or is inert in $\O_K = \Z[\alpha]$, then one need only determine if the minimal polynomial for $\alpha$ has double roots, distinct roots, or is irreducible mod $p$ (respectively).
	
	\item Finding an element $A \in \GL_m(\Z/n\Z)$ such that $\Phi_d(A) = 0 \mod n$ is often laborious, and we haven't found a method much better than simply having Sage do the following: 
	\begin{enumerate}[(i)]
		\item Take a matrix $A \in \Mat_m(\Z/n\Z)$.
		\item Verify that $\det(A)$ is relatively prime to $n$. If not, go back to (i).
		\item Compute the multiplicative order of $A \mod n$.
		\item If the order from (iii) is not $d$, go back to (i).
		\item If the order from (iii) is $d$, compute $\Phi_d(A) \mod n$. If this is $0 \in \Mat_m(\Z/n\Z)$, then return $A$. Otherwise, go back to (i).
	\end{enumerate} 
	However, we have noticed that such an element seems to exist in $\GL_m(\Z/p^e\Z)$ only when $d$ divides $\# \GL_m(\Z/p\Z)$. Thus, such an element should exist in $\GL_m(\Z/n\Z)$ only when $d \mid (\# \GL_m(\Z/p\Z))$ for every $p \mid n$. The converse is not necessarily true, since (for example) one can verify computationally that there is no $A \in \GL_2(\Z/25\Z)$ which satisfies $\Phi_5$.
	
	\item Otherwise, when trying to find an element $A \in \GL_m(\Z/n\Z)$ of order $d$ that doesn't need to satisfy $\Phi_d$, we often found it to be faster simply to find the order of some random matrix $B$ (say the order is $c$), and if $d$ divides $c$, then we set $A = B^{c/d}$.
	
	\item When trying to find elements $A \in \GL_m(\Z/p^e\Z)$ of order $p^a$ (that don't need to satisfy $\Phi_{p^a}$), one can verify (using the Binomial Theorem) that if $B \in \Mat_m(\Z/p^a\Z)$ and the matrix $A = I + p^{e - a} B$ is invertible, then $A$ will have order dividing $p^a$. When $m = 1$, we can be more explicit: an element $\omega \in (\Z/p^e\Z)^\times$ has order $p^a$ if and only if $\omega = 1 + p^{e - a} \beta$ for some $\beta \in (\Z/p^a\Z)^\times$.
	
	\item The above statement similarly works for the automorphism group used in RCFPs. That is, if $\beta \in (\O_K/p^a\O_K)^\times$ and $A = 1 + p^{e - a}\beta$ is an element of $(\O_K/p^e\O_K)^\times$, then $A$ will have order $p^a$.
\end{itemize}

\section{Code} 

\label{sec:Code}

Most of the code used to generate these images was written in Python, though we used Sage for convenience reasons to generate the animations (explained in Section \ref{sec:tracemaps}) and the Laurent polynomials $g_d$ (explained in Theorem \ref{thm:DGL} and Section \ref{sec:tracemaps}). Many algorithmic aspects of computing elliptic curve torsion points were based on the algorithms in \cite{CohenClassField, CohenNumThry}, and these are often cited in the comments of the code itself. Readers may access our code at the following GitHub link: \textcolor{blue}{\href{https://github.com/SamanthaPlatt/GaussianPeriodsandAnaloguesCode}{https://github.com/SamanthaPlatt/GaussianPeriodsandAnaloguesCode}}

\newpage

\bibliography{ExperBib.bib}{}
\bibliographystyle{plain}

\end{document}